\documentclass[10pt,leqno,twoside]{amsart}

\usepackage{enumerate}

\usepackage{amssymb}

\usepackage{amssymb,amsthm,amsmath,eucal,mathrsfs}
\usepackage{bm}

\usepackage{tikz, subfigure, xcolor} 

\usepackage{pgfplots}

\usepackage{cite}

\usepackage{amsmath,verbatim}

\usepackage{amsthm}

\usepackage{amssymb}

\usepackage{amsfonts}

\usepackage{dsfont}

\usepackage{hyperref}

\newtheorem{theorem}{Theorem}[section]

\newtheorem{lemma}[theorem]{Lemma}

\newtheorem{proposition}[theorem]{Proposition}

\theoremstyle{definition}

\newtheorem{remark}[theorem]{Remark}

\numberwithin{equation}{section}

\renewcommand{\Im}{{\ensuremath{\mathrm{Im\,}}}} %Imagin????rteil nicht alsfraktur

\providecommand{\norm}[1]{\left\Vert#1\right\Vert} %Norm

\newcommand{\1}{\mathbb{I}}

\newcommand\restr[2]{{% we make the whole thing an ordinary symbol
  \left.\kern-\nulldelimiterspace % automatically resize the bar with \right
  #1 % the function
  \vphantom{\big|} % pretend it's a little taller at normal size
  \right|_{#2} % this is the delimiter
  }}

\usepackage{eucal}

\title[Equivalent media]{The equivalent media generated by bubbles of high contrasts:\\
Volumetric metamaterials and metasurfaces}
\author[Ammari, Challa, Choudhury, Sini]{Habib Ammari $^* $, Durga Prasad Challa $^{**} $, Anupam Pal Choudhury $^{\dag} $, Mourad Sini$^{\ddag} $}

\keywords{bubbly media, Foldy-Lax approximation, effective medium theory, metamaterials, metascreen.\\
$^* $ Department of Mathematics, ETH Z\"urich, R\"amistrasse 101, CH-8092 Z\"urich, Switzerland. E-mail: habib.ammari@math.ethz.ch\\
$^{**} $ Faculty of Mathematics, Indian Institute of Technology Tirupati, Tirupati, India. Email: chsmdp@iittp.ac.in. This author was partially supported by the Austrian Science Fund (FWF): P28971-N32 and DST SERB MATRICS (Mathematical Research Impact Centric Support) MTR/2017/000539.\\
$^{\dag}$ Indian Institute of Technology (IIT) Indore, Indore 453552, India. Email: anupampcmath@gmail.com. A part of this work was done when this author was a post-doctoral research scholar at RICAM and was supported by the Austrian Science Fund (FWF): P28971-N32.  \\ 
$^{\ddag}$ RICAM, Austrian Academy of Sciences,
Altenbergerstrasse 69, A-4040, Linz, Austria.
Email: mourad.sini@oeaw.ac.at. This author is partially supported by the Austrian Science Fund (FWF): P28971-N32}

\allowdisplaybreaks
\begin{document}
\begin{abstract}  
We deal with the point-interaction approximations for the acoustic wave fields generated by a cluster of highly contrasted bubbles for a wide range of densities and bulk moduli contrasts.
 We derive the equivalent fields when the cluster of bubbles is appropriately distributed (but not necessarily periodically) in a bounded domain $\Omega$ of $\mathbb{R}^3$. 
 We handle two situations. 
 \begin{enumerate}
 \item In the first one, we distribute the bubbles to occupy a $3$ dimensional domain. For this case, we show that the equivalent speed of propagation changes sign when the medium is excited with
 frequencies smaller or larger than (but not necessarily close to) the Minnaert resonance. As a consequence, this medium behaves as a reflective or absorbing depending on whether the used frequency is smaller or larger than this resonance. 
 In addition, if the used frequency is extremely close to this resonance, for a cluster of bubbles with density above a certain threshold, then the medium behaves as a \textquoteleft wall', i.e. allowing no incident sound to penetrate.
 
 \item In the second one, we distribute the bubbles to occupy a $2$ dimensional (open or closed) surface, not necessarily flat. For this case, we show that the equivalent medium is modeled by a Dirac potential supported on
 that surface. The sign of the surface potential changes for frequencies smaller or larger than the Minnaert resonance, i.e. it behaves as a smart metasurface reducing or amplifying the transmitted sound across it. As in the $3$D case, if the used frequency is extremely close to 
 this resonance, for a cluster of bubbles with density above an appropriate threshold, then the surface  allows no incident sound to be transmitted across the surface, i.e. it behaves as a white screen. 
 \end{enumerate}   
 
\end{abstract}
\maketitle
\section{Introduction}
Recently, there has been a great interest in developing materials to control waves (as acoustic, electromagnetic or elastic waves) in unprecedented ways. For this purpose, the engineers provided us with clever designs 
to create such \textquoteleft artificial' materials with unusual responses. One of such designs is based on inclusions arranged in specific ways so that when the mentioned waves  interact with them new properties emerge. 
The used inclusions are made of usual materials but they are of smaller scales, usually at the micro or nano scales, enjoying high contrasts as compared to the background medium where they are embedded. These two features 
in choosing and using the inclusions, namely the way of arranging them and their proper scales, are crucial. 
Precisely, the choice of the proper scaled and contrasted inclusions allow us to create (subwavelength) resonances which are extremely close to the real line. The arrangement of such inclusions provide effective macroscopic
media with changing behaviors while excited with frequencies close to the mentioned resonances. In addition, we can distribute the inclusions to fill-in volumetric ($3$ D) domains, $2$ D surfaces or $1$ D curves. 
These combinations provide us with, respectively, volumetric metamaterials, metasurfaces and metawires.

In this paper, we deal with acoustic waves generated by micro-scaled bubbles having highly contrasted bulk moduli (and densities). With those scales, Minnaert resonances occur, see \cite{Habib-Minnaert} and \cite{ACCS}. 
Distributing such bubbles in ($3$ D) domains and in $2$ D surfaces, we confirm, in particular, the possibility to generate volumetric metamaterials and (non-necessarily flat) open or closed metasurfaces with interesting
properties. Formal derivation of the effective medium in $3$ D domains are derived in \cite{Papanicoulaou-2} and a justification is provided in \cite{Habib-bubbles} for frequencies near the Minnaert resonance. 
In \cite{H-F-G-L-Z}, the effective medium corresponding to a periodic distribution of the bubbles in a flat and infinite $2$ D surface (a plane) is studied for frequencies near the Minnaert resonance. Compared 
to these results, here we deal with general shaped (open or closed) surfaces (where no periodicity is needed) and for any fixed frequency. In addition, for both $3$ D and $2$ D domains, we consider all the three 
regimes on the denseness of bubbles. In the low regime, the cluster of bubbles has no effect, i.e. there is no reflection. In the medium regime, the equivalent medium allows both reflection and transmission of the wave. 
In this regime, we can control the amount of reflection and transmission and hence increase or reduce it at our will. In the third regime (i.e. high regime), the equivalent medium allows no transmission, i.e. 
it behaves as a wall for the $3$ D case or a white screen for the $2$ D case. More detailed properties of such designs are provided later after stating the results in section \ref{Main-Results}.  

\bigskip

The rest of the paper is organized as follows. In section \ref{Background}, we describe the mathematical background and fix some notations and in section \ref{Main-Results}, we state and discuss the main results with 
appreciable details. In section \ref{Formal-Arguments}, we provide the formal arguments with the key ideas behind the proofs and then in section \ref{Section-volume} and section \ref{Section-surface}, we give the full details
 to justify the main results which are Theorem \ref{Main-theorem-volumes} and Theorem \ref{Main-theorem-surfaces} respectively. 
Finally, in an appendix, we gather few technical results used in the proof of the main results.  

\section{ Background and notations }\label{Background}
In this section, we discuss about the mathematical model and recall some results related to point-interaction approximations for the acoustic wave fields, from \cite{ACCS}, that we shall use in this work.
\bigskip

Let us denote by  $\{D_{s}\}_{s=1}^{M}$ a finite collection of small bubbles in $\mathbb{R}^{3}$ of the form $D_{s}:= \delta B_m+z_{s}$, where
 $B_m$ is an open, bounded (with Lipschitz boundary), simply connected set in $\mathbb{R}^{3}$ containing the origin, and $z_{s}$ specify the locations of the 
bubble. The parameter $\delta > 0 $ characterizes the smallness assumption on the bubbles.  
Let us consider piecewise constant densities of the form 
\begin{equation}
\rho_{\delta}(x)=\begin{cases}
                 \rho_{0}, \ x \in \mathbb{R}^{3}\diagdown \overline{\cup_{l=1}^{M} D_{l}},\\
                 \rho_{s}, \ x \in D_{s}, \ s=1,...,M, \end{cases}
\label{model1}
\end{equation}
%\end{comment}
%
%
and piecewise constant bulk modulus in the analogous form
\begin{equation}
k_{\delta}(x)=\begin{cases}
                k_{0}, \ x \in \mathbb{R}^{3}\diagdown \overline{\cup_{l=1}^{M} D_{l}},\\
                k_{s}, \ x \in D_{s}, \ s=1,...,M,
               \end{cases}
\label{model2}
\end{equation}
where $\rho_{0},\rho_{s}, k_{0},k_{s}$ are positive constants. Thus $\rho_{0}$ and $k_{0}$ denote the density and bulk modulus of the 
background medium and $\rho_{s}$ and $k_{s}$ denote the density and bulk modulus of the bubbles respectively.\\
The mathematical model for describing the acoustic scattering  by the collection of small bubbles $D_s, s=1, ..., M$ is as follows:
\begin{equation}
 \begin{cases}
  \Delta u + \kappa_{0}^{2} u =0 \ \text{in} \ \mathbb{R}^{3}\diagdown \overline{\cup_{l=1}^{M} D_{l}}, \\
   \Delta u + \kappa_{s}^{2} u =0 \ \text{in}\ D_{s},\ s=1,...,M,\\
   \left.u \right\vert_{-}-\left.u \right\vert_{+}=0, \ \text{on} \ \partial D_{s},\ s=1,...,M,\\
   \frac{1}{\rho_{s}}\left.\frac{\partial u}{\partial \nu^{s}} \right\vert_{-}-\frac{1}{\rho_{0}} \left.\frac{\partial u}{\partial \nu^{s}} \right\vert_{+} =0\ \text{on} \ \partial D_{s},\ s=1,...,M,\\
   \frac{\partial u^{s}}{\partial \vert x \vert}-i \kappa_{0} u^{s} =o (\frac{1}{\vert x \vert}) , \ \vert x \vert \rightarrow \infty \ (S.R.C),
 \end{cases}
\label{model4}
\end{equation}
where $\omega > 0$ is a given frequency and $\kappa_{0}^{2}=\omega^{2} \frac{\rho_0}{k_0}$ and $\kappa_{s}^{2}=\omega^{2} \frac{\rho_s}{k_s}$. Here the total field $u:=u^{I}+u^{s}$, where $u^{I}$ denotes the incident field (we restrict to
plane incident waves) and $u^{s}$ denotes the scattered waves. 
We note that the scattered field $u^s $ can be written as 
\[u^{s}(x,\theta)= \frac{e^{i \kappa_{0} \vert x \vert}}{\vert x \vert} u^{\infty}(\hat{x}, \theta)  + O(\vert x \vert^{-2}), \ \vert x \vert \rightarrow \infty, \]
where $\hat{x}:=\frac{x}{\vert x \vert} $ and $u^{\infty}(\hat{x}, \theta) $ denotes the far-field pattern corresponding to the unit vectors $\hat{x},\theta $, i.e. the incident and propagation directions respectively.                         
\bigskip

To describe the collection of small bubbles, we use the following parameters:
\begin{equation}\label{scales-bubbles}
a:=\max\limits_{1\leq m\leq M } diam (D_m) ~~\big[=\delta \max\limits_{1\leq m\leq M } diam (B_m)\big], \mbox{ and } 
d:=\min\limits_{\substack{m\neq j\\1\leq m,j\leq M }} d_{mj},
\mbox{ where }  \,d_{mj}:=dist(D_m, D_j).
\end{equation}

The distribution of the small bubbles is modeled as follows:
\begin{enumerate}
\item Given $\omega_{max}$, we take $\omega \in (0, \omega_{max}]$ and $a$ such that $\omega_{max}\; a <<1$.

\item the number $M~:=~M(a)~:=~\mathcal{O}(a^{-s})\leq M_{max} a^{-s}$ with a given positive constant $M_{max}$,\\
\item the minimum distance $d~:=~d(a)~\approx ~a^t$, i.e. $d_{min} a^t \leq d(a) \leq d_{max}a^t $, with given positive constants $d_{min}$ and $d_{max}$, \\
\item the coefficients $k_m, \rho_m$ satisfy the conditions: 
\begin{equation}
 \frac{\rho_m}{\rho_0}= C_{\rho} a^{\beta},\ \beta>0, (\mbox{ i.e. } \frac{\rho_m}{\rho_0} \ll 1),
 \label{constant}
 \end{equation}
 keeping the relative speed of propagation uniformly bounded, i.e. 
 \begin{equation}\label{speeds}
 \frac{\kappa^2_{m}}{\kappa^2_0}:= \frac{\rho_{m}k_{0}}{k_{m} \rho_{0}}=\frac{\rho_{m}}{\rho_{0}}\frac{k_{0}}{k_{m}} \sim 1, \mbox{ as } a \ll 1.
 \end{equation}
\end{enumerate}
Here the real numbers $s$, $t$ and $\beta$ are assumed to be non negative.
%\end{enumerate}

\bigskip

To state our results, let us first denote $\hat{A_{l}}:=\frac{1}{\vert{\partial D_l}\vert}\int_{ \partial D_l}\int_{ \partial D_l}\frac{(s-s')}{\vert{s-s'}\vert} \cdot\nu_{s'} \,ds'\,ds$ and define $\omega_{M}^2:=\frac{8\pi\; k_l}{(\rho_l-\rho_0) \hat{A_{l}}}.$ The constant $\omega_M$ is an approximation of the real part of the Minnaert resonance created by each bubble, see \cite{Habib-Minnaert, ACCS}. To simplify the exposition of our results in our ongoing work, all the bubbles are assumed to be identical in shape, and have the same density and bulk modulus. In particular they have the same Minneart resonance. The starting point of our work is the following point-approximation expansion of the acoustic scattered waves generated by the above cluster of bubbles, see \cite{ACCS} for more details. \\
Let \[\Phi_{\kappa}(x,y):= \frac{e^{i \kappa \vert x-y \vert}}{4\pi \vert x-y \vert}, \ \text{for} \ x,y \in \mathbb{R}^3, \] 
denote the fundamental solution of the Helmholtz equation in three dimensions with a fixed wave number $\kappa $. We recall that the farfield pattern, corresponding to $ \Phi_k$, is given by $ \Phi_{\kappa}^\infty(\hat{x}, y):= e^{-i \kappa \hat{x} \cdot y}$, where $\hat{x}=\frac{x}{\vert x \vert} $.\\

\begin{theorem}\label{Main-theorem}(see \cite{ACCS})
Under the conditions $0\leq t < \frac{1}{2}, \; 0 \leq s \leq \frac{3}{2},$ $\beta=1+\gamma,$ with $0 \leq \gamma \leq 1$ and $s+\gamma \leq 2$ we have the following expansions.
\begin{enumerate}
\item Assume that  $\gamma<1 $ or $\gamma=1$ with $\omega$ being away from $\omega_M$, i.e. $\vert 1-\frac{\omega^2_M}{\omega^2}\vert \geq l_0$ with a positive constant $l_0$ independent of $a, a \ll 1$. 
Then
\begin{equation}\label{Away-from-resonance}
 u^\infty(\hat{x}, \theta)= \sum^M_{m=1}\Phi_{\kappa_0}^\infty(\hat{x}, z_m)Q_m +\mathcal{O}(a^{2-s}+a^{3-\gamma-2t-s})
\end{equation}
under the additional condition on $t$: $t\geq \frac{s}{3}.$

\item Assume that $\gamma =1$ and the frequency $\omega$ is near $\omega_M$, i.e. $1-\frac{\omega^2_M}{\omega^2}=l_Ma^{h_1},\; h_1> 0$. Then

\begin{equation}\label{Near-resonance}
 u^\infty(\hat{x}, \theta)= \sum^M_{m=1}\Phi_{\kappa_0}^\infty(\hat{x}, z_m) Q_m+\mathcal{O}(a^{2-s-2h_{1}}+a^{3-2t-2s-2h_{1}})
\end{equation}
under the additional conditions on $t$ and $h_1$ given by
\bigskip

\begin{itemize}
 \item $t\geq \frac{s}{3}$ and $s+h_1\leq 1$ if $l_M<0$. 
\bigskip

\item $t\geq \frac{s}{3} $, $t+h_1 \leq 1 $, $s+h_1<\min \{\frac{3}{2}-t, 2-h_1 \} $, if $l_M>0$.

\end{itemize}
\end{enumerate}
\bigskip

The vector $(Q_m)^M_{m=1}$ is the solution of the following algebraic system
\begin{equation}\label{LAS-1-theorem}
 {\bf{C_m}}^{-1} Q_m +\sum_{l\neq m}\Phi_{\kappa_0}(z_l, z_m)Q_l=-u^I(z_m),\; m=1, ..., M,
\end{equation}
with
\begin{equation}
 {\bf{C_m}}:=\frac{\kappa_m^2 \vert{ D_m}\vert}{\frac{\rho_{m}}{\rho_m-\rho_{0}}-\frac{1}{8\pi}\kappa_m^2{\hat{A}_m}}\; \mbox{ and } 
 \hat{A}_m:=\frac{1}{\vert{\partial D_m}\vert}\int_{ \partial D_m}\int_{ \partial D_m}\frac{(s-s')}{\vert{s-s'}\vert} \cdot\nu_{s'} \,ds'\,ds.
 \label{Cm}
\end{equation}

 \noindent The algebraic system (\ref{LAS-1-theorem}) is invertible under one of the following conditions:
 \bigskip
 
 \begin{enumerate}
  \item The coefficients ${\bf{C_m}}$ are negative and $\max \vert {\bf{C_m}}\vert =O(a^s)$, as $a \ll 1$. This condition holds if
  \bigskip
  
  \begin{enumerate}
  \item $\gamma<1$ or $\gamma=1$ with $\omega$ being away from $\omega_M$ and we have the relations $ 0 \leq \gamma \leq 1, \; \gamma+s\leq 2$ and $ \frac{s}{3}\leq t \leq 1$.
  \bigskip
  
  \item $\gamma=1$ and the frequency $\omega$ approaches $\omega_M$ from below ($l_M<0$), i.e. $\omega < \omega_M$, and we have the relations $ \frac{s}{3} \leq t \leq 1$ and $1-h_{1}-s \geq 0 $.
\end{enumerate}
\bigskip

 \item The coefficients ${\bf{C_m}}$ are positive and one of the following conditions is fulfilled
 \begin{enumerate}
 \item $\max \vert {\bf{C_m}}\vert =O(a^t)$, as $a \ll 1$, and $\tau:=\min_{1\leq j,m \leq M,\ j\neq m} cos(\kappa_{0}\vert z_{m}-z_{j} \vert)>0$. 
 The first conditions holds if $\gamma=1$, the frequency $\omega$ approaches $\omega_M$ from above ($l_M>0$), i.e. $\omega > \omega_M$, and we have the relations $0\leq t \leq 1-h_{1} $ and $s \leq 1 $.
 \bigskip

 \item $\max \vert {\bf{C_m}}\vert =O(a^s)$, as $a\ll 1$. This condition holds if $\gamma=1$ and the frequency $\omega$ approaches $\omega_M$ from above ($l_M>0$), 
 i.e. $\omega > \omega_M$, and we have the relations $ \frac{s}{3} \leq t \leq 1$ and $1-h_{1}-s \geq 0 $.
 \end{enumerate}
 \end{enumerate}
\end{theorem}
\bigskip

From (\ref{Away-from-resonance}), (\ref{Near-resonance}) and (\ref{LAS-1-theorem}), we see that the knowledge of the parameters $\mathbf{C}_m$'s provides approximation formulas to evaluate the scattered waves. 
We call these parameters the scattering coefficients. As we will see in this work, the scales and signs of these coefficients provide the kind of properties of the equivalent media generated by the bubbles. 
Next, we describe briefly these scales and sign properties.

\subsection*{The scaling and sign of the scattering coefficients $\mathbf{C}_m $'s}\label{scaling}
We describe the scaling and sign of $\mathbf{C}_m$ in the regimes to be considered in the following sections. \\
%Let $\overline{\mathbf{C}}_m $ denote the scale-independent (independent of $a$) quantity corresponding to $\mathbf{C}_m $. \\
\begin{itemize}
\item If $\gamma<1 $ or $\gamma=1 $ and the frequency $\omega $ is away from the resonance, the scaling of $\mathbf{C}_m $ is as follows. From \eqref{Cm}, we see that
\begin{equation}
\mathbf{C}_{m}=\overline{\mathbf{C}}_m \cdot \frac{a^3}{a^{1+\gamma}-a^2}=\overline{\mathbf{C}}_m \cdot \frac{a^{2-\gamma}}{1-a^{1-\gamma}}=\overline{\mathbf{C}}_m \cdot a^{2-\gamma} \  ,
\notag
\end{equation}
where $\overline{\mathbf{C}}_m=\mathcal{O}(1) $ as $a \rightarrow 0 $.\\
When $\gamma <1 $, the term $\frac{\rho_m}{\rho_m-\rho_0} $ in \eqref{Cm} dominates and since it has a negative sign, it follows that $\mathbf{C}_m<0 $. In the case of $\gamma=1 $ and the frequency being away from the resonance, \[\mathbf{C}_m \gtrless 0  \ \text{if and only if}\ \omega \gtrless \omega_{M} .\]
\item Next we note that  the Minnaert resonance $\omega_{M}^{2} $ is of the form
\begin{equation}
\begin{aligned}
\omega_{M}^{2}= \frac{8\pi k_l}{(\rho_l-\rho_0) \hat{A}_{l}} 
&= \underbrace{-8\pi \frac{k_l}{\rho_l}  \cdot \frac{1}{\hat{A}_{l} \frac{\rho_0}{\rho_l}}}_{\sim 1}-\underbrace{8\pi \frac{k_l}{\rho_l}  \cdot \frac{1}{\hat{A}_{l} \frac{\rho_0}{\rho_l}}\frac{\rho_l}{\rho_0}}_{\mathcal{O}(a^2)}+\dots ,
\end{aligned}
\notag
\end{equation}
and when the frequency $\omega $ is near the resonance $\omega_{M}$, that is, \[1-\frac{\omega^2_M}{\omega^2}=l_Ma^{h_1},\; 0<h_1\leq 1, \ l_M \neq 0  ,\] using \eqref{Cm} we can express $\mathbf{C}_m $ in the form 
\[\mathbf{C}_m=-\frac{8 \pi \vert D_m \vert}{l_{M} a^{h_1} \hat{A}_{m}} .\] 
We set
 \begin{equation}\label{dominant-M-resonance}
\overline{\omega}^2_M:=-8\pi\frac{k_m}{\rho_0 \overline{\tilde{A}}_m},
\end{equation} where $\overline{\hat{A}}_m:=\frac{1}{\vert{\partial B_m}\vert}\int_{ \partial B_m}\int_{ \partial B_m}\frac{(s-s')}{\vert{s-s'}\vert} \cdot\nu_{s'} \,ds'\,ds$, then we can write $\mathbf{C}_m$ as
\[\mathbf{C}_m=\overline{\omega}^2_M \frac{\vert B_m\vert \rho_0}{l_M k_m}a^{1-h_1}\]
and set 
\begin{equation}\label{bar-C-m}
\overline{\mathbf{C}}_m:=\overline{\omega}^2_M \frac{\vert B_m\vert}{l_M}\frac{ \rho_0}{k_m}.
\end{equation}

Since $\hat{A}_{m} $ is negative, the sign of $\mathbf{C}_m $, therefore, is the same as $l_{M} $, that is, \[\mathbf{C}_m \gtrless 0 \ \text{if and only if}\ l_{M} \gtrless 0 .\] From the above expression, it also follows that $\mathbf{C}_m= \overline{\mathbf{C}}_{m} a^{1-h_1}$, where $\overline{\mathbf{C}}_m=\mathcal{O}(1) $ as $a \rightarrow 0 $.
\end{itemize}
\bigskip

As we assumed the bubbles to be identical in shape, and have the same density and bulk modulus, in what follows, we shall assume $\mathbf{C}_{m} $ to be same (and equal to $\mathbf{C} $) for all $m=1, \dots, M $ and the corresponding $ \overline{\mathbf{C}}_{m} $  will be denoted by $\overline{\mathbf{C}} $. For $m=1, \dots, M $, we shall also denote the quantities 
$$B_m, \rho_m, k_m, \kappa_m, \hat{A}_m $$ by $$ B, \rho, k, \kappa, \hat{A} $$
 respectively. 
From appendix \ref{scale}, we can further observe that $\overline{\mathbf{C}}  $ can be written as 
\[\overline{\mathbf{C}} = 
\begin{cases}
\overline{\mathbf{C}}_{lead}+\mathcal{O}\left( a^{1-\gamma}\right),\  &\text{if}\  \gamma<1, \\
\overline{\mathbf{C}}_{lead}+\mathcal{O}\left( a^2 \right), \  &\text{if}\  \gamma=1 \ \text{and $\omega $ is away from $\omega_M $},
\end{cases}\]
where
\begin{equation}
\overline{\mathbf{C}}_{lead}=
\begin{cases}
-\kappa^{2} \vert B \vert C^{-1}_{\rho}=-\omega^2\vert B \vert\frac{\rho_0}{k} ,\  &\text{if}\  \gamma<1, \\
-\kappa^{2} \vert B  \vert C^{-1}_{\rho} \left[1+\frac{1}{8\pi} \kappa^{2} \overline{\hat{A}} C^{-1}_{\rho} \right]^{-1}=-\omega^2 \frac{\vert B \vert}{1-\frac{\omega^2}{\overline{\omega}_{M}^2}}\frac{\rho_0}{k}, \  &\text{if}\  \gamma=1 \ \text{and $\omega $ is away from $\omega_M $}.
\end{cases}
\notag
\end{equation}

Finally, from the above approximation of $\omega_M$, we observe that 
\begin{equation}
\omega^2_M=\overline{\omega}^2_M+\mathcal{O}(a^2). 
\end{equation}
\bigskip

Based on Theorem \ref{Main-theorem} and the discussion right after, our goal is to describe and quantify the dominating fields related to the approximations (\ref{Away-from-resonance}) and (\ref{Near-resonance})
 in appropriate regimes related to the distributions of the cluster of bubbles.  
\bigskip

\section{Main results}\label{Main-Results}
In this section, we state the main results of our work. For this, we divide it into three subsections. In the first one, we describe the regimes under which our cluster have no effect on the background media, i.e.
the scattered waves are not affected by the cluster when the radius $a$ becomes small. This is done regardless on how the cluster is distributed.
In the second subsection, we distribute the small bubbles in volumetric sets and describe the regimes under which we can have responses.
In the third subsection, we distribute the small bubbles on (close or open) surfaces. In these two subsections, we describe the equivalent dominating fields in the different regimes and discuss
the relevance of these results.

\subsection{The case when the bubbles have no effect on the background media}
We have seen in \eqref{Away-from-resonance} and \eqref{Near-resonance} that in suitable regimes, the far field can be expressed as \[u^\infty(\hat{x}, \theta)= \sum^M_{m=1}\Phi_{\kappa_0}^\infty(\hat{x}, z_m)Q_m + o(1), \ a \rightarrow 0. \] 
Also $\left\vert  \sum^M_{m=1}\Phi_{\kappa_0}^\infty(\hat{x}, z_m)Q_m \right\vert \leq M \max \vert \mathbf{C} \vert$. Therefore 
\begin{itemize}
\item if $\gamma<1 $ and $\gamma+s<2 $, we have $
M \max \vert \mathbf{C} \vert= \mathcal{O}(a^{-s} \cdot a^{2-\gamma})=\mathcal{O}(a^{2-\gamma-s}) =o(1), \ a \rightarrow 0$ and hence $u^\infty(\hat{x}, \theta) \rightarrow 0,\ \text{as}\ a \rightarrow 0$.

\item if $\gamma=1 $ and the frequency is away from the resonance with $s<1 $, we have 
$M \max \vert \mathbf{C}  \vert= \mathcal{O}(a^{-s} \cdot a^{2-\gamma} )=\mathcal{O}(a^{2-\gamma-s})=\mathcal{O}(a^{1-s})=o(1),\ a \rightarrow 0$ and hence $u^\infty(\hat{x}, \theta) \rightarrow 0,\ \text{as}\ a \rightarrow 0$.

\item if $\gamma=1 $ and the frequency is near the resonance with $s+h_{1}<1 $, we have
$M \max \vert \mathbf{C}  \vert= \mathcal{O}(a^{-s} \cdot a^{1-h_{1}})=\mathcal{O}(a^{1-h_{1}-s})=o(1),\ a \rightarrow 0$ and hence $u^\infty(\hat{x}, \theta) \rightarrow 0,\ \text{as}\ a \rightarrow 0$.

\end{itemize}

Thus in these cases, the bubbles have no effect on the background media, as $a<<1$, irrespective of their volumetric or surface distribution.
\subsection{Application to volumetric metamaterials}
Let us now discuss about the volumetric distribution of the bubbles (see also \cite{ACKS,CMS}).

Let $\Omega$ be a bounded domain, say of unit volume. 
We divide $\Omega$ into $[a^{-s}]$ subdomains $\Omega_m,\; m=1, ..., [a^{-s}]$, \footnote{As an example, taking $a:=N^{-\frac{1}{s}}$, with $N$ an integer and $N>>1$, we have $a<<1$ and $[a^{-s}]=N$.} such that each 
$\Omega_m$ contains $D_m$, i.e. $z_m \in \Omega_m$, and some of the other $D_j$'s. We assume that the number of bubbles in each $\Omega_m$, 
for $m=1, ..., [a^{-s}]$, is uniformly bounded in terms of $m$. To describe correctly this number of obstacles, let us be given a function 
 $K: \mathbb{R}^3\rightarrow \mathbb{R}$ as a non-negative, continuous and bounded potential. 
 Let each $\Omega_m$, $m\in \mathbb{N}$, be a cube of volume $a^s\frac{[K(z_m)+1]}{K(z_m)+1}$
 and contains $[K(z_m) +1]\; (=[K(z_m)] +1)$ bubbles \footnote{For a given real and positive number $x$, we denote by $[x]$, the unique integer $n$ 
 such that $n\leq x \leq n+1$, i.e. $n$ is the floor number.}. We set $K_{max}:=\sup_{z_m}[K(z_m) +1]$, hence $M=\sum^{[a^{-s}]}_{m=1}[K(z_m) +1]\leq K_{max}[a^{-s}]=O(a^{-s})$. The function $K$ describes then the local distribution of the holes, i.e. the number of bubbles in each $\Omega_m$ is fixed as $[K(z_m) +1]$.
 \bigskip

One way to do it, using a given function $K$, is to put the location $z_1 $ of the first bubble $D_1 $ in the 'center' of $\Omega$ and 
 then surround it with the cube $\Omega_1$ of volume $a^s \frac{[K(z_1)+1]}{K(z_1)+1}$. Inside $\Omega_1$, add the other $[K(z_1)]$ bubbles. Starting from $\Omega_1$, build up the other $\Omega_m$'s in a Rubik style respecting their volumes and the number of 
 bubbles included inside them using the function $K$ as discussed above.

 Few remarks are in order:
 \begin{enumerate}
  \item If we distribute the bubbles periodically, then $K\equiv 0$, the $\Omega_m$'s are identical (modulo a translation) and $\vert \Omega_m\vert=a^s$.
  \item $K$ can be identically zero but the bubbles can be distributed non-periodically. In general, if $K_{\mid_{\Omega}}$ is an integer, then also $\vert \Omega_m\vert=a^s\frac{[K(z_m)+1]}{K(z_m)+1} =a^s$ and the bubbles can be distributed non-periodically.
  \item Assume now that $K$ is, eventually, a variable function. Hence $\vert \Omega_m\vert =a^s \frac{[K(z_m)+1]}{K(z_m)+1}<a^s$ and 
  $Vol(\lim_{a \rightarrow 0}\cup^{[a^{-s}]}_{m=1}\Omega_m)
  =\int_{\Omega}\frac{[K(z)+1]}{K(z)+1}dz < \vert \Omega \vert$. Hence $\lim_{a \rightarrow 0}\cup^{[a^{-s}]}_{m=1}\Omega_m \subsetneq \Omega$. In this case, to $\Omega$ 
  we cut\footnote{Recall that we denoted by $[x]$ the floor number, i.e. the unique integer $n$ such that $n\leq x \leq n+1$.} a layer of volume $\vert \Omega \vert -\int_{\Omega}\frac{[K(z)+1]}{K(z)+1}dz$ and a constant depth starting from $\partial \Omega$. 
  We denote the resulting domain by $\Omega$ too. Note that this last domain has the same regularity as the former.
   
 \end{enumerate}

 \bigskip

 As $\Omega$ can have an arbitrary shape, the set of the cubes intersecting $\partial \Omega$ is not empty (unless if $\Omega$ has a simple shape as a cube). 
 Later in our analysis, we will need the estimate of the volume of this set. 
Since each $\Omega_j$ has volume of the order $a^s$, and then its maximum radius is of the order $a^{\frac{1}{3}s}$, then the intersecting surfaces with $\partial \Omega$ has an area of the order $a^{\frac{2}{3}s}$.
As the area of $\partial \Omega$ is of the order one, we conclude that the number of such cubes will not exceed the order $a^{-\frac{2}{3}s}$. Hence the volume of this set will not exceed the order 
$a^{-\frac{2}{3}s}a^{s}=a^{\frac{1}{3}s}$, as $a \rightarrow 0$.

\bigskip

\begin{theorem}\label{Main-theorem-volumes}\footnote{The error terms in \eqref{vol1}, \eqref{vol3} and \eqref{vol4} are given explicitly in terms of the corresponding used parameters in \eqref{gamma-small-vol}-\eqref{gamma-away-vol}, \eqref{gamma-near-vol} and \eqref{volume-blow} respectively.}
Let the bubbles be distributed in a bounded domain $\Omega $ according to a given non-negative, real-valued function $K \in C^{0,\lambda}(\overline{\Omega}),\; \lambda \in (0, 1) $, with their number $M:=M(a):=\mathcal{O}(a^{-s}) $ and their minimum distance $d:=d(a):=a^{t} $ as described above.\\
Let us consider the scattering problem
\begin{equation}
\begin{aligned}
&\left(\Delta + n(x) \right)u^{t}_{a}=0,\ \text{in} \ \mathbb{R}^{3},
\end{aligned}
\notag
\end{equation}
\begin{equation}
\begin{aligned}
&u^{t}_{a}=u^{s}_{a}+e^{i \kappa_{0} x \cdot \theta},
\end{aligned}
\notag
\end{equation}
\begin{equation}
\begin{aligned}
&\frac{\partial u^{s}_{a}}{\partial \vert x \vert}-i \kappa_{0} u^{s}_{a} =o\left(\frac{1}{\vert x \vert} \right), \ \vert x \vert\rightarrow \infty,
\end{aligned}
\notag
\end{equation}
\begin{itemize}
\item[(1)] Suppose the conditions $0\leq t<\frac{1}{2} $, $\beta=1+\gamma $ hold and 
\begin{itemize}
\item[(a)] either $\gamma<1, \gamma+s=2$, or \\
\item[(b)] $\gamma=1,\ s=1, $ and $\omega $ away from Minnaert resonance.
\end{itemize}
Then
\begin{equation}
\begin{aligned}
u^{\infty}(\hat{x},\theta)-u^{\infty}_{a}(\hat{x},\theta)=o(1), \mbox{ as } a<<1, \mbox{ uniformly in terms of } \hat{x} \mbox{ and } \theta.
\end{aligned}
\label{vol1}
\end{equation}
In the case $(a)$,
 \begin{equation}
 n:=\omega^2 \rho_0\left[k_{0}^{-1} + (K+1)\vert B \vert k^{-1} \chi_{\Omega} \right] 
 \label{neff-1}
 \end{equation} 
 and in the case $(b)$,
 \begin{equation} 
  n:=\omega^2 \rho_0\left[k_{0}^{-1}+(K+1) \frac{\vert B \vert}{1-\frac{\omega^2}{\overline{\omega}_{M}^2}} k^{-1} \chi_{\Omega} \right].
  \label{neff-2}
  \end{equation}
\item[(2)] Suppose that $\gamma=1$ and $\omega $ is near the Minnaert resonance,i.e. $1-\frac{\omega^2_M}{\omega^2}=l_M a^{h_1}$, with $l_M \neq 0$ and $h_1 \in (0, 1)$ where $s$ and $t$ satisfying the conditions \[\ s=1-h_1 \mbox{ and } \frac{s}{3} \leq t<\min\{1-h_1,\frac{1}{2}\}.\]
Then
\begin{equation}
\begin{aligned}
u^{\infty}(\hat{x},\theta)-u^{\infty}_{a}(\hat{x},\theta)=o(1), \mbox{ as } a<<1, \mbox{ uniformly in terms of } \hat{x} \mbox{ and } \theta.
\end{aligned}
\label{vol3}
\end{equation}
In this case, 
\begin{equation}
 n:= \overline{\omega}^2_M \rho_0 \left[k_{0}^{-1} -(K+1)\frac{\vert B \vert}{l_M}k^{-1} \chi_{\Omega} \right] .
 \label{neff-3}
 \end{equation}
 %\end{itemize}
\item[(3)] Suppose that $\gamma=1$ and $\omega $ is near the Minnaert resonance, i.e. $1-\frac{\omega^2_M}{\omega^2}=l_M a^{h_1}$, with $l_M>0$ and $h_1 $, $s$ and $t$ satisfy the conditions
\begin{equation}
\begin{aligned}
 0<1-h_{1}<s \leq 3t< \min \left\{\frac{3}{2}-t-h_{1}, \left(1+\frac{2\lambda}{15}\right)(1-h_1) \right\},\ h_1 < \frac{1}{6}.
\end{aligned}
\notag
\end{equation}
Then provided $\kappa_{0}^{2} $ is not an eigenvalue for the Dirichlet Laplacian in $\Omega $, we have
\begin{equation} 
\begin{aligned}
u^{\infty}(\hat{x},\theta)-u^{\infty}_{D}(\hat{x},\theta)=o(1), \mbox{ as } a<<1, \mbox{ uniformly in terms of } \hat{x} \mbox{ and } \theta,
\end{aligned}
\label{vol4}
\end{equation}
where $u^{\infty}_{D} $ is determined by the exterior Dirichlet problem
\begin{equation}
\begin{aligned}
\left(\Delta+\kappa_{0}^{2} \right)u^{t}_{D}&=0, \ \text{in}\ \mathbb{R}^{3}\setminus \overline{\Omega},
\end{aligned}
\notag
\end{equation}
\begin{equation}
\begin{aligned}
u^{t}_{D}:=u^{s}_{D}+e^{i\kappa_{0} x \cdot \theta}&=0,\ \text{on}\ \partial \Omega,
\end{aligned}
\notag
\end{equation}
\begin{equation}
\begin{aligned}
\frac{\partial u^{s}_{D}}{\partial \nu}-i\kappa_0 u^{s}_{D}&=o\left(\frac{1}{\vert x \vert} \right),\ \vert x \vert \rightarrow \infty.
\end{aligned}
\notag
\end{equation}
\end{itemize}
\end{theorem}

\bigskip

According to these results, we distinguish three regimes regarding the denseness of the bubbles.
\bigskip

\begin{enumerate}
\item {\it{Low regime}}. This regime is related to one of the following conditions: ($\gamma<1 $ and $\gamma+s<2 $) or ($\gamma=1 $ and the frequency is away from the Minnaert resonance with $s<1 $) or ($\gamma=1 $ and the frequency is near the resonance with $s+h_{1}<1 $). Under these conditions, the scattered fields vanish as $a<<1$, meaning that the corresponding cluster is weak and reflects no incident wave. 
\bigskip

\item {\it{Medium regime}}. This regime is related to the following conditions:
\begin{enumerate}
\item [(a)] If $\gamma<1, \gamma+s=2$ (in which case, $\omega$ is of course away from the Minnaert resonance), then the effective medium is composed of the background one to which we add, locally in $\Omega$, a positive term coming from the cluster. 
\bigskip

\item[(b)] If $\gamma=1,\ s=1, $ and $\omega $ away from the resonance. In this case, the effective medium is composed of the background one to which we add, locally in $\Omega$, a term coming from the cluster which changes sign whether the frequency is lower or higher than the Minnaert resonance. When the sign is positive, this means we have more reflection and otherwise we have more transmission. Hence sending incident waves at frequencies lower or higher than the Minnaert resonance $\omega_M$, the medium changes its behavior from transmitting to reflecting. The parameters modeling the bubbles, as the shape and the contrasts, which we can tune to get the sign we wish, are at our disposal. Then we can increase or decrease the transmission (or the reflection) by tuning appropriately these bubbles. 

\item [(c)] If $\gamma=1 $ and the frequency is near the resonance with $s+h_{1}=1 $, then we are in the same situation as in $(b)$. The difference is that, as the cluster density is estimated as $a^{-s}$ with $s=1-h_{1}$ and we can choose $h_1$ as close as we want to $1$, we see that we can derive effective medium having the same properties as in $(b)$, but with a very low number of bubbles, namely $M\sim a^{-s}$ with $s$ as close as we want to $0$. Hence with very low number of bubbles but using incident frequencies very close to the Minnaert resonance we can achieve same effects as if we use a quite dense cluster $(s=1)$ but with non resonating incident frequencies. 
This is in accordance with what is believed in the engineering community.  
\end{enumerate}
We observe in the form of the coefficient $n$ that after adding the bubbles, the density $\rho_0$ of the background did not change while the bulk mudulus $k_0$ is perturbed locally in $\Omega$. 
\bigskip

\item {\it{High regime}}. If $\omega $ is near, but larger than (with $l_M>0$), the Minnaert resonance $\omega_M$ and $1< s+h_{1}<\min \{\frac{3}{2}-t, 2-h_1\}$\footnote{The relevant condition is $1< s+h_{1}$ while the other part $s+h_{1}<\min \{ \frac{3}{2}-t,2-h_1\}$ is due to technical limitations.}, then the medium behaves as a totally reflecting one, i.e. as wall. Clusters of bubbles with densities of the order $M\sim a^{-s}$ with $s+h_{1} > 1$, allow no incident waves, sent at nearly resonating frequencies, to penetrate.   
\end{enumerate}

Finally, it is worth mentioning that for any given frequency of incidence $\omega$, we can choose (or tune) the properties of the bubbles so that the corresponding Minnaert resonance will be located near or close to it.
This way, for a given frequency $\omega$, we can be in any of the regimes described above by appropriately choosing the bubbles. Hence, we can tune the bubbles so that the incident sound, sent at the frequency $\omega$, will be more/less reflected or transmitted, across $\partial \Omega$, at our will.

\subsection{Application to metasurfaces}
Let $\Sigma \subset \mathbb{R}^{3}$ be such that either
\begin{itemize}
\item $\Sigma=\partial D $ for some open connected subset $D$ of  $\mathbb{R}^{3}$, or
\item $\Sigma$ is an open subset of $\Gamma $, where $\Gamma=\partial D  $ for some open connected subset $D$ of  $\mathbb{R}^{3}$.
\end{itemize}

In order to find a convenient way for counting the bubbles, we shall further assume that $\Sigma $ can be parametrized by a finite number of charts and we shall work with the image of a single such chart (which we shall continue to denote by $\Sigma $) at a time. In this way, we shall possibly overcount the bubbles but since we shall have to deal with only finite number of charts, the error estimates would still be valid.\\
Without loss of generality, let $\Sigma$ be of unit surface area. Now given a non-negative, Holder continuous function\footnote{By an abuse of notation, we denote the function by $K$ even in the case of distribution on the surfaces. Unlike the case of volumetric distributions, this function is only defined on the surface.} $K:\partial D \rightarrow \mathbb{R} $ , let $\Sigma_j, \ j=1,\dots,[a^{-s}] $ be a square (or quadrilateral) of area $a^{s} \frac{[K(z_j)+1]}{K(z_j)+1} $ which contains the centers of $[K(z_j)+1] $ bubbles. 
We fill-in $\Sigma$ with the $[a^{-s}]$ subdomains $\Sigma_j,\; j=1, ..., [a^{-s}]$ in a similar way as we did for the volumetric distribution.

In addition, similar to the case of volumetric distribution, we can now estimate the total area of squares $\Sigma_j $ touching the boundary $\partial \Sigma $ of $\Sigma$ as follows. 
Since area of each $\Sigma_{j} $ is of the order $a^{s} $, the radius of $\Sigma_j $ is of the order $a^{\frac{s}{2}} $. Therefore the length of intersecting curves with $\partial \Sigma $ is of order $a^{\frac{s}{2}}$. Now since the length of $\partial \Sigma $ is of order one, it follows that the number of such squares is not more than order $a^{-\frac{s}{2}} $. Therefore the total area covered by such $\Sigma_{j} $'s cannot exceed $a^{-\frac{s}{2}} \cdot a^{s}=a^{\frac{s}{2}} $.
\begin{theorem}\label{Main-theorem-surfaces}\footnote{The error terms in \eqref{sur1}, \eqref{sur3} and \eqref{sur4} are given explicitly in terms of the corresponding used parameters in \eqref{gamma-small-sur}-\eqref{gamma-away-sur}, \eqref{gamma-near-sur} and \eqref{final-sur-blow} respectively.}
Let the locations $z_{m}$'s  of the bubbles be distributed on a bounded subset $\Sigma $ (say of unit area) according to a given non-negative, real-valued function $K \in C^{0,\lambda}(\overline{\Sigma}) $ with their number $M:=M(a):=\mathcal{O}(a^{-s}) $ and their minimum distance $d:=d(a):=a^{t} $ as described above.\\
Let us consider the scattering problem
\begin{equation}
\begin{aligned}
&\left(\Delta +\kappa_{0}^{2} \right)u^{t}_{a}=0,\ \text{in} \ \mathbb{R}^{3}\setminus \Sigma,
\end{aligned}
\notag
\end{equation}
\begin{equation}
[u^{t}_{a}]=0, \ \left[\frac{\partial u^{t}_{a}}{\partial \nu} \right]-\sigma u^{t}_{a}=0, \ \text{on} \ \Sigma,
\notag
\end{equation}
\begin{equation}
\begin{aligned}
&u^{t}_{a}=u^{s}_{a}+e^{i \kappa_{0} x \cdot \theta},
\end{aligned}
\notag
\end{equation}
\begin{equation}
\begin{aligned}
&\frac{\partial u^{s}_{a}}{\partial \vert x \vert}-i \kappa_{0} u^{s}_{a} =o\left(\frac{1}{\vert x \vert} \right), \ \vert x \vert\rightarrow \infty,
\end{aligned}
\notag
\end{equation}
\begin{itemize}
\item[(1)] Suppose the conditions $0\leq t<\frac{1}{2} $, $\beta=1+\gamma $ hold and 
\begin{itemize}
\item[(a)] either $\gamma<1, \gamma+s=2$, or \\
\item[(b)] $\gamma=1, s=1,$ and $\omega $ away from Minnaert resonance.
\end{itemize}
Then
\begin{equation}
\begin{aligned}
u^{\infty}(\hat{x},\theta)-u^{\infty}_{a}(\hat{x},\theta)=o(1), \mbox{ as } a<<1, \mbox{ uniformly in terms of } \hat{x} \mbox{ and } \theta.
\end{aligned}
\label{sur1}
\end{equation}
In the case $(a)$, \[ \sigma:=-\omega^2 (K+1) \vert B  \vert \frac{\rho_0}{k} \] and in the case $(b)$, \[\sigma:=-\omega^2 (K+1)\frac{\vert B  \vert}{1-\frac{\omega^2}{\overline{\omega}_{M}^2}} \frac{\rho_0}{k} .\]
\item[(2)] Suppose that $\gamma=1$ and $\omega $ is near the Minnaert resonance,i.e. $1-\frac{\omega^2_M}{\omega^2}=l_M a^{h_1}$, with $l_M \neq 0$ and $h_1 \in (0, 1)$ where $s$ and $t$ satisfying the conditions \[\ s=1-h_1 \mbox{ and } \frac{s}{3} \leq t<\min\{1-h_1,\frac{1}{2}\}.\]
Then
\begin{equation}
\begin{aligned}
u^{\infty}(\hat{x},\theta)-u^{\infty}_{a}(\hat{x},\theta)=o(1), \mbox{ as } a<<1, \mbox{ uniformly in terms of } \hat{x} \mbox{ and } \theta.
\end{aligned}
\label{sur3}
\end{equation}
In this case, \[\sigma:= \overline{\omega}^2_M (K+1) \frac{\vert B \vert}{l_M}\frac{ \rho_0}{k} .\]
%\end{itemize}
\item[(3)] Suppose that $\gamma=1$ and $\omega $ is near the Minnaert resonance, i.e. $1-\frac{\omega^2_M}{\omega^2}=l_M a^{h_1}$, with $l_M>0$ and $h_1$, $s$ and $t$ satisfy the conditions
\begin{equation}
\begin{aligned}
 0<1-h_{1}<s \leq 3t< \min \left\{\frac{3}{2}-t-h_{1}, \left(1+\frac{\lambda}{7}\right)(1-h_1) \right\},\ h_1 < \frac{1}{6}.
\end{aligned}
\notag
\end{equation}
Then provided $\kappa_{0}^{2} $ is not an eigenvalue for the Dirichlet Laplacian in $D$, we have
\begin{equation}
\begin{aligned}
&u^{\infty}(\hat{x},\theta)-u^{\infty}_{D}(\hat{x},\theta)=o(1), \mbox{ as } a<<1, \mbox{ uniformly in terms of } \hat{x} \mbox{ and } \theta
\end{aligned}
\label{sur4}
\end{equation}
where
\begin{itemize}
\item if $\Sigma $ is an open surface, $u^{t}_{D}$ is determined by the Dirichlet crack problem 
\begin{equation}
\left(\Delta+\kappa_{0}^{2} \right)u^{t}_{D}=0, \ \text{in}\ \mathbb{R}^{3}\setminus \Sigma,
\notag
\end{equation}
\begin{equation}
u^{t}_{D}=0,\ \text{on}\ \Sigma,
\notag
\end{equation}
with the Sommerfeld radiation conditions satisfied by $u^{t}_{D}-u^{I} $, and
\item if $\Sigma=\partial D $ for some connected open subset $D \subset \mathbb{R}^3$, 
then $u^{t}_{D}$ is the unique solution to the exterior Dirichlet problem 
\begin{equation}
\left(\Delta+\kappa_{0}^{2} \right)u^{t}_{D}=0, \ \text{in}\ \mathbb{R}^{3}\setminus \overline{D},
\notag
\end{equation}
\begin{equation}
u^{t}_{D}=0,\ \text{on}\ \partial D,
\notag
\end{equation}
with the Sommerfeld radiation conditions satisfied by $u^{t}_{D}-u^{I} $.
\end{itemize}
\end{itemize}
\end{theorem}

\bigskip

As in the $3$ D case, we distinguish three regimes.

\begin{enumerate}
\item {\it{Low regime}}. If ($\gamma<1 $ and $\gamma+s<2 $) or ($\gamma=1 $ and the frequency is away from the Minnaert resonance with $s<1 $) or ($\gamma=1 $ and the frequency is near the resonance with $s+h_{1}<1 $), then the scattered fields vanish as $a<<1$, i.e. the corresponding cluster is weak and reflects no incident wave. 
\bigskip

\item {\it{Medium regime}}. While the contribution coming from a cluster of bubbles distributed in a volumetric domain $\Omega$
 is represented by a $3$ D potential supported in $\Omega$, precisely $n_{eff}\; \chi_\Omega$\footnote{where $n_{eff}:=n- \omega^2 \rho_0 k_{0}^{-1}  $ in the cases \eqref{neff-1} and \eqref{neff-2}, and $n_{eff}:=n-\overline{\omega}_{M}^2 \rho_0 k_{0}^{-1}  $ in the case \eqref{neff-3}. }, with a well characterized potential $n_{eff}$, the contribution coming from a cluster of bubbles distributed in a surface $\Sigma$ is represented by a Dirac potential supported on $\Sigma$, precisely $\sigma \  \delta_\Sigma$. As for the $3$ D distribution of the bubbles, we have the following properties that are encoded in the definition of the coefficient $\sigma$.
\bigskip

The properties of these surface potentials differ according to the following sub-regimes:
\begin{enumerate}
\item [(a)] If $\gamma<1, \gamma+s=2$ (and then $\omega$ is away from the Minnaert resonance), then the surface potential supported on $\Sigma$ has a multiplicative density which is positive.
The amplitude of this density describes to what extent the screen $\Sigma$ is reflecting, i.e. there is more reflection than transmission as 
\begin{equation}\label{ratio-surface}
\frac{\left[\frac{\partial u^{t}_{a}}{\partial \nu} \right]}{u^{t}_{a}}=\sigma>0.
\end{equation}

\bigskip

\item[(b)] If $\gamma=1,\ s=1, $ and $\omega $ away from resonance, then the surface potential  coming from the cluster changes sign whether the frequency is lower or higher than the Minnaert resonance. When the sign is positive, this means we have more reflection, see \eqref{ratio-surface}  and otherwise we have more transmission as in this case
\begin{equation}\label{ratio-surface-negative}
\frac{\left[\frac{\partial u^{t}_{a}}{\partial \nu} \right]}{u^{t}_{a}}=\sigma<0.
\end{equation} 
As the coefficient $\sigma$ is given by parameters modeling the bubbles, which are at our disposal, then we can increase or decrease the transmission (or the reflection) by tuning appropriately these bubbles. 

\item [(c)] If $\gamma=1 $ and the frequency is near the resonance with $s+h_{1}=1 $, then we are in the same situation as in $(b)$. However, and as in the $3$ D case, we can derive the effective surface potential having the same properties as in $(b)$, but with a very low number of bubbles, namely $M\sim a^{-s}$ with $s$ as close as we want to $0$ using incident frequencies very close to the Minnaert resonance with $h_1$ very close to $1$, through the relation $s=1-h_{1}$.   
\end{enumerate}
\bigskip

\item {\it{High regime}}. If $\omega $ is near, but larger than (with $l_M>0$) the Minnaert resonance $\omega_M$ and $1< s+h_{1}<\min \{ \frac{3}{2}-t, 2-h_1\}$, then the surface behaves as a totally reflecting one, i.e. as a dark screen. Clusters of bubbles with densities of the order $M\sim a^{-s}$ with $s+h_{1} > 1$, and distributed allong a given surface $\Sigma$, allow no incident waves, sent at nearly resonating frequencies, to be transmitted through $\Sigma$.   
\end{enumerate}

As mentioned for the volumetric metamaterials, for a given frequency $\omega$, we can be in any of the regimes described above by appropriately choosing the bubbles. This means that for a given frequency of incidence $\omega$, we can tune the bubbles so that the incident sound, sent at the frequency $\omega$, will be more/less reflected or transmitted, across $\Sigma$, at our will.   
\bigskip

Here the surface $\Sigma$ can be open or closed. In the case when it is closed, then an incoming incident waves \textquoteleft comes from outside' and the meaning of (\ref{ratio-surface}) is that the total reflected energy $\frac{\frac{\partial (u^{t}_{a})_+}{\partial \nu}}{(u^{t}_{a})_+}$ is larger than the transmitted one $\frac{\frac{\partial (u^{t}_{a})_-}{\partial \nu}}{(u^{t}_{a})_-}$.
Here the subscripts $+$ and $-$ refer to total fields outside and inside of the domain that encloses $\Sigma$.  In the case when it is open, the interpretations of  (\ref{ratio-surface}) and  (\ref{ratio-surface-negative}) need more precision. Even though we assumed that $\Sigma$ is a part of a closed surface and the orientation of $\Sigma$ is inherited from the one of the closed one, an incident wave can be incoming or outgoing or mixed. Hence, in these cases, (\ref{ratio-surface}) and  (\ref{ratio-surface-negative}) should be interpreted as explained above for incoming incident, in the opposite way for outgoing waves and both (and accordingly) for the mixed case.   
\bigskip

\section{Formal arguments}\label{Formal-Arguments}
In this section, we outline some formal arguments which motivate our strategy for the proofs of the main results.
\subsection{Arguments to generate volumetric metamaterials}
We use $s^*$ as a parameter which can be equal $2- \gamma$, or $1-h_1$, $h_1 \in [0, 1)$. We rewrite (\ref{LAS-1-theorem}) as 
$$
 -\frac{Q_m}{{\bf{C}}} +a^{(s^*-s})\sum_{\substack{j=1 \\ j\neq m}}^{M}\overline{\bf{C}} a^{s}\; \Phi_{\kappa_0}(z_m,z_j)\left(-\frac{Q_j}{{\bf{C}}}\right)=u^{i}(z_m, \theta),~~
$$
for $ m=1,..., M$. Similarly, we rewrite the representation (\ref{Away-from-resonance}) (respectively \eqref{Near-resonance}) as
\begin{eqnarray}\label{x oustdie1 D_m farmain-recent***}
u^\infty(\hat{x},\theta)
\hspace{-.05cm}=\hspace{-.1cm}\;- a^{(s^*-s)} \sum_{m=1}^{M}e^{-i{\kappa_0}\hat{x}\cdot z_m}\overline{\bf{C}} a^{s} \left(-\frac{Q_m}{{\bf{C}}}\right)\hspace{-.03cm}+\hspace{-.03cm}o(1), \; a \rightarrow 0. 
 \end{eqnarray}
 Let us introduce the Lippmann-Schwinger  equation 
 \begin{equation}\label{Acoustic-model}
 Y +a^{(s^*-s)}\int_{\Omega}\overline{\bf{C}}(K(z)+1)\Phi_{\kappa_0}(\cdot, z)Y(z)\; dz\; =u^{i}(\cdot, \theta) 
 \end{equation} modeling the unique solution of the problem $
 \Delta Y +\kappa_0^2 Y -a^{(s^*-s)}\overline{\bf{C}}(K(z)+1)\chi_{\Omega}Y\; =0$ with (S.R.C).
The far-field corresponding to the solution of (\ref{Acoustic-model}) has the form
\begin{equation}\label{far-field-acoustic}
 Y^{\infty}(\hat{x}, \theta):=-a^{(s^*-s)}\int_{\Omega}e^{-i{\kappa_0} \hat{x}\cdot z} \overline{\bf{C}}(K(z)+1)Y(z)dz.
\end{equation}

Based on the fact that $M=\sum_{m=1}^{[a^{-s}]}\sum^{[K(z_m)]+1}_{j=1} 1$, $\Omega=\lim_{a \rightarrow 0
}\cup^{[a^{-s}]}_{j=1} \Omega_j$ with the volume of $ \Omega_j$, for $j=1,..., [a^{-s}]$, equals $a^s\frac{[K(z_j)]+1}{K(z_j)+1}$, we derive the approximation $
  u^\infty(\hat{x},\theta)-Y^\infty(\hat{x},\theta) = o(1),\; a \rightarrow 0. $

\subsection{The arguments to generate metascreens}

As for the metametrial case, we use $s^*$ as a parameter which can be equal $2- \gamma$, or $1-h_1$, $h_1 \in [0, 1)$ and rewrite (\ref{LAS-1-theorem}) as 
$$
 -\frac{Q_m}{{\bf{C}}} +a^{(s^*-s)}\sum_{\substack{j=1 \\ j\neq m}}^{M}\overline{\bf{C}} a^{s}\; \Phi_{\kappa_0}(z_m,z_j)\left(-\frac{Q_j}{{\bf{C}}}\right)=u^{i}(z_m, \theta),~~
$$
for $ m=1,..., M$. Similarly, we rewrite the representation (\ref{Away-from-resonance}) (respectively \eqref{Near-resonance}) as
\begin{eqnarray}\label{x oustdie1 D_m farmain-recent***2}
u^\infty(\hat{x},\theta)
\hspace{-.05cm}=\hspace{-.1cm}\;- a^{(s^*-s)}\sum_{m=1}^{M}e^{-i{\kappa_0}\hat{x}\cdot z_m}\overline{\bf{C}} a^{s} \left(-\frac{Q_m}{{\bf{C}}}\right)\hspace{-.03cm}+\hspace{-.03cm}o(1), \; a \rightarrow 0. 
 \end{eqnarray}
 Let us now introduce the boundary integral  equation on the surface $\Sigma$:
 \begin{equation}\label{Acoustic-model-2}
 Y +a^{(s^*-s)}\int_{\Sigma}\overline{\bf{C}}(K(z)+1)\Phi_{\kappa_0}(\cdot, z)Y(z)\; dz\; =u^{i}(\cdot, \theta). 
 \end{equation} Then $Y:=-a^{(s^*-s)}\int_{\Sigma}\overline{\bf{C}}(K(z)+1)\Phi_{\kappa_0}(\cdot, z)Y(z)\; dz\;+u^{i}(\cdot, \theta)$  solves the problem $
 \Delta Y +\kappa_0^2 Y\; =0$ in $\mathbb{R}^3\setminus{\bar{\Sigma}}$  with (S.R.C) and the transmission conditions $[Y]=0$ and $[\nabla Y\cdot \nu]+a^{(s^*-s)}\overline{\bf{C}}(K(z)+1)Y=0$ across $\Sigma$.
The corresponding far-field  is
\begin{equation}\label{far-field-acoustic2}
 Y^{\infty}(\hat{x}, \theta):=-a^{(s^*-s)}\int_{\Sigma}e^{-i{\kappa_0} \hat{x}\cdot z}\overline{\bf{C}}(K(z)+1)Y(z)dz.
\end{equation}
Based on the fact that $M=\sum_{m=1}^{[a^{-s}]}\sum^{[K(z_m)]+1}_{j=1} 1$, $\Sigma=\lim_{a \rightarrow 0
}\cup^{[a^{-s}]}_{j=1} \Sigma_j$ with the area of $ \Sigma_j$, for $j=1,..., [a^{-s}]$, equals to $a^s\frac{[K(z_j)]+1}{K(z_j)+1}$, we derive the approximation $
  u^\infty(\hat{x},\theta)-Y^\infty(\hat{x},\theta) = o(1),\; a \rightarrow 0. $
\bigskip
~
\bigskip

For both the volumetric and surface distribution of the bubbles, we see that as $s<s^*$, $Y^\infty(\hat{x},\theta)=o(1)$. Also, we have the exact limiting models when $s=s^*$. What is left is to characterize the limiting models when $s>s^*$. We describe this case in the following subsections. 

\subsection{The extreme cases $s>s^*$: Volumetric metamaterials}
Recall that \[Y +a^{(s^*-s)}\int_{\Omega}\overline{\bf{C}}(K(z)+1)\Phi_{\kappa_0}(\cdot, z)Y(z)\; dz\; =u^{i}(\cdot, \theta).\] The question is how to characterize $\lim_{a \rightarrow 0} Y(\cdot,\theta)?$
We set $h:= a^{\frac{s-s^*}{2}}$, $V_0:=\overline{\bf{C}}(K(z)+1)$. First, we show that $Y$ satisfies the following boundary-value problem
\begin{equation}
\begin{aligned}
&(\Delta+\kappa_0^2-h^{-2} V_{0})Y=0, \ \text{in} \ \Omega, \\
&\frac{\partial Y}{\partial \nu}-\mathbf{S}_{\kappa_0}^{-1}\left[-\frac{1}{2}Id+\mathbf{K}_{\kappa_0} \right]Y=\mathbf{S}_{\kappa_0}^{-1}u^{I}, \ \text{on}\ \partial\Omega.
\end{aligned}
\label{key-one-variational-pb}
\end{equation}
Then we have
\begin{equation}
a(Y, Y)= \left\langle \mathbf{S}_{\kappa_0}^{-1}u^{I},Y \right\rangle_{-\frac{1}{2},\frac{1}{2}}\ ,
\label{key-one-variational-formulation}
\end{equation}
where
\begin{equation}
a(Y,Y):= \int_{\Omega} \vert \nabla Y \vert^2+\int_{\Omega} \left(-\kappa_0^{2} +h^{-2}V_{0} \right)Y \cdot \overline{Y}-\int_{\partial \Omega} BY \cdot \overline{Y},
\label{key-one-quadratic-form}
\end{equation}
with $BY:=\mathbf{S}_{\kappa_0}^{-1}\left[-\frac{1}{2}Id+\mathbf{K}_{\kappa_0} \right]Y, \ Y \in H^{1}(\Omega) $. Here $\mathbf{S}_{\kappa_0}$ and $ \mathbf{K}_{\kappa_0}$ stand for the single and double layer potentials at the frequency $\kappa_0$, see (\ref{layer-potentials}) for the explicit definitions.\\
We prove the following inequality
\begin{equation}
Re\left[-\int_{\partial \Omega} BY\cdot \overline{Y} \right]\geq -\epsilon \norm{Y}^{2}_{H^{1}(\Omega)}-C(\epsilon)\norm{Y}_{L^{2}(\Omega)}^{2},
\label{key-one}
\end{equation}
with which we deduce that there exists $h_0 <<1 $ such that for any $h<h_0 $, we have
\begin{equation}
Re \ a(Y,Y)\geq (1-\epsilon) \int_{\Omega} \vert \nabla Y \vert^{2}+\left(-\kappa_0^2+Ch^{-2}-\frac{5\epsilon}{4}-C(\epsilon) \right) \int_{\Omega} \vert Y \vert^2 \geq \tilde{C} \norm{Y}^{2}_{H^{1}(\Omega)}, 
\label{key-one-coercivity}
\end{equation}
where $\tilde{C} $ is a positive constant.\\

Based on (\ref{key-one-coercivity}) and (\ref{key-one-variational-formulation}), we deduce that $\Vert Y \Vert_{H^1(\Omega)}$ is uniformly bounded. Integrating by parts in \eqref{key-one-variational-pb} and using the estimate for $\norm{Y}_{H^{\frac{1}{2}}(\partial \Omega)}$, and the fact \[\frac{\partial Y}{\partial \nu}=\mathbf{S}_{\kappa_0}^{-1}\left[-\frac{1}{2}Id+\mathbf{K}_{\kappa_0} \right]Y+\mathbf{S}_{\kappa_0}^{-1}u^{I}, \ \text{on}\ \partial\Omega ,\] we obtain
\[\int_{\Omega} \vert \nabla Y\vert^2+\int_{\Omega} \left(h^{-2}V_{0}-\kappa_{0}^{2} \right) \vert Y\vert^2 =\int_{\partial \Omega} \frac{\partial Y}{\partial \nu} Y =\mathcal{O}(1), \]
whence it follows that $\norm{Y}_{L^2(\Omega)}=\mathcal{O}(h) $. Therefore, using interpolation, we have the estimate $\Vert Y \Vert_{H^r(\Omega)}=O(h^{1-r})$ or $\Vert Y\Vert_{H^r(\partial \Omega)}=O(h^{\frac{1}{2}-r})$. Hence $Y^{\infty}(\hat{x},\theta)-u^{\infty}_D(\hat{x},\theta) =o(1),\; h<<1$. Here $(\Delta +\kappa_0^2)u_D^s=0, \mbox{ in } \mathbb{R}^3\setminus{\overline{\Omega}} \mbox{ and } u_D^s=-u^i(\cdot, \theta) \mbox{ on } \partial \Omega $ with (S.R.C). Finally
 $u^{\infty}(\hat{x},\theta)-u^{\infty}_D(\hat{x},\theta) =o(1),\; a<<1$.

\subsection{The extreme cases $s>s^*$:  Metascreens}
We show the idea for $\Sigma$ as a closed surfaces. Recall that 
\begin{equation}\label{key-2-integral-equation}
Y +a^{(s^*-s)}\int_{\Sigma}\overline{\bf{C}}(K(z)+1)\Phi_{\kappa_0}(\cdot, z)Y(z)\; dz\; =u^{i}(\cdot, \theta).
\end{equation}
To characterize $\lim_{a \rightarrow 0} Y(\cdot,\theta)$ let us set, as for the volumetric case, $h:= a^{\frac{s-s^*}{2}}$, $\sigma:=\overline{\bf{C}}(K(z)+1)$.  As a first step, we observe that the scattering problem \eqref{scat-3a}-\eqref{scat-3d} can be transformed into the equivalent boundary value problem 
\begin{equation}
\begin{aligned}
&(\Delta +\kappa_{0}^2)Y =0,\ \text{in} \ B_{R}\setminus \Sigma, \\
&[Y]=0,\ \left[\frac{\partial Y}{\partial \nu} \right]- h^{-2} \sigma Y=0, \ \text{on} \ \Sigma, \\
&\left[\frac{\partial Y}{\partial \nu} \right]-TY= \frac{\partial u^{I}}{\partial \nu} -Tu^{I} ,\ \text{on} \ \partial B_{R},
\end{aligned}
\label{key-2-Scattering-problem}
\end{equation}
where $T: H^{\frac{1}{2}}(\partial B_R) \rightarrow H^{-\frac{1}{2}}(\partial B_R) $ is the Dirichlet to Neumann (D-N) map for the exterior problem on $\mathbb{R}^{3}\setminus B_R $.\\
We derive
\begin{equation}
\int_{B_R} \vert \nabla Y \vert^2 -\kappa_{0}^{2} \int_{B_R} \vert Y \vert^2 + h^{-2}\int_{\Sigma} \overline{\sigma}  \vert Y \vert^2 -\langle TY,Y \rangle_{-\frac{1}{2},\frac{1}{2}}=\left\langle  \frac{\partial u^{I}}{\partial \nu}-Tu^{I},Y \right\rangle_{-\frac{1}{2},\frac{1}{2}}.
\label{key-2-variational-formulation}
\end{equation}
Now using the fact that the operator $T$ can be decomposed into a coercive and a smoothing part, we obtain the inequality
\begin{equation}
\left\langle T Y, Y \right\rangle_{-\frac{1}{2},\frac{1}{2}}\geq C \Vert Y\Vert_{H^\frac{1}{2}(\partial B_R)}- \epsilon \Vert Y\Vert_{H^1(B_R)} -C(\epsilon)\Vert Y\Vert_{L^2(B_R)}.
\end{equation}
Plugging this in (\ref{key-2-variational-formulation}), we derive the estimate
\begin{equation}
\begin{aligned}
\left(1-\epsilon \right) \norm{\nabla Y}_{L^{2}(B_R)}^{2}&+ \left(-\kappa_{0}^2+Ch^{-2}-\epsilon-C(\epsilon) \right) \norm{Y}_{L^2(\Sigma)}^{2} + C \norm{Y}_{H^{\frac{1}{2}}(\partial B_R)}^{2} \\
& \leq C \norm{Y}_{H^{\frac{1}{2}}(\partial B_R)}+\mathcal{O}(1).
\end{aligned}
\label{key-2-inequality}
\end{equation}
From the integral equation (\ref{key-2-integral-equation}) and the invertibility properties of the single layer potentials, we derive the estimate 
\begin{equation}
\norm{Y}_{H^{\frac{1}{2}}(\partial B_R)} \leq C h^{-2} \norm{Y}_{L^{2}(\Sigma)} + O(1).
\end{equation}
With this estimate in (\ref{key-2-inequality}), we deduce that $\norm{Y}_{L^{2}(\Sigma)}=O(1)$.
Once more, using (\ref{key-2-inequality}), as for $h$ small enough the constants are positive, we derive that $\norm{Y}_{H^{\frac{1}{2}}(\partial B_R)}=O(1)$ and then the improved estimate
$\norm{Y}_{L^{2}(\Sigma)}=O(h)$.  
 Hence $Y^{\infty}(\hat{x},\theta)-u_D^{\infty}(\hat{x},\theta) =o(1),\; h<<1$. Here $(\Delta +\kappa_0^2)u_{D}^s=0, \mbox{ in } \mathbb{R}^3\setminus{\Sigma},\; u_D^s=-u^i(\cdot, \theta) \mbox{ on } \Sigma$ with (S.R.C). Finally $u^{\infty}(\hat{x},\theta)-u^{\infty}_D(\hat{x},\theta) =o(1),\; a<<1$.

\bigskip
~~~
\bigskip

\section{The case of metamaterials}\label{Section-volume}
In this section, we deal with the case of volumetric distribution of the gas bubbles. In this direction, we first study the volume integral equations corresponding to the equivalent scattering problem and then compare this equivalent problem to the original scattering problem.
\subsection{The volume integral equation and corresponding estimates}
Let us consider the scattering problem
\begin{equation}
\begin{aligned}
&\left(\Delta +\kappa_{0}^{2}-h_{*}V_{0} \right)u^{t}_{a}=0,\ \text{in} \ \mathbb{R}^{3},
\end{aligned}
\label{scat-1a}
\end{equation}
\begin{equation}
\begin{aligned}
&u^{t}_{a}=u^{s}_{a}+e^{i \kappa_{0} x \cdot \theta},
\end{aligned}
\label{scat-1b}
\end{equation}
\begin{equation}
\begin{aligned}
&\frac{\partial u^{s}_{a}}{\partial \vert x \vert}-i \kappa_{0} u^{s}_{a} =o\left(\frac{1}{\vert x \vert} \right), \ \vert x \vert\rightarrow \infty,
\end{aligned}
\label{scat-1c}
\end{equation}
and the corresponding Lippmann-Schwinger equation
\begin{equation}
\begin{aligned}
Y(z)+h_{*} \int_{\Omega} \Phi_{\kappa_0}(z,y)V_{0}(y) Y(y) dy=u^{I}(z),\ z \in \mathbb{R}^{3},
\end{aligned}
\label{vol-int}
\end{equation}
where $h_{*} $ is a positive real number and $u^{I}(z)=e^{i\kappa_{0}z \cdot \theta} $. Note that $V_{0}=K^{M}\overline{\mathbf{C}} $ in our case, where $K^M:=(K+1) $.\\
It is easy to see that $u^{t}_{a} $ is a solution of \eqref{scat-1a}-\eqref{scat-1c} if and only if 
\begin{equation}
\begin{aligned}
u_{a}^{t}&=
\begin{cases}
Y, &\ \text{in}\ \Omega,\\
u^{I}-h_{*} \int_{\Omega} \Phi_{\kappa_0}(z,y)V_{0}(y) Y(y) dy, &\ \text{in} \ \mathbb{R}^{3}\setminus \overline{\Omega},
\end{cases}
\end{aligned}
\notag
\end{equation} 
where $Y$ satisfies \eqref{vol-int}.\\
In what follows, we are interested in the two cases (in the limit as $a\rightarrow 0$), stated below.
\begin{itemize}
\item $h_{*}=\mathcal{O}(1) $, 
\item $h_{*}=a^{1-h_1-s}, s+h_1>1  $.
\end{itemize}
When $h_{*}=\mathcal{O}(1)\ \text{as}\ a\rightarrow 0 $, the existence and uniqueness of solution $Y$ to the integral equation \eqref{vol-int} follows by a standard argument based on Fredholm alternative (see \cite{ACKS}). A similar argument also suffices in the case $h_{*}=a^{1-h_1-s} $ since $a^{1-h_1-s}V_{0} $ is real-valued and the unique solution belongs to the class $H^{2}_{loc}(\mathbb{R}^{3}) $ (see \cite{CMS}, \cite{CK}).\\
We shall next establish asymptotic estimates for the solution $Y$ and its gradient $ \nabla Y$ in suitable norms.\\
We recall that the single layer potential $\mathbf{S}_{\kappa_0} $, double layer potential $\mathbf{K}_{\kappa_0} $ and the adjoint $\mathbf{K}_{\kappa_0}^{*} $ of the double layer potential are defined as
\begin{equation}
\begin{aligned}
&\mathbf{S}_{\kappa_0}\phi(x):=\int_{\partial \Omega} \Phi_{\kappa_{0}}(x,y) \phi(y) dy, \\
&\mathbf{K}_{\kappa_0}\phi(x):=\int_{\partial \Omega} \frac{\partial \Phi_{\kappa_{0}}(x,y)}{\partial \nu(y)} \phi(y) dy,\\
&\mathbf{K}_{\kappa_0}^{*}\phi(x):=\int_{\partial \Omega}  \frac{\partial \Phi_{\kappa_{0}}(x,y)}{\partial \nu(x)} \phi(y) dy,
\end{aligned}
\label{layer-potentials}
\end{equation}
where for $x,y \in \mathbb{R}^3$,\[\Phi_{\kappa_0}(x,y):= \frac{e^{i \kappa_{0} \vert x-y \vert}}{4\pi \vert x-y \vert} \] denotes the fundamental solution of the Helmholtz equation in three dimensions with a fixed wave number $\kappa_{0} $. In the case $\kappa_{0}=0 $, we shall denote the corresponding operators by $\mathbf{S}_{0}, \mathbf{K}_{0} $ and $\mathbf{K}_{0}^{*} $.\\
Further the single and double layer potentials satisfy the mapping properties (see \cite{St})
\begin{equation}
\begin{aligned}
\mathbf{S}_{\kappa_0}: H^{s-1}(\partial \Omega) &\rightarrow H^{s}(\partial \Omega), \ 0\leq s\leq 1,  \\
\frac{1}{2}Id-\mathbf{K}_{\kappa_0}^{*}: H^{s}(\partial \Omega) &\rightarrow H^{s}(\partial \Omega), \  -1\leq s\leq 0, \\
%&{\color{red}{S_{\kappa_0}-S_{0}:  H^{-s}(\partial \Omega) \rightarrow H^{1}(\partial \Omega), \ 0\leq s\leq 1,}}\\
%&{\color{red}{K_{0}^{*}-K_{\kappa_0}^{*}:H^{-s}(\partial \Omega) \rightarrow H^{\frac{1}{2}}(\partial \Omega), 0\leq s\leq 1}}
\end{aligned}
\label{e107}
\end{equation}
and the single layer potential $\mathbf{S}_{\kappa_0} $ is invertible provided $\kappa_{0}^{2} $ is not an eigenvalue for the Dirichlet Laplacian. \\
We begin by taking note of the following result on the smoothing properties of the operators $\mathbf{S}_{\kappa_0}-\mathbf{S}_{0} $ and $\mathbf{K}_{\kappa_0}^{*}-\mathbf{K}_{0}^{*} $.
\begin{lemma}\label{map-prop}
The following mapping properties hold true.
\begin{equation}
\begin{aligned}
&\mathbf{S}_{\kappa_0}-\mathbf{S}_{0}:  H^{-\frac{1}{2}}(\partial \Omega) \rightarrow H^{\frac{5}{2}}(\partial \Omega), \\
&\mathbf{K}_{\kappa_0}^{*}-\mathbf{K}_{0}^{*}:H^{-\frac{1}{2}}(\partial \Omega) \rightarrow H^{\frac{3}{2}}(\partial \Omega). 
\end{aligned}
\notag
\end{equation}
\end{lemma}
\begin{proof}
The mapping property for $\mathbf{S}_{\kappa_0}-\mathbf{S}_0 $ is given in Section 6.9, \cite{St}. To establish the mapping properties of $\mathbf{K}_{\kappa_0}^{*}-\mathbf{K}_{0}^{*} $ we proceed as follows. \\
Let $u \in H^{-\frac{1}{2}}(\partial \Omega) $ and denote $W_{\kappa_0}:=\mathbf{S}_{\kappa_0}u,\ W_{0}:=\mathbf{S}_{0}u $. Then using the mapping properties of the single layer potential $\mathbf{S}_{\kappa_0} $ and the operator $\mathbf{S}_{\kappa_0}-\mathbf{S}_0 $, it follows that $W_{\kappa_0}-W_0$ satisfies 
\begin{equation}
\begin{aligned}
&\Delta (W_{\kappa_0}-W_0)=-\kappa_{0}^{2} \mathbf{S}_{\kappa_0}u \in H^{1}(\Omega),\\
&W_{\kappa_0}-W_{0}\Big\vert_{\partial \Omega}\in H^{\frac{5}{2}}(\partial \Omega).
\end{aligned}
\notag
\end{equation}
Using elliptic regularity, this implies that $W_{\kappa_0}-W_0 \in H^{3}(\Omega). $ This further implies that
\begin{equation}
(\mathbf{K}_{\kappa_0}^{*}-\mathbf{K}_{0}^{*})u= \frac{\partial(W_{\kappa_0}-W_0)}{\partial \nu}\Big\vert_{-} \in H^{\frac{3}{2}}(\partial \Omega), 
\notag
\end{equation}
whence it follows that $\mathbf{K}_{\kappa_0}^{*}-\mathbf{K}_{0}^{*}: H^{-\frac{1}{2}}(\partial \Omega) \rightarrow H^{\frac{3}{2}}(\partial \Omega) $.
\end{proof}
The next result provides us with an estimate of the trace of the total field $u_{a}^t $ on the boundary $ \partial \Omega$ in terms of the parameter $h_{*}=a^{1-s-h_1}$ in the regime $s+h_1>1 $.
\begin{theorem}\label{vol-blow}
Assume that $\kappa_{0}^{2} $ is not an eigenvalue for the Dirichlet laplacian in $\Omega $. 
Then for sufficiently small $a $, the total field corresponding to the scattering problem \eqref{scat-1a}-\eqref{scat-1c} satisfies the estimate
\begin{equation}
\begin{aligned}
&\norm{u^t_a}_{H^{\alpha}(\Omega)}=\mathcal{O}\left(a^{\frac{s+h_1-1}{2}(1-\alpha)}\right),\ \alpha \in [0,1].
%&\norm{u^{t}_{a} \Big\vert_{\partial \Omega}}_{H^{\alpha}(\partial \Omega)} \leq C a^{\frac{s+h_1-1}{2}(\frac{1}{2}-\alpha)} ,\ \alpha \in \left[0,\frac{1}{2}\right],
\end{aligned}
\notag
\end{equation}
\end{theorem}
\begin{proof}
Let us define the semi-classical parameter $h:=a^{\frac{s+h_1-1}{2}} $. Since $s+h_1>1 $ and $a$ is small, it follows that the parameter $h <<1 $. \\
Our first step is to transform the scattering problem into an equivalent boundary-value problem posed in $ \Omega$. To do so, we recall that using the Green's formula, we can write 
\begin{equation}
\int_{\Omega} u_{a}^{t} \Delta \Phi_{\kappa_0}(\cdot, z) -\Phi_{\kappa_0}(\cdot, z) \Delta u_{a}^{t} = \int_{\partial \Omega} u_{a}^{t} \frac{\partial \Phi_{\kappa_0}}{\partial \nu}(\cdot,z)- \Phi_{\kappa_0}(\cdot,z)\frac{\partial u_{a}^t}{\partial \nu}, \ z\in \Omega^{c}.
\label{e100}
\end{equation}
Now, we recall that \[\Delta \Phi_{\kappa_0}+\kappa_{0}^{2} \Phi_{\kappa_0}=-\delta, \ \text{in}\ \mathbb{R}^{3}, \]
and from the Lippmann-Schwinger equation \eqref{vol-int}, it follows that
\begin{equation}
(\Delta+\kappa_{0}^2-h^{-2} V_{0})u_{a}^{t}=0, \ \text{in} \ \mathbb{R}^3.
\label{e101}
\end{equation}
Using these relations in \eqref{e100}, we have
\begin{equation}
\begin{aligned}
\int_{\partial \Omega} u_{a}^{t} \frac{\partial \Phi_{\kappa_0}}{\partial \nu}(\cdot,z)- \Phi_{\kappa_0}(\cdot,z)\frac{\partial u_{a}^t}{\partial \nu}&=-\kappa_{0}^{2} \int_{\Omega} u_{a}^{t} \Phi_{\kappa_0}(\cdot,z)-\int_{\Omega} \Phi_{\kappa_0}(\cdot,z) \left(-\kappa_{0}^{2} +h^{-2} V_{0} \right) u_{a}^{t}\\
&=-\int_{\Omega} h^{-2} V_{0} u_{a}^{t} \Phi_{\kappa_0}(\cdot,z)=u_{a}^{t}(z)-u^{I}(z).
\end{aligned}
\label{e102}
\end{equation}
Next we take the trace on $\partial \Omega $, for any $z \in \partial \Omega $, and deduce that
\begin{equation}
\begin{aligned}
&u_{a}^{t}(z)-u^{I}(z)= -\int_{\partial \Omega} \Phi_{\kappa_0}(\cdot,z) \frac{\partial u_{a}^t}{\partial \nu}+\int_{\partial \Omega} u_{a}^{t} \frac{\partial \Phi_{\kappa_0}}{\partial \nu}(\cdot,z)+\frac{1}{2} u_{a}^{t}(z)\\
\Rightarrow &\int_{\partial \Omega} \Phi_{\kappa_0}(\cdot,z) \frac{\partial u_{a}^t}{\partial \nu}+\frac{1}{2} u_{a}^{t}(z)-\int_{\partial \Omega} u_{a}^{t} \frac{\partial \Phi_{\kappa_0}}{\partial \nu}(\cdot,z)=u^{I}(z)\\
\Rightarrow &\int_{\partial \Omega} \Phi_{\kappa_0}(\cdot,z) \frac{\partial u_{a}^t}{\partial \nu}-\left[-\frac{1}{2}Id+\mathbf{K}_{\kappa_0} \right]u_{a}^{t}=u^{I}\\
\Rightarrow &\frac{\partial u_{a}^t}{\partial \nu}-\mathbf{S}_{\kappa_0}^{-1}\left[-\frac{1}{2}Id+\mathbf{K}_{\kappa_0} \right]u_{a}^{t}=\mathbf{S}_{\kappa_0}^{-1}u^{I}. 
\end{aligned}
\label{e103}
\end{equation}
Therefore $u_{a}^t$ satisfies the following boundary-value problem
\begin{equation}
\begin{aligned}
&(\Delta+\kappa_0^2-h^{-2} V_{0})u_{a}^{t}=0, \ \text{in} \ \Omega, \\
&\frac{\partial u_{a}^t}{\partial \nu}-\mathbf{S}_{\kappa_0}^{-1}\left[-\frac{1}{2}Id+\mathbf{K}_{\kappa_0} \right]u_{a}^{t}=\mathbf{S}_{\kappa_0}^{-1}u^{I}, \ \text{on}\ \partial\Omega.
\end{aligned}
\label{e104}
\end{equation}
The variational formulation of the problem \eqref{e104} can be written as follows: \\
Find a unique $u \in H^{1}(\Omega) $ such that 
\begin{equation}
a(u,v)= \left\langle \mathbf{S}_{\kappa_0}^{-1}u^{I},v \right\rangle_{-\frac{1}{2},\frac{1}{2}}\ \forall v \in H^{1}(\Omega),
\notag
\end{equation}
where
\begin{equation}
a(u,v):= \int_{\Omega} \nabla u \cdot \nabla \overline{v}+\int_{\Omega} \left(-\kappa_0^{2} +h^{-2}V_{0} \right)u \cdot \overline{v}-\int_{\partial \Omega} Bu \cdot \overline{v},
\label{e105}
\end{equation}
with $Bu:=\mathbf{S}_{\kappa_0}^{-1}\left[-\frac{1}{2}Id+\mathbf{K}_{\kappa_0} \right]u, \ u \in H^{1}(\Omega) $. \\
To deal with the term $-\int_{\partial \Omega} Bu\cdot \overline{u} $, we split it as follows:
\begin{equation}
\begin{aligned}
-\int_{\partial \Omega} Bu\cdot \overline{u}&=\int_{\partial \Omega} \mathbf{S}_{\kappa_0}^{-1}\left[\frac{1}{2}Id-\mathbf{K}_{\kappa_0} \right]u \cdot \overline{u}\\
&=\int_{\partial \Omega} \overline{\mathbf{S}_{\kappa_0}\left(\mathbf{S}_{\kappa_0}^{-1}u \right)} \left[\frac{1}{2}Id-\mathbf{K}_{\kappa_0}^{*} \right]\mathbf{S}_{\kappa_0}^{-1} u \ \left(\text{since}\ \mathbf{K}_{\kappa_0}\mathbf{S}_{\kappa_0}=\mathbf{S}_{\kappa_0}\mathbf{K}_{\kappa_0}^{*}\right)\\
&=\int_{\partial \Omega} \overline{\mathbf{S}_{0}\left(\mathbf{S}_{\kappa_0}^{-1} u \right)} \left[\frac{1}{2}Id-\mathbf{K}_{\kappa_0}^{*} \right]\mathbf{S}_{\kappa_0}^{-1} u+\int_{\partial \Omega} \overline{\left(\mathbf{S}_{\kappa_0}-\mathbf{S}_0 \right) \left(\mathbf{S}_{\kappa_0}^{-1} u \right)} \left[\frac{1}{2}Id-\mathbf{K}_{\kappa_0}^{*} \right]\mathbf{S}_{\kappa_0}^{-1} u  \\
&=\int_{\partial \Omega} \mathbf{S}_{0}\overline{\left(\mathbf{S}_{\kappa_0}^{-1} u \right)} \left[\frac{1}{2}Id-\mathbf{K}_{0}^{*} \right]\mathbf{S}_{\kappa_0}^{-1} u+
\int_{\partial \Omega} \mathbf{S}_{0}\overline{\left(\mathbf{S}_{\kappa_0}^{-1} u \right)} \left[\mathbf{K}^{*}_{0}-\mathbf{K}_{\kappa_0}^{*} \right]\mathbf{S}_{\kappa_0}^{-1} u \\
&\quad +\int_{\partial \Omega} \overline{\left(\mathbf{S}_{\kappa_0}-\mathbf{S}_0 \right) \left(\mathbf{S}_{\kappa_0}^{-1} u \right)} \left[\frac{1}{2}Id-\mathbf{K}_{\kappa_0}^{*} \right]\mathbf{S}_{\kappa_0}^{-1} u. 
\end{aligned}
\label{e106}
\end{equation}
Using the properties \eqref{e107} and lemma \ref{map-prop}, we deal with the last two terms in \eqref{e106} in the following manner.\\
For $u \in H^{\frac{1}{2}}(\partial \Omega) $, choosing $s\in \left(\frac{1}{2},1 \right) $, we can write
\begin{equation}
\begin{aligned}
\left\vert \int_{\partial \Omega} \overline{\left(\mathbf{S}_{\kappa_0}-\mathbf{S}_0 \right) \left(\mathbf{S}_{\kappa_0}^{-1} u \right)} \left[\frac{1}{2}Id-\mathbf{K}_{\kappa_0}^{*} \right]\mathbf{S}_{\kappa_0}^{-1} u \right\vert &\lesssim \norm{\overline{\left(\mathbf{S}_{\kappa_0}-\mathbf{S}_0 \right) \left(\mathbf{S}_{\kappa_0}^{-1} u \right)}}_{H^{s}(\partial \Omega)} \cdot \norm{ [\frac{1}{2}Id-\mathbf{K}_{\kappa_0}^{*} ]\mathbf{S}_{\kappa_0}^{-1} u }_{H^{-s}(\partial \Omega)}\\
&\lesssim \norm{\overline{\mathbf{S}_{\kappa_0}^{-1} u} }_{H^{-\frac{1}{2}}(\partial \Omega)} \cdot \norm{ \mathbf{S}_{\kappa_0}^{-1} u }_{H^{-s}(\partial \Omega)}\\
&\lesssim \norm{ u }_{H^{\frac{1}{2}}(\partial \Omega)} \cdot \norm{ u }_{H^{-s+1}(\partial \Omega)}\\
&\lesssim \norm{ u }_{H^{1}(\Omega)} \cdot \norm{ u }_{H^{-s+\frac{3}{2}}( \Omega)}\\
&\lesssim \epsilon \norm{ u }^{2}_{H^{1}(\Omega)} +\frac{1}{4 \epsilon} \norm{ u }^{2}_{H^{-s+\frac{3}{2}}( \Omega)}\\
&\lesssim \epsilon \norm{ u }^{2}_{H^{1}(\Omega)}+\frac{1}{4\epsilon} \left[ \norm{ u }^{2(\frac{3}{2}-s)}_{H^{1}( \Omega)} \cdot \norm{ u }^{2(1+s-\frac{3}{2})}_{L^{2}( \Omega)} \right] \\
&\lesssim \frac{5\epsilon}{4} \norm{ u }^{2}_{H^{1}(\Omega)}+ C(\epsilon)\norm{ u }^{2}_{L^{2}(\Omega)},
\end{aligned}
\label{e108}
\end{equation}
where $\epsilon $ is sufficiently small enough and $C(\epsilon) $ depends only on $\epsilon $ and $s$.\\
Similarly since $u \in H^{\frac{1}{2}}(\partial \Omega) $, we have
\begin{equation}
\begin{aligned}
\left\vert \int_{\partial \Omega} \mathbf{S}_{0}\overline{\left(\mathbf{S}_{\kappa_0}^{-1} u \right)} \left[\mathbf{K}^{*}_{0}-\mathbf{K}_{\kappa_0}^{*} \right]\mathbf{S}_{\kappa_0}^{-1} u \right\vert &\leq \norm{ \mathbf{S}_{0} \overline{\left(\mathbf{S}_{\kappa_0}^{-1} u  \right)}}_{H^{-s}(\partial \Omega)} \cdot \norm{\left[\mathbf{K}^{*}_{0}-\mathbf{K}_{\kappa_0}^{*} \right]\mathbf{S}_{\kappa_0}^{-1} u}_{H^{s}(\partial \Omega)}, \ s\in (0,\frac{1}{2})\\
&\lesssim \norm{ \mathbf{S}_{0} \overline{\left(\mathbf{S}_{\kappa_0}^{-1} u \right)}}_{H^{s}(\partial \Omega)} \cdot \norm{\left[\mathbf{K}^{*}_{0}-\mathbf{K}_{\kappa_0}^{*} \right]\mathbf{S}_{\kappa_0}^{-1} u}_{H^{\frac{1}{2}}(\partial \Omega)}\\
&\lesssim \norm{u}_{H^{s}(\partial \Omega)} \cdot \norm{\mathbf{S}_{\kappa_0}^{-1}u}_{H^{-\frac{1}{2}}(\partial \Omega)} \\
&\lesssim \norm{u}_{H^{s}(\partial \Omega)} \cdot \norm{u}_{H^{\frac{1}{2}}(\partial \Omega)} \\
&\lesssim \norm{u}_{H^{s+\frac{1}{2}}(\Omega)} \cdot \norm{u}_{H^{1}(\Omega)} \\
&\lesssim \epsilon \norm{u}^{2}_{H^{1}(\Omega)}+\frac{1}{4 \epsilon} \norm{u}^{2}_{H^{s+\frac{1}{2}}(\Omega)}\\
&\lesssim \epsilon \norm{ u }^{2}_{H^{1}(\Omega)}+\frac{1}{4\epsilon} \left[ \norm{ u }^{2(s+\frac{1}{2})}_{H^{1}( \Omega)} \cdot \norm{ u }^{2(1-s-\frac{1}{2})}_{L^{2}( \Omega)} \right] \\
&\lesssim \frac{5\epsilon}{4} \norm{ u }^{2}_{H^{1}(\Omega)}+ C(\epsilon)\norm{ u }^{2}_{L^{2}(\Omega)}.
\end{aligned}
\label{e109}
\end{equation}
For the first term of \eqref{e106}, we use the fact that $\frac{1}{2}Id-\mathbf{K}_{0}^{*} $ is positive definite on $H^{-\frac{1}{2}}(\partial \Omega) $ equipped with the scalar product $\langle \mathbf{S_{0}}u,v\rangle $, to deduce that
\begin{equation}
\begin{aligned}
&Re \int_{\partial \Omega} \mathbf{S}_{0}\left(\overline{\mathbf{S}_{\kappa_0}^{-1} u} \right) \left[\frac{1}{2}Id-\mathbf{K}_{0}^{*} \right]\mathbf{S}_{\kappa_0}^{-1} u \\
&\qquad=\int_{\partial \Omega} \mathbf{S}_{0}\left(Re [\mathbf{S}_{\kappa_0}^{-1} u] \right) \left[\frac{1}{2}Id-\mathbf{K}_{0}^{*} \right]Re[\mathbf{S}_{\kappa_0}^{-1} u]+ \int_{\partial \Omega} \mathbf{S}_{0}\left(Im[\mathbf{S}_{\kappa_0}^{-1} u] \right) \left[\frac{1}{2}Id-\mathbf{K}_{0}^{*} \right]Im[\mathbf{S}_{\kappa_0}^{-1} u] \\
&\qquad \geq C \left[\norm{Re[\mathbf{S}_{\kappa_0}^{-1} u]}^{2}_{H^{\frac{1}{2}}(\partial \Omega)}+\norm{Im[\mathbf{S}_{\kappa_0}^{-1} u]}^{2}_{H^{\frac{1}{2}}(\partial \Omega)} \right]\\
&\qquad \geq C \norm{\mathbf{S}_{\kappa_0}^{-1}u}_{H^{-\frac{1}{2}}(\partial \Omega)}^{2} \geq C \norm{u}_{H^{\frac{1}{2}}(\partial \Omega)}^{2}.
\end{aligned}
\label{e110}
\end{equation}
Using \eqref{e108}-\eqref{e110} in \eqref{e106}, we can write
\begin{equation}
\begin{aligned}
Re\left[-\int_{\partial \Omega} Bu\cdot \overline{u} \right]&\geq C \norm{u}^{2}_{H^{\frac{1}{2}}(\partial \Omega)}-\frac{5\epsilon}{4} \norm{u}^{2}_{H^{1}(\Omega)}-C(\epsilon)\norm{u}_{L^{2}(\Omega)}^{2} \geq -\frac{5\epsilon}{4} \norm{u}^{2}_{H^{1}(\Omega)}-C(\epsilon)\norm{u}_{L^{2}(\Omega)}^{2}.
\end{aligned}
\label{e111}
\end{equation}
Using this in \eqref{e105}, we see that there exists $h_0 <<1 $ such that for any $h<h_0 $, we have
\begin{equation}
\begin{aligned}
Re \ a(u,u)&\geq \int_{\Omega} \vert \nabla u \vert^{2}+\int_{\Omega} \left(-\kappa_0^2+h^{-2}V_{0} \right)\vert u\vert^{2}-\frac{5\epsilon}{4} \norm{u}^{2}_{H^{1}(\Omega)}-C(\epsilon)\norm{u}_{L^{2}(\Omega)}^{2}\\
&\geq (1-\frac{5\epsilon}{4}) \int_{\Omega} \vert \nabla u \vert^{2}+\left(-\kappa_0^2+Ch^{-2}-\frac{5\epsilon}{4}-C(\epsilon) \right) \int_{\Omega} \vert u \vert^2 \geq \tilde{C} \norm{u}^{2}_{H^{1}(\Omega)}, 
\end{aligned}
\label{e112}
\end{equation}
where $\tilde{C} $ is a positive constant.\\
Therefore using the Lax-Milgram lemma, we infer that the boundary value problem has a unique weak solution which is $u_{a}^t$ and it satisfies the estimate
\begin{equation}
\norm{u_{a}^t}_{H^{1}(\Omega)} \lesssim \norm{u^I}_{H^{1}(\Omega)}.
\label{e113}
\end{equation}
The estimate \eqref{e113} immediately implies that the trace satisfies $\norm{u_{a}^t}_{H^{\frac{1}{2}}(\partial \Omega)}=\mathcal{O}(1) $.\\
Now integrating by parts in \eqref{e104} and using the estimate for $\norm{u_{a}^t}_{H^{\frac{1}{2}}(\partial \Omega)}$, and the fact \[\frac{\partial u_{a}^t}{\partial \nu}=\mathbf{S}_{\kappa_0}^{-1}\left[-\frac{1}{2}Id+\mathbf{K}_{\kappa_0} \right]u_{a}^{t}+\mathbf{S}_{\kappa_0}^{-1}u^{I}, \ \text{on}\ \partial\Omega ,\] we observe that
\[\int_{\Omega} \vert \nabla u^{t}_{a}\vert^2+\int_{\Omega} \left(h^{-2}V_{0}-\kappa_{0}^{2} \right) \vert u^{t}_{a}\vert^2 =\int_{\partial \Omega} \frac{\partial u^{t}_{a}}{\partial \nu} u^{t}_{a} =\mathcal{O}(1), \]
whence it follows that $\norm{u^t_a}_{L^2(\Omega)}=\mathcal{O}(h) $. Therefore, using interpolation, we have the estimate
\[\norm{u^t_a}_{H^{\alpha}(\Omega)}=\mathcal{O}(h^{1-\alpha}),\ \alpha \in [0,1]. \]
\end{proof}
\begin{itemize}
\item Using theorem \ref{vol-blow}, we can compare the far-field $u^{\infty}_{a} $ corresponding to the scattering problem \eqref{scat-1a}-\eqref{scat-1c} to the far-field $u^{\infty}_{D} $ corresponding to the Dirichlet scattering problem
\begin{equation}
\left(\Delta+\kappa_{0}^{2} \right)u^{t}_{D}=0, \ \text{in}\ \mathbb{R}^{3}\setminus \overline{\Omega},
\label{scat-2a}
\end{equation}
\begin{equation}
u^{t}_{D}:=u^{s}_{D}+e^{i\kappa_{0} x \cdot \theta}=0,\ \text{on}\ \partial \Omega,
\label{scat-2b}
\end{equation}
\begin{equation}
\frac{\partial u^{s}_{D}}{\partial \nu}-i\kappa_0 u^{s}_{D}=o\left(\frac{1}{\vert x \vert} \right),\ \vert x \vert \rightarrow \infty.
\label{scat-2c}
\end{equation}
To see this, we observe that from theorem \ref{vol-blow}, using trace order estimates, we can write 
\[\norm{u^t_a}_{H^{\alpha}(\partial \Omega)}=\mathcal{O}\left(a^{\frac{s+h_1-1}{2} \left(\frac{1}{2}-\alpha\right)}\right),\ \alpha \in [0,\frac{1}{2}], \] whence it follows that $\norm{u^{t}_{a} \Big\vert_{\partial \Omega}}_{L^{2}(\partial \Omega)}=\norm{u^{s}_{a}(x)+e^{i \kappa_{0} x\cdot\theta} \Big\vert_{\partial \Omega}}_{L^{2}(\partial \Omega)} =\mathcal{O}\left(a^{\frac{s+h_1-1}{4}}\right) $.\\
Now since $u^{s}_{a} $ satisfies 
$\left(\Delta+\kappa_{0}^{2} \right)u^{s}_{a}=0, \ \text{in}\ \mathbb{R}^{3}\setminus \overline{\Omega} $ together with the radiation conditions, and boundary values satisfying $\norm{u^{s}_{a}(x)+e^{i \kappa_{0} x\cdot\theta} \Big\vert_{\partial \Omega}}_{L^{2}(\partial \Omega)} =\mathcal{O}(a^{\frac{s+h_1-1}{4}}) $, 
the well-posedness of the forward scattering problem in the exterior domain $\mathbb{R}^{3}\setminus \overline{\Omega} $ implies that the corresponding far-fields satisfy the estimates
\begin{equation}
u^{\infty}_{a}(\hat{x},\theta)-u^{\infty}_{D}(\hat{x},\theta) =\mathcal{O}(a^{\frac{s+h_1-1}{4}}).
\label{vol-est-1000}
\end{equation}
\end{itemize}
Next we prove estimates for the $L^{\infty}$ norms of $Y$ and $\nabla Y$ which shall be used to prove the required asymptotic approximations.
\begin{proposition}\label{blow-est-vol}
The solution $Y$ to the volume integral equation \eqref{vol-int} satisfies the following estimates.\\
\begin{itemize}
\item If $h_{*}=\mathcal{O}(1),\ a \rightarrow 0 $, then{\footnote{In this case we do not need the condition that $\kappa_{0}^{2} $ is not an eigenvalue for the Dirichlet Laplacian in $\Omega $.}} 
\begin{equation}
\norm{Y}_{L^{\infty}(\Omega)}=\mathcal{O}(1),\ \norm{\nabla Y}_{L^{\infty}(\Omega)}=\mathcal{O}(1).
\label{vol-est-1}
\end{equation}
\item If $h_{*}=a^{1-h_1-s}$, $s+h_1>1 $, and $\kappa_{0}^{2} $ is not an eigenvalue for the Dirichlet laplacian in $\Omega $, then
\begin{equation}
\norm{Y}_{L^{\infty}(\Omega)}=\mathcal{O}(a^{\frac{1-s-h_1}{2}}),\ \norm{\nabla Y}_{L^{\infty}(\Omega)}=\mathcal{O}(a^{\frac{1-s-h_{1}}{2} (1+\alpha)}),\ a \rightarrow 0, 
\label{vol-est-2}
\end{equation}
where $\alpha \in (\frac{1}{2},1] $. 
\end{itemize}
\end{proposition}
\begin{proof}
The above estimates can be proved following the arguments in \cite{ACKS, CMS}. For the sake of completion, we provide an outline of the proofs here.\\
\begin{itemize} 
\item[(a)] Let us first consider the case $h_{*}=\mathcal{O}(1), \ a \rightarrow 0 $. 
In this case, the invertibility of the integral equation \eqref{vol-int} from $L^{2}(\Omega) $ to $L^{2}(\Omega) $ immediately implies that $\norm{Y}_{L^2(\Omega)}\leq C \norm{u^I}_{L^2(\Omega)} $. Also, from the integral equation \eqref{vol-int}, using the mapping property of the single-layer potential, we can write
\begin{equation}
\begin{aligned}
\norm{Y}_{H^2(\Omega)}&\leq \norm{\mathbf{S}_{\kappa_0}\left[V_{0} Y\right]}_{H^2(\Omega)}+ \norm{u^I}_{H^2(\Omega)}\\
&\leq C \norm{V_{0} Y}_{L^{2}(\Omega)}+\norm{u^I}_{H^2(\Omega)}\leq C \norm{u^I}_{H^2(\Omega)}=\mathcal{O}(1).
\end{aligned}
\notag
\end{equation} 
Using the Sobolev embedding $H^{2}(\Omega) \subset L^{\infty}(\Omega) $, we can now conclude that $\norm{Y}_{L^{\infty}(\Omega)}=\mathcal{O}(1) $. The estimate $\norm{\nabla Y}_{L^{\infty}(\Omega)}=\mathcal{O}(1) $ follows similarly using the $W^{2,p}$ regularity of $Y$ (see \cite{ACKS}).\\
\item[(b)] Next let $h_{*}=a^{1-s-h_1} $. Then from theorem \ref{vol-blow}, it follows that $Y$ satisfies the estimate \[\norm{Y}_{H^{\alpha}(\Omega)}=\mathcal{O}\left( a^{\frac{s+h_1-1}{2}(1-\alpha)}\right), \ \alpha \in [0,1] .\] In particular, for $\alpha=0 $, we arrive at the estimate $\norm{Y}_{L^{2}(\Omega)}=\mathcal{O}\left( a^{\frac{s+h_1-1}{2}}\right) $. \\
Using this in the volume integral equation \eqref{vol-int}, this gives
\begin{equation}
\vert Y(z) \vert = \mathcal{O}\left(1 \right)+ \mathcal{O}\left(a^{1-s-h_1} \right) \mathcal{O}\left(a^{\frac{s+h_1-1}{2}} \right)= \mathcal{O}\left(a^{\frac{1-s-h_1}{2}} \right),
\notag
\end{equation}
since $s+h_1 >1 $,
and therefore $\norm{Y}_{L^{\infty}(\Omega)}=\mathcal{O}\left(a^{\frac{1-s-h_1}{2}}\right) $.\\
Again from the volume integral equation \eqref{vol-int}, we can write
\begin{equation}
\left\vert \nabla Y(z) \right\vert=\mathcal{O}(1)+a^{1-h_1-s} \norm{V_{0}(\cdot) \nabla_{z}\Phi_{\kappa_0}(z,\cdot)}_{L^{p}(\Omega)} \norm{Y}_{L^{p'}(\Omega)}.
\label{vol-est-3}
\end{equation}
Since $\norm{\nabla_{z}\Phi_{\kappa_0}(z,\cdot)}_{L^{p}(\Omega)}<\infty $ for $p<\frac{3}{2} $, we need to estimate $\norm{Y}_{L^{p'}(\Omega)} $ for $p'>3 $.\\
By Sobolev embeddings, we have $\norm{Y}_{L^{p'}(\Omega)} \leq C \norm{Y}_{H^{\alpha}(\Omega)}, \ \frac{1}{p'}=\frac{1}{2}-\frac{\alpha}{3} $ where $\alpha>\frac{1}{2} $. Therefore 
$\norm{Y}_{L^{p'}(\Omega)}=\mathcal{O}\left(a^{\frac{s+h_1-1}{2}(1-\alpha)}\right), \ \alpha>\frac{1}{2} $. Using this in \eqref{vol-est-3}, we deduce that
\begin{equation}
\vert \nabla Y(z) \vert = \mathcal{O}\left(1 \right)+ \mathcal{O}\left(a^{1-s-h_1} \right) \mathcal{O}\left(a^{\frac{s+h_1-1}{2}(1-\alpha)} \right)= \mathcal{O}\left(a^{\frac{1-s-h_1}{2}}(1+\alpha) \right),
\notag
\end{equation}
and hence $\norm{\nabla Y}_{L^{\infty}(\Omega)}=\mathcal{O}\left(a^{\frac{1-s-h_{1}}{2} (1+\alpha)}\right), \ \alpha >\frac{1}{2} $.
\end{itemize}
\end{proof}
\begin{remark}
Note that in \eqref{vol-est-2}, we can take $\alpha $ to be as close to $\frac{1}{2} $ as intended.
\end{remark}

\subsection{The asymptotic approximations}
We shall now describe the proof of the results in the regime $\gamma=1, t+h_1 \leq 1,\ 1<s+h_1<\min \{\frac{3}{2}-t, 2-h_1\} $ and when the frequency is near the resonance with $l_M >0 $. As discussed already in section \ref{scaling}, $\mathbf{C} $ is positive in this regime. The proofs of the results in the other cases follow similarly and therefore we skip them to avoid repetition.\\ 
To begin with, we write the algebraic system \eqref{LAS-1-theorem} as
\begin{equation}
Y_m+\sum_{{j=1} \atop {j\neq m}}^{M} \Phi_{\kappa_0}(z_m,z_j) \overline{\mathbf{C}} Y_j a^{1-h_1}= u^{I}(z_m),
\label{est-vol-0}
\end{equation}
where $Y_j=-\mathbf{C}^{-1} Q_j,\ \mathbf{C}=\overline{\mathbf{C}} a^{1-h_1},\ j=1,\dots,M $.
\\
The main step in the proof lies in comparing \eqref{est-vol-0} to the volume integral equation
\begin{equation}
Y(z)+a^{1-h_1-s} \int_{\Omega} \Phi_{\kappa_0}(z,y) K^{M}(y) \overline{\mathbf{C}} Y(y) dy=u^{I}(z).
\notag
\end{equation}
For $m=1,\dots,M $, we rewrite the above integral equation in the form
\begin{equation}
\begin{aligned}
&Y(z_m)+a^{1-s-h_1}\sum_{{j=1} \atop {j\neq m}}^{M} \Phi_{\kappa_0}(z_m,z_j) \overline{\mathbf{C}} Y(z_j) a^{s}\\
&= u^{I}(z_m)+a^{1-s-h_1} \underbrace{ \left[ \sum_{{j=1} \atop {j \neq m}}^{M}  \Phi_{\kappa_0}(z_m,z_j) \overline{\mathbf{C}} Y(z_j) a^{s}
 - \sum_{{j=1} \atop {j \neq m}}^{[a^{-s}]}  \Phi_{\kappa_0}(z_m,z_j) K^M(z_j) \overline{\mathbf{C}} Y(z_j) Vol(\Omega_j) \right]}_{A_1}\\
&\quad +a^{1-s-h_1} \underbrace{ \left[ \sum_{{j=1} \atop {j \neq m}}^{[a^{-s}]}  \Phi_{\kappa_0}(z_m,z_j) K^M(z_j) \overline{\mathbf{C}} Y(z_j) Vol(\Omega_j) - \int_{{\cup_{{j=1} \atop {j \neq m}}^{[a^{-s}]}}\Omega_j}  \Phi_{\kappa_0}(z_m,y) K^M(y) \overline{\mathbf{C}} Y(y) dy \right]}_{B_1} \\
&\quad -a^{1-s-h_1} \underbrace{\int_{\Omega_m}  \Phi_{\kappa_0}(z_m,y) K^M(y) \overline{\mathbf{C}} Y(y) dy}_{C_1} -a^{1-s-h_1} \underbrace{\int_{\Omega \setminus \cup_{j=1 }^{[a^{-s}]} \Omega_j}  \Phi_{\kappa_0}(z_m,y) K^M(y) \overline{\mathbf{C}} Y(y) dy}_{D_1} .
\end{aligned}
\label{est-vol-1}
\end{equation}
We next estimate the quantities $A_1,B_1,C_1 $ and $D_1 $. 
To estimate the term $C_1$, we proceed as follows. Using \eqref{vol-est-2}, we can write
\begin{equation}
\begin{aligned}
\vert C_1 \vert=\left\vert \int_{\Omega_m} \Phi_{\kappa_0}(z_m,y) K^{M}(y) \overline{\mathbf{C}} Y(y) dy \right\vert 
&\leq \norm{K^M}_{L^{\infty}(\Omega)} \vert \overline{\mathbf{C}} \vert \norm{Y}_{L^{\infty}(\Omega)} \left\vert \int_{\Omega_m} \Phi_{\kappa_0} (z_m,y) dy \right\vert \\
&\leq C \ a^{\frac{1-s-h_{1}}{2}} \left\vert \int_{\Omega_m} \Phi_{\kappa_0} (z_m,y) dy \right\vert \\
 &\leq \frac{C}{4 \pi} a^{\frac{1-s-h_{1}}{2}} \Big(\int_{B(z_m,r)} \frac{1}{\vert z_m-y \vert} dy + \int_{\Omega_m \setminus B(z_m,r)} \frac{1}{\vert z_m-y \vert} dy \Big) \\
&\leq \frac{C}{4 \pi} a^{\frac{1-s-h_{1}}{2}} \underbrace{\left(2 \pi r^2+\frac{1}{r} \left[a^s-\frac{4}{3} \pi r^3 \right] \right)}_{:=l(r,a)},
\end{aligned}
\notag
\end{equation}
where $r<\frac{1}{2} a^{\frac{s}{3}}$.
Now the term $l(r,a) $ in the right hand side attains its maximum for $r=\left(\frac{3}{4} \pi a^{s} \right)^{\frac{1}{3}} $ and its maximum value is $\frac{3}{2} \left(\frac{4}{3 \pi} \right)^{\frac{1}{3}} a^{\frac{2s}{3}} $. Therefore
\begin{equation}
\begin{aligned}
\left\vert C_1 \right\vert &= \mathcal{O}\left(a^{\frac{1-s-h_{1}}{2}} a^{\frac{2s}{3}}\right).
\end{aligned}
\label{est-vol-3}
\end{equation}
Let us next estimate the term $B_1$. In order to be able to estimate the sum of the integrals efficiently, we count the bubbles in the following manner. Relative to each $\Omega_m $, we distinguish the other bubbles (located outside) as near and far ones. One convenient way to do this is to take the $\Omega_j $, which are assumed to be cubes, as arranged in a cuboid fashion, like in the Rubik's cube. It is easy to see that in such an arrangement, the total number of cubes upto the $n^{th} $ layer is $(2n+1)^{3},\ n=0,\dots,[a^{-\frac{s}{3}}] $ and $\Omega_m $ is located at the centre. So the number of bubbles located in the $n^{th} $ layer ($n \neq 0 $) will be atmost $[(2n+1)^{3}-(2n-1)^{3}] $ and their distance from $D_m $ is more than $n\left(a^{\frac{s}{3}}-\frac{a}{2} \right) $. For more explanations, we refer to Section 3.3. of \cite{ACCS}.\\
Next, let us define \[f(z_m,y):=\Phi_{\kappa_0}(z_m,y)Y(y). \] Using Taylor's expansion, we then have
\[f(z_m,y)-f(z_m,z_l)=(y-z_l)R_{l}(z_m,y), \]
where
\begin{equation}
\begin{aligned}
R_{l}(z_m,y)&=\int_{0}^{1} \nabla_y f(z_m,y-\beta(y-z_l)) d\beta \\
&=\int_{0}^{1} \left[\nabla_y \Phi_{\kappa_0}(z_m,y-\beta(y-z_l)) \right] Y(y-\beta(y-z_l)) d\beta\\
&\qquad \qquad+\int_{0}^{1} \Phi_{\kappa_0}(z_m,y-\beta(y-z_l))\left[\nabla_{y} Y(y-\beta(y-z_l)) \right] d\beta.
\end{aligned}
\notag
\end{equation}
Now for $x\neq y $, $\nabla_{y} \Phi_{\kappa_0}(x,y)= \Phi_{\kappa_0}(x,y)\left[\frac{1}{\vert x-y \vert}-i\kappa_0 \right] \frac{x-y}{\vert x-y \vert} $  and hence for $l\neq m $, we can write
\[\vert \Phi_{\kappa_0}(z_m,y-\beta(y-z_l)) \vert \leq \frac{1}{4 \pi n \left(a^{\frac{s}{3}}-\frac{a}{2} \right)} , \ \vert \nabla_y \Phi_{\kappa_0}(z_m,y-\beta(y-z_l)) \vert \leq \frac{1}{4 \pi n \left(a^{\frac{s}{3}}-\frac{a}{2} \right)} \left[\frac{1}{ n \left(a^{\frac{s}{3}}-\frac{a}{2} \right)} +\kappa_{0} \right] , \]
and therefore
\begin{equation}
\begin{aligned}
\vert R_{l}(z_m,y)\vert&\leq \frac{1}{n(a^{\frac{s}{3}}-\frac{a}{2})} \left(\left[\frac{1}{n(a^{\frac{s}{3}}-\frac{a}{2})} +\kappa_0 \right] \int_{0}^{1} \vert Y(y-\beta(y-z_l)) \vert d\beta+\int_{0}^{1} \vert \nabla_{y}Y(y-\beta(y-z_l)) \vert d\beta\right)\\
&\leq \frac{C}{n(a^{\frac{s}{3}}-\frac{a}{2})} \left(\left[\frac{1}{n(a^{\frac{s}{3}}-\frac{a}{2})} +\kappa_0 \right]\norm{Y}_{L^{\infty}(\Omega)}+ \norm{\nabla Y}_{L^{\infty}(\Omega)} \right).
\end{aligned}
\label{est-vol-3a}
\end{equation}
Now note that, we can write
\begin{equation}
\begin{aligned}
B_{1} &=\sum_{{j=1} \atop {j \neq m}}^{[a^{-s}]}  \Phi_{\kappa_0}(z_m,z_j) K^M(z_j) \overline{\mathbf{C}} Y(z_j) Vol(\Omega_j) -\int_{{\cup_{{j=1} \atop {j \neq m}}^{[a^{-s}]}}\Omega_j}  \Phi_{\kappa_0}(z_m,y) K^M(y) \overline{\mathbf{C}} Y(y) dy \\
&=\underbrace{\sum_{{j=1} \atop {j \neq m}}^{[a^{-s}]} \int_{\Omega_j} \Phi_{\kappa_0}(z_m,z_j) K^M(z_j) \overline{\mathbf{C}} Y(z_j) dy - \sum_{{j=1} \atop {j \neq m}}^{[a^{-s}]} \int_{\Omega_j} \Phi_{\kappa_0}(z_m,y) K^M(z_j) \overline{\mathbf{C}} Y(y) dy}_{B_{1,1}} \\
&\ +\underbrace{\sum_{{j=1} \atop {j \neq m}}^{[a^{-s}]} \int_{\Omega_j} \Phi_{\kappa_0}(z_m,y) K^M(z_j) \overline{\mathbf{C}} Y(y) dy-\sum_{{j=1} \atop {j \neq m}}^{[a^{-s}]} \int_{\Omega_j} \Phi_{\kappa_0}(z_m,y) K^M(y) \overline{\mathbf{C}} Y(y) dy}_{B_{1,2}}.
\end{aligned}
\label{est-vol-4}
\end{equation}
Therefore using \eqref{est-vol-3a} and \eqref{vol-est-2}, we obtain
\begin{equation}
\begin{aligned}
&\vert B_{1,1} \vert \leq \norm{K^M}_{L^{\infty}(\Omega)} \vert \overline{\mathbf{C}} \vert \sum_{{j=1} \atop {j \neq m}}^{[a^{-s}]} \left\vert  \int_{\Omega_j} \left[\Phi_{\kappa_0}(z_m,z_j)  Y(z_j)-\Phi_{\kappa_0}(z_m,y)  Y(y) \right] \right\vert \\
&\leq \norm{K^M}_{L^{\infty}(\Omega)} \vert \overline{\mathbf{C}} \vert \sum_{n=1}^{[a^{-\frac{s}{3}}]}  \left[(2n+1)^3-(2n-1)^3 \right] a^{\frac{s}{3}} a^{s} \frac{c}{n \left(a^{\frac{s}{3}}-\frac{a}{2} \right)} \left(\left[\frac{1}{n \left(a^{\frac{s}{3}}-\frac{a}{2} \right)} + \kappa_{0} \right] \norm{Y}_{L^{\infty}(\Omega)}+ \norm{\nabla Y}_{L^{\infty}(\Omega)} \right) \\
&=\mathcal{O}\left(\sum_{n=1}^{[a^{-\frac{s}{3}}]} \left[24n^2+2 \right] a^{\frac{4s}{3}} \left[\frac{1}{n^2} a^{-\frac{2s}{3}} \norm{Y}_{L^{\infty}(\Omega)}+\frac{1}{n} a^{-\frac{s}{3}} \norm{\nabla Y}_{L^{\infty}(\Omega)} \right] \right) \\
&=\mathcal{O}\left(\sum_{n=1}^{[a^{-\frac{s}{3}}]} \left[24n^2+2 \right] a^{\frac{4s}{3}} \left[\frac{1}{n^2} a^{-\frac{2s}{3}} a^{\frac{1-s-h_{1}}{2}} +\frac{1}{n} a^{-\frac{s}{3}} a^{\frac{1-s-h_{1}}{2}(1+\alpha)} \right] \right)\\
&= \mathcal{O}\left(a^{\frac{2s}{3}}a^{\frac{1-s-h_{1}}{2}}\sum_{n=1}^{[a^{-\frac{s}{3}}]} \left[24+\frac{2}{n^2} \right]+a^{s}a^{\frac{1-s-h_{1}}{2}(1+\alpha)} \sum_{n=1}^{[a^{-\frac{s}{3}}]} \left[24n+\frac{2}{n} \right] \right)\\
&=\mathcal{O} \left(a^{\frac{2s}{3}}a^{\frac{1-s-h_{1}}{2}} a^{-\frac{s}{3}} +a^{s}a^{\frac{1-s-h_{1}}{2}(1+\alpha)}  a^{-\frac{2s}{3}} \right)= \mathcal{O} \left(a^{\frac{s}{3}}a^{\frac{1-s-h_{1}}{2}}  +a^{\frac{1-s-h_{1}}{2}(1+\alpha)}  a^{\frac{s}{3}} \right).
\end{aligned}
\label{est-vol-5}
\end{equation}
Similarly using the fact $K \in C^{0,\lambda}(\overline{\Omega})$, we deduce 
\begin{equation}
\begin{aligned}
\vert B_{1,2} \vert &\leq  \norm{Y}_{L^{\infty}(\Omega)} \vert \overline{\mathbf{C}} \vert \sum_{{j=1} \atop {j \neq m}}^{[a^{-s}]} \int_{\Omega_j}  \left\vert \Phi_{\kappa_0}(z_m,y) \right\vert \left\vert K^M(z_j)-  K^M(y) \right\vert dy \\
&\leq \norm{Y}_{L^{\infty}(\Omega)} \vert \overline{\mathbf{C}} \vert \sum_{n=1}^{[a^{-\frac{s}{3}}]}  \left[(2n+1)^3-(2n-1)^3 \right] \frac{c}{n \left(a^{\frac{s}{3}}-\frac{a}{2} \right)} \int_{\Omega_j} \vert z_j-y \vert^{\lambda} [K]_{C^{0,\lambda}(\overline{\Omega})} dy \\
&\leq \norm{Y}_{L^{\infty}(\Omega)} \vert \overline{\mathbf{C}} \vert [K]_{C^{0,\lambda}(\overline{\Omega})} \sum_{n=1}^{[a^{-\frac{s}{3}}]}  \left[(2n+1)^3-(2n-1)^3 \right] \frac{c}{n \left(a^{\frac{s}{3}}-\frac{a}{2} \right)} a^{s} a^{\frac{s}{3}\lambda} \\
&=\mathcal{O}\left(a^{\frac{1-s-h_{1}}{2}}\sum_{n=1}^{[a^{-\frac{s}{3}}]} [24n^2+2] a^{s+\frac{s\lambda}{3}} \frac{1}{n} a^{-\frac{s}{3}} \right)=\mathcal{O}\left(a^{\frac{1-s-h_{1}}{2}} a^{\frac{2s+s\lambda}{3}} \sum_{n=1}^{[a^{-\frac{s}{3}}]} [24n+\frac{2}{n}] \right)\\
&=\mathcal{O}\left(a^{\frac{1-s-h_{1}}{2}} a^{\frac{s\lambda}{3}} \right).
\end{aligned}
\label{est-vol-5a}
\end{equation}
From \eqref{est-vol-5} and \eqref{est-vol-5a}, we obtain
\begin{equation}
\begin{aligned}
 B_{1}  &= \mathcal{O}\left(a^{\frac{1-s-h_{1}}{2}} a^{\frac{s\lambda}{3}} +a^{\frac{1-s-h_{1}}{2}(1+\alpha)}  a^{\frac{s}{3}}\right).
\end{aligned}
\label{est-vol-6a}
\end{equation}
To estimate the term $D_1$, we distinguish between the following two cases (see also \cite{CMS}).
\begin{itemize}
\item[(a)] The point $z_m$ is away from the boundary $\partial \Omega  $ and so $\Phi_{\kappa_0}(z_m,\cdot) $ is bounded in $y$ near the boundary.

\item[(b)] The point $z_m$ is located near one of the $\Omega_j $'s touching the boundary $\partial \Omega $. In this case, we split the estimate into two parts. By $ N_m$ we denote the part that involves $\Omega_j $'s close to $z_m $, and we denote the remaining part by $F_m $. The integral over $F_m $ can be estimated in a manner similar to the case $(a)$ discussed above. Also note that $F_m \subset \Omega\setminus \cup_{j=1}^{[a^{-s}]} \Omega_j $ and so $Vol(F_m) $ is of the order $a^{\frac{s}{3}} $ as $a \rightarrow 0 $. \\
To estimate the integral over $N_m $, we observe that owing to the fact $a$ is small, the $\Omega_j $'s  close to $z_m $ are located near a small region of the boundary $\partial \Omega $. Since we assume that the boundary is smooth enough, this region can be assumed to be flat. We now divide this layer into concentric layers as in the estimate of $B_1$. In this case, we have at most $(2n+1)^{2} $ cubes intersecting the surface, for $n=0,\dots,[a^{-\frac{s}{3}}] $. So the number of bubbles in the $n^{th} $ layer $(n \neq 0)$ will be at most $[(2n+1)^{2}-(2n-1)^{2}] $ and their distance from $D_m $ is atleast $n\left(a^{\frac{s}{3}}-\frac{a}{2} \right) $. 
\end{itemize}
Therefore we can write
\begin{equation}
\begin{aligned}
\vert D_1 \vert&=\left\vert \int_{\Omega \setminus \cup_{j=1 }^{[a^{-s}]} \Omega_j}  \Phi_{\kappa_0}(z_m,y) K^M(y) \overline{\mathbf{C}} Y(y) dy \right\vert\\
 &=\left\vert \int_{N_m}  \Phi_{\kappa_0}(z_m,y) K^M(y) \overline{\mathbf{C}} Y(y) dy \right\vert+\left\vert \int_{F_m}  \Phi_{\kappa_0}(z_m,y) K^M(y) \overline{\mathbf{C}} Y(y) dy \right\vert \\
&\leq \sum_{l=1}^{[a^{-\frac{2s}{3}}]} \norm{K^M}_{L^{\infty}(\Omega)} \norm{Y}_{L^{\infty}(\Omega)} \vert \overline{\mathbf{C}} \vert Vol(\Omega_l) \frac{1}{d_{ml}} + \norm{\Phi_{\kappa_0}(z_m,\cdot)}_{L^{\infty}(F_m)}   \norm{K^M}_{L^{\infty}(\Omega)} \norm{Y}_{L^{\infty}(\Omega)} \vert \overline{\mathbf{C}} \vert Vol(F_m) \\
&\leq \mathcal{O}(a^{\frac{1-h_{1}-s}{2}}) \norm{K^M}_{L^{\infty}(\Omega)} \norm{Y}_{L^{\infty}(\Omega)} \vert \overline{\mathbf{C}} \vert a^{s} \sum_{l=1}^{[a^{-\frac{2s}{3}}]} \frac{1}{d_{ml}} +C a^{\frac{1-h_{1}-s}{2}} a^{\frac{s}{3}} \\
&\leq \mathcal{O}(a^{\frac{1-h_{1}-s}{2}}) \norm{K^M}_{L^{\infty}(\Omega)} \norm{Y}_{L^{\infty}(\Omega)} \vert \overline{\mathbf{C}} \vert a^{s} \sum_{l=1}^{[a^{-\frac{s}{3}}]} \Big[(2n+1)^2-(2n-1)^2 \Big] \left(\frac{1}{n \Big( a^{\frac{s}{3}}-\frac{a}{2} \Big)} \right) +C a^{\frac{1-h_{1}-s}{2}} a^{\frac{s}{3}}\\
&= \mathcal{O}(a^{\frac{1-h_{1}-s}{2}}) \norm{K^M}_{L^{\infty}(\Omega)} \norm{Y}_{L^{\infty}(\Omega)} \vert \overline{\mathbf{C}} \vert a^{s} \mathcal{O}(a^{-\frac{2s}{3}})+\mathcal{O}(a^{\frac{1-h_{1}-s}{2}}) \mathcal{O}(a^{\frac{s}{3}}),
\end{aligned}
\label{est-vol-2}
\end{equation}
and hence
\begin{equation}
\vert D_1 \vert=\mathcal{O}(a^{\frac{1-h_{1}-s}{2}} a^{\frac{s}{3}}).
\label{est-vol-2a}
\end{equation}
Finally, we deal with the term $A_1 $ as follows. Note that using the fact that $Vol(\Omega_j)=a^{s} \frac{[K^{M}(z_j)]}{K^{M}(z_j)} $, we can write
\begin{equation}
\begin{aligned}
A_{1}&=\sum_{{j=1} \atop {j \neq m}}^{M}  \Phi_{\kappa_0}(z_m,z_j) \overline{\mathbf{C}} Y(z_j) a^{s} -\sum_{{j=1} \atop {j \neq m}}^{[a^{-s}]}  \Phi_{\kappa_0}(z_m,z_j) K^M(z_j) \overline{\mathbf{C}} Y(z_j) Vol(\Omega_j)\\
&=\sum_{{l=1} \atop {{l \neq m} \atop {z_l \in \Omega_m}}}^{[K^{M}(z_m)]} \Phi_{\kappa_0}(z_m,z_l) \overline{\mathbf{C}} Y(z_l) a^{s}  +\sum_{{j=1} \atop {j \neq m}}^{[a^{-s}]} \sum_{{l=1} \atop {z_l \in \Omega_j}}^{[K^M(z_j)]} \Phi_{\kappa_0}(z_m,z_l) \overline{\mathbf{C}} Y(z_l) a^{s}\\
&\quad - \sum_{{j=1} \atop {j \neq m}}^{[a^{-s}]}  \Phi_{\kappa_0}(z_m,z_j) K^M(z_j) \overline{\mathbf{C}} Y(z_j) Vol(\Omega_j)\\
&=\overline{\mathbf{C}} a^{s} \underbrace{\sum_{{l=1} \atop {{l \neq m} \atop {z_l \in \Omega_m}}}^{[K^{M}(z_m)]} \Phi_{\kappa_0}(z_m,z_l) Y(z_l)}_{E_{1}}+\sum_{{j=1} \atop {j \neq m}}^{[a^{-s}]} \overline{\mathbf{C}} a^{s}  \underbrace{\left[\left(\sum_{{l=1} \atop {z_l \in \Omega_j}}^{[K^M(z_j)]}\Phi_{\kappa_0}(z_m,z_l) Y(z_l)  \right)-\Phi_{\kappa_0}(z_m,z_j) \left[K^M(z_j)\right] Y(z_j) \right]}_{E_{2}^{j}}.
\end{aligned}
\label{est-vol-6}
\end{equation}
Now it is easy to see that
\begin{equation}
\vert \overline{\mathbf{C}} a^{s} E_{1} \vert \leq \frac{C (K_{max}-1) \overline{\mathbf{C}}}{4 \pi} \frac{a^s}{d}.
\label{est-vol-7}
\end{equation}
To estimate the term $E^{j}_{2} $, we note that
\begin{equation}
\begin{aligned}
E^{j}_{2}&= \left[\left(\sum_{{l=1} \atop {z_l \in \Omega_j}}^{[K^M(z_j)]}\Phi_{\kappa_0}(z_m,z_l) Y(z_l)  \right)-\Phi_{\kappa_0}(z_m,z_j) \left[K^M(z_j)\right] Y(z_j) \right]\\
&=\sum_{{l=1} \atop {z_{l} \in \Omega_j}}^{[K^{M}(z_j)]} \left( \Phi_{\kappa_0}(z_m,z_l) Y(z_l)-\Phi_{\kappa_0}(z_m,z_j) Y(z_j) \right).
\end{aligned}
\label{est-vol-8}
\end{equation}
Therefore arguing as in the case of $B_1$ and using \eqref{vol-est-2}, we obtain
\begin{equation}
\begin{aligned}
&\left\vert \sum_{{j=1} \atop {j \neq m}}^{[a^{-s}]} \overline{\mathbf{C}} a^{s} E^{j}_{2} \right\vert \leq \vert \overline{\mathbf{C}} \vert a^{s} \left\vert \sum_{{j=1} \atop {j \neq m}}^{[a^{-s}]}  E^{j}_{2} \right\vert \\
&\leq \vert \overline{\mathbf{C}} \vert\ a^{s} \norm{[K^{M}]}_{L^{\infty}(\Omega)} \sum_{n=1}^{[a^{-\frac{s}{3}}]}  \left[(2n+1)^3-(2n-1)^3 \right] a^{\frac{s}{3}} \frac{C}{n \left(a^{\frac{s}{3}}-\frac{a}{2} \right)} \left(\left[\frac{1}{n \left(a^{\frac{s}{3}}-\frac{a}{2} \right)} + \kappa_{0} \right] \norm{Y}_{L^{\infty}(\Omega)}+ \norm{\nabla Y}_{L^{\infty}(\Omega)} \right)\\
&=\mathcal{O} \left(a^{\frac{2s}{3}}a^{\frac{1-s-h_{1}}{2}} a^{-\frac{s}{3}} +a^{s}a^{\frac{1-s-h_{1}}{2}(1+\alpha)}  a^{-\frac{2s}{3}} \right).
\end{aligned}
\label{est-vol-9}
\end{equation}
Hence
\begin{equation}
\vert A_1 \vert=\mathcal{O}\left(a^{\frac{s}{3}}a^{\frac{1-s-h_{1}}{2}} +a^{\frac{1-s-h_{1}}{2}(1+\alpha)}  a^{\frac{s}{3}}+a^{s-t} \right).
\label{est-vol-9a}
\end{equation}
Using the estimates \eqref{est-vol-3},\eqref{est-vol-2a},\eqref{est-vol-6a} and \eqref{est-vol-9a} in \eqref{est-vol-1}, we deduce
\begin{equation}
\begin{aligned}
&Y(z_m)+\sum_{{j=1} \atop {j \neq m}}^{M} \Phi_{\kappa_0}(z_m,z_j) \overline{\mathbf{C}} Y(z_j) a^{1-h_1}\\
&=u^{I}(z_m)+\mathcal{O}\left(a^{1-s-h_1}a^{\frac{1-s-h_1}{2}}a^{\frac{s\lambda}{3}}+a^{\frac{s}{3}}a^{1-s-h_1}a^{\frac{1-s-h_1}{2}(1+\alpha)}\right)\\
&\quad +\mathcal{O}\left( a^{1-s-h_1}a^{\frac{1-s-h_1}{2}}a^{\frac{2s}{3}} \right)+\mathcal{O}\left(a^{1-s-h_1} a^{s-t} \right) \\
&=u^{I}(z_m)+\mathcal{O}\left(a^{\frac{3}{2}(1-s-h_1)+\frac{s\lambda}{3}}+a^{\frac{3}{2}(1-s-h_1)+\frac{2s}{3}}+a^{\frac{3+\alpha}{2}(1-s-h_1)+\frac{s}{3}}+a^{1-h_1-t} \right).
\end{aligned}
\label{est-vol-10}
\end{equation}
Now comparing \eqref{est-vol-10} with \eqref{est-vol-0}, we can deduce
\begin{equation}
\begin{aligned}
&(Y_m-Y(z_m))+\sum_{{j=1} \atop {j \neq m}}^{M} \Phi_{\kappa_0}(z_m,z_j) \overline{\mathbf{C}} (Y_j-Y(z_j)) a^{1-h_1}\\
&\qquad =\mathcal{O}\left(a^{\frac{3}{2}(1-s-h_1)+\frac{s\lambda}{3}}+a^{\frac{3}{2}(1-s-h_1)+\frac{2s}{3}}+a^{\frac{3+\alpha}{2}(1-s-h_1)+\frac{s}{3}}+a^{1-h_1-t} \right)\\
&\qquad =\mathcal{O}\left(a^{\frac{3}{2}(1-s-h_1)+\frac{s\lambda}{3}}+a^{\frac{3+\alpha}{2}(1-s-h_1)+\frac{s}{3}}+a^{1-h_1-t} \right).
\end{aligned}
\label{est-vol-11}
\end{equation}
Thus we see that $\left(Y_j-Y(z_j)\right)^{M}_{j=1} $ satisfies the linear algebraic system \eqref{est-vol-0} albeit a different right hand term. Using the invertibility of the algebraic system, we can now derive the estimate
\begin{equation}
\sum_{m=1}^{M} \vert Y_m-Y(z_m) \vert=\mathcal{O}\left(M \left[a^{\frac{3}{2}(1-s-h_1)+\frac{s\lambda}{3}}+a^{\frac{3+\alpha}{2}(1-s-h_1)+\frac{s}{3}} +a^{1-h_1-t}\right]\right).
\label{est-vol-12}
\end{equation}
Now we are in a position to compare the far-field values. To do so, we recall (from \eqref{Near-resonance}) that the far-field satisfies the estimate
\begin{equation}
u^{\infty}(\hat{x},\theta)= -\sum_{j=1}^{M} e^{-i\kappa_{0} \hat{x} \cdot z_{j}} \overline{\mathbf{C}} Y_{j}  a^{1-h_{1}}+\mathcal{O}(a^{2-s-2h_{1}}+a^{3-2t-2s-2h_{1}}).
\notag
\end{equation}
Next let us denote
\begin{equation}
u^{\infty}_{a}(\hat{x},\theta)=-a^{1-h_{1}-s} \int_{\Omega} e^{-i\kappa_{0} \hat{x}\cdot y} K^{M}(y)\overline{\mathbf{C}} Y(y) dy.
\notag
\end{equation}
Then
\begin{equation}
\begin{aligned}
&u^{\infty}(\hat{x},\theta)-u_{a}^{\infty}(\hat{x},\theta)\\
&=a^{1-h_1-s}\left[\int_{\Omega} e^{-i\kappa_{0} \hat{x} \cdot y} K^{M}(y) \overline{\mathbf{C}} Y(y) dy-\sum_{j=1}^{M} e^{-i \kappa_{0} \hat{x} \cdot z_{j}} \overline{\mathbf{C}} Y_{j} a^{s} \right] +\mathcal{O}(a^{2-s-h_1}+a^{3-2t-2s-2h_1})
\end{aligned}
\notag
\end{equation}
which we can further write as
\begin{equation}
\begin{aligned}
&u^{\infty}(\hat{x},\theta)-u_{a}^{\infty}(\hat{x},\theta)\\
&=a^{1-h_1-s}\int_{\Omega \setminus \cup_{j=1}^{[a^{-s}]} \Omega_j} e^{-i\kappa_{0} \hat{x} \cdot y} K^{M}(y) \overline{\mathbf{C}} Y(y) dy +a^{1-h_1-s}\sum_{j=1}^{[a^{-s}]} \int_{\Omega_j} e^{-i\kappa_{0} \hat{x} \cdot y} K^{M}(y) \overline{\mathbf{C}} Y(y) dy\\
&\quad -a^{1-h_1-s}\sum_{j=1}^{M} e^{-i \kappa_{0} \hat{x} \cdot z_{j}} \overline{\mathbf{C}} Y_{j} a^{s} +\mathcal{O}(a^{2-s-h_1}+a^{3-2t-2s-2h_1}).
\end{aligned}
\notag
\end{equation}
Therefore
\begin{equation}
\begin{aligned}
&u^{\infty}(\hat{x},\theta)-u_{a}^{\infty}(\hat{x},\theta)\\
&=a^{1-h_1-s}\sum_{j=1}^{[a^{-s}]} \int_{\Omega_j} e^{-i\kappa_{0} \hat{x} \cdot y} K^{M}(y) \overline{\mathbf{C}} Y(y) dy-a^{1-h_1-s}\sum_{j=1}^{[a^{-s}]} \sum_{{l=1} \atop {z_{l} \in \Omega_j}}^{[K^{M}(z_j)]} e^{-i \kappa_{0} \hat{x} \cdot z_{l}} \overline{\mathbf{C}} Y_{l} a^{s}\\
&\quad +\mathcal{O}(a^{2-s-2h_1}+a^{3-2t-2s-2h_1})+\underbrace{\mathcal{O}\left(a^{1-h_{1}-s} \norm{K^{M}}_{L^{\infty}(\Omega)} \vert \overline{\mathbf{C}} \vert \norm{Y}_{L^{\infty}(\Omega)} Vol( \Omega\setminus \cup_{j=1}^{M} \Omega_j ) \right)}_{a^{1-h_{1}-s} a^{\frac{1-h_{1}-s}{2}} a^{\frac{s}{3}}}\\
\end{aligned}
\notag
\end{equation}
which we can rewrite as
\begin{equation}
\begin{aligned}
&u^{\infty}(\hat{x},\theta)-u_{a}^{\infty}(\hat{x},\theta)\\
&= a^{1-h_1-s}\sum_{j=1}^{[a^{-s}]} K^{M}(z_j) \overline{\mathbf{C}} \int_{\Omega_j} \left[e^{-i\kappa_{0} \hat{x} \cdot y}  Y(y)-e^{-i\kappa_{0} \hat{x} \cdot z_j}  Y(z_j)  \right] dy\\
&\quad +a^{1-h_1-s} \left[\sum_{j=1}^{[a^{-s}]} \int_{\Omega_j} e^{-i\kappa_{0} \hat{x} \cdot y} K^{M}(y) \overline{\mathbf{C}} Y(y) dy-\sum_{j=1}^{[a^{-s}]} \int_{\Omega_j} e^{-i\kappa_{0} \hat{x} \cdot y} K^{M}(z_j) \overline{\mathbf{C}} Y(y) dy \right]\\
&\quad +\sum_{j=1}^{[a^{-s}]}\overline{\mathbf{C}} a^{1-h_1} \left[\sum_{{l=1} \atop {z_{l} \in \Omega_j}}^{[K^{M}(z_j)]} \left(e^{-i \kappa_{0} \hat{x} \cdot z_{j}} Y(z_j) -e^{-i \kappa_{0} \hat{x} \cdot z_{l}} Y(z_l)\right) +\sum_{{l=1} \atop {z_{l} \in \Omega_j}}^{[K^{M}(z_j)]} e^{-i \kappa_{0} \hat{x} \cdot z_{l}} (Y(z_l)- Y_{l})\right] \\
&\quad +\mathcal{O}\left(a^{\frac{3}{2}(1-h_{1}-s)+\frac{s}{3}}+a^{2-s-2h_{1}}+a^{3-2t-2s-2h_{1}}\right).
\end{aligned}
\notag
\end{equation}
This further implies
\begin{equation}
\begin{aligned}
&u^{\infty}(\hat{x},\theta)-u_{a}^{\infty}(\hat{x},\theta) \\
&= a^{1-h_1-s}\sum_{j=1}^{[a^{-s}]} K^{M}(z_j) \overline{\mathbf{C}} \int_{\Omega_j} \left[e^{-i\kappa_{0} \hat{x} \cdot y}  Y(y)-e^{-i\kappa_{0} \hat{x} \cdot z_j}  Y(z_j)  \right] dy \\
&\quad +a^{1-h_1-s} \left[\sum_{j=1}^{[a^{-s}]} \int_{\Omega_j} e^{-i\kappa_{0} \hat{x} \cdot y} K^{M}(y) \overline{\mathbf{C}} Y(y) dy-\sum_{j=1}^{[a^{-s}]} \int_{\Omega_j} e^{-i\kappa_{0} \hat{x} \cdot y} K^{M}(z_j) \overline{\mathbf{C}} Y(y) dy \right]\\
&\quad +\sum_{j=1}^{[a^{-s}]}\overline{\mathbf{C}} a^{1-h_1} \sum_{{l=1} \atop {z_{l} \in \Omega_j}}^{[K^{M}(z_j)]} \left(e^{-i \kappa_{0} \hat{x} \cdot z_{j}} Y(z_j) -e^{-i \kappa_{0} \hat{x} \cdot z_{l}} Y(z_l)\right) \\
&\quad +\mathcal{O}\left(a^{\frac{3}{2}(1-h_{1}-s)+\frac{s}{3}}+a^{2-s-2h_{1}}+a^{3-2t-2s-2h_{1}}\right)\\
&\quad +\mathcal{O}\left(M a^{1-h_1} \left[a^{\frac{3}{2}(1-s-h_1)+\frac{s\lambda}{3}}+a^{\frac{3+\alpha}{2}(1-s-h_1)+\frac{s}{3}} +a^{1-h_1-t}\right] \right).
\end{aligned}
\label{est-vol-13}
\end{equation}
Now
\begin{equation}
\begin{aligned}
&\sum_{j=1}^{[a^{-s}]}\overline{\mathbf{C}} a^{1-h_1} \sum_{{l=1} \atop {z_{l} \in \Omega_j}}^{[K^{M}(z_j)]} \left(e^{-i \kappa_{0} \hat{x} \cdot z_{j}} Y(z_j) -e^{-i \kappa_{0} \hat{x} \cdot z_{l}} Y(z_l)\right)\\
&\qquad \qquad =\mathcal{O}(a^{1-h_1-s} a^{\frac{s }{3}} \norm{\nabla Y}_{L^{\infty}(\Omega)})+\mathcal{O}(a^{1-h_1-s} a^{\frac{s}{3}} \norm{Y}_{L^{\infty}(\Omega)})\\
&\qquad \qquad =\mathcal{O}(a^{1-h_1-s} a^{\frac{s }{3}} a^{\frac{1-h_1-s}{2}(1+\alpha)})+\mathcal{O}(a^{1-h_1-s} a^{\frac{s }{3}} a^{\frac{1-h_1-s}{2}}),
\end{aligned}
\label{est-vol-14}
\end{equation}
\begin{equation}
\begin{aligned}
&a^{1-h_1-s} \left[\sum_{j=1}^{[a^{-s}]} \int_{\Omega_j} e^{-i\kappa_{0} \hat{x} \cdot y} K^{M}(y) \overline{\mathbf{C}} Y(y) dy-\sum_{j=1}^{[a^{-s}]} \int_{\Omega_j} e^{-i\kappa_{0} \hat{x} \cdot y} K^{M}(z_j) \overline{\mathbf{C}} Y(y) dy \right] \\
&\quad =\mathcal{O}\left(a^{1-h_1-s} a^{\frac{1-h_1-s}{2}} a^{\frac{s\lambda}{3}} \right),
\end{aligned}
\label{est-vol-15}
\end{equation}
and
\begin{equation}
\begin{aligned}
&a^{1-h_1-s}\sum_{j=1}^{[a^{-s}]} K^{M}(z_j) \overline{\mathbf{C}} \int_{\Omega_j} \left[e^{-i\kappa_{0} \hat{x} \cdot y}  Y(y)-e^{-i\kappa_{0} \hat{x} \cdot z_j}  Y(z_j)  \right] dy \\
&=\mathcal{O}\left(a^{1-h_1-s} \left[\sum_{n=1}^{[a^{-\frac{s}{3}}]} [24n^2+2] a^{\frac{4s}{3}} \norm{Y}_{L^{\infty}(\Omega)}+\sum_{n=1}^{[a^{-\frac{s}{3}}]} [24n^2+2] a^{\frac{4s}{3}} \norm{\nabla Y}_{L^{\infty}(\Omega)} \right]\right)\\
&=\mathcal{O}\left(a^{1-h_1-s} a^{\frac{1-h_1-s}{2}} a^{\frac{s}{3}}\right)+\mathcal{O}(a^{1-h_1-s} a^{\frac{s }{3}} a^{\frac{1-h_1-s}{2}(1+\alpha)}) .
\end{aligned}
\label{est-vol-16}
\end{equation}
Using these in \eqref{est-vol-13}, we obtain
\begin{equation}
\begin{aligned}
u^{\infty}(\hat{x},\theta)-u_{a}^{\infty}(\hat{x},\theta)
&=\mathcal{O}\left(a^{2-s-2h_{1}}+a^{3-2t-2s-2h_{1}}+ a^{1-h_1-s} \left[a^{\frac{3}{2}(1-s-h_1)+\frac{s\lambda}{3}}+a^{\frac{3+\alpha}{2}(1-s-h_1)+\frac{s}{3}} +a^{1-h_1-t}\right]\right) \\
&=\mathcal{O}\left(a^{2-s-2h_{1}}+a^{3-2t-2s-2h_{1}}+  a^{\frac{5}{2}(1-s-h_1)+\frac{s\lambda}{3}}+a^{\frac{5+\alpha}{2}(1-s-h_1)+\frac{s}{3}} +a^{1-h_1-s} a^{1-h_1-t} \right)\\
&=\mathcal{O}\left(a^{2-s-2h_{1}}+a^{3-2t-2s-2h_{1}}+  a^{\frac{5}{2}(1-s-h_1)+\frac{s\lambda}{3}}+a^{(\frac{11}{4})_{+}\left(1-s-h_1\right)+\frac{s}{3}} +a^{1-h_1-s} a^{1-h_1-t} \right),
\end{aligned}
\label{est-vol-17}
\end{equation}
since we can choose $\alpha \in (\frac{1}{2},1] $ to be as close to $\frac{1}{2} $ as desired.\\
In the regime under consideration, we already have $2-s-2h_{1}>0 $ and $3-2t-2s-2h_{1}>0 $.
\begin{itemize}
\item Let us look at the term $a^{1-h_1-s} \cdot a^{1-h_1-t}=a^{2-2h_{1}-s-t} $. If the condition $h_{1}+t<\frac{1}{2} $ is satisfied, then
\[2-2h_{1}-s-t> 2-h_{1}-s-\frac{1}{2}=\frac{3}{2}-h_{1}-s>0, \] in the regime under consideration, that is, $1<s+h_{1}<\min\{\frac{3}{2}-t,2-h_1\} $.\\
Thus a sufficient condition can be written as
\begin{equation}
0<1-h_{1}<s \leq 3t< \min \left\{\frac{3}{2}-t-h_{1}, 2-2h_1 \right\},\ h_1 < \frac{1}{6}. 
\label{cond-vol-1}
\end{equation}
\item Next let us look for conditions to guarantee that $\frac{5}{2}(1-h_1-s)+\frac{s\lambda}{3}>0 $. Now $\frac{5}{2}(1-h_1-s)+\frac{s\lambda}{3}=\frac{5}{2}-\frac{5 h_1}{2}-\frac{5s}{2}+\frac{s \lambda}{3} $.\\
Hence if $s+h_{1}<1+\frac{2s\lambda}{15} $, then we can guarantee $\frac{5}{2}-\frac{5 h_1}{2}-\frac{5s}{2}+\frac{s \lambda}{3}>0 $. \\
Therefore a sufficient condition, in this case, can be written as
\begin{equation}
\begin{aligned}
&1<s+h_{1}<1+\frac{2s\lambda}{15}.
\end{aligned}
\label{cond-vol-2}
\end{equation}
\item We next want conditions to guarantee $(\frac{11}{4})_{+}(1-h_1-s)+\frac{s}{3}>0 $. Note that $(\frac{11}{4})_{+}(1-h_1-s)+\frac{s}{3}=(\frac{11}{4})_{+}-(\frac{11}{4})_{+}h_{1}+\left(\frac{1}{3}-(\frac{11}{4})_{+} \right)s $.\\
Now if $s<\frac{(\frac{11}{4})_{+}}{(\frac{11}{4})_{+}-\frac{1}{3}}(1-h_1) $, then we can guarantee $(\frac{11}{4})_{+}-(\frac{11}{4})_{+}h_{1}+\left(\frac{1}{3}-(\frac{11}{4})_{+} \right)s>0 $.\\
Hence a sufficient condition, in this case, can be written as
\begin{equation}
0<1-h_{1}<s<\frac{(\frac{11}{4})_{+}}{(\frac{11}{4})_{+}-\frac{1}{3}}(1-h_1). 
\label{cond-vol-3}
\end{equation}
\end{itemize}
Therefore we have the following set of sufficient conditions.
\begin{equation}
\begin{aligned}
& 0<1-h_{1}<s \leq 3t< \min \left\{\frac{3}{2}-t-h_{1}, 2-2h_1 \right\},\ h_1 < \frac{1}{6},\\
& 1<s+h_{1}<1+\frac{2s\lambda}{15},\\
&0<1-h_{1}<s<\frac{(\frac{11}{4})_{+}}{(\frac{11}{4})_{+}-\frac{1}{3}}(1-h_1).
\end{aligned}
\notag
\end{equation}

Now if $s< (1+\frac{2\lambda}{15})(1-h_1) $, then $s+h_1<1+\frac{2s\lambda}{15} $. Also if $\alpha \in (\frac{1}{2}, \frac{2}{3}) $, using the fact that $\lambda \in (0,1) $, it follows that $(1+\frac{2\lambda}{15})(1-h_1)<\frac{(\frac{11}{4})_{+}}{(\frac{11}{4})_{+}-\frac{1}{3}}(1-h_1) $.  Hence the above set of sufficient conditions can be replaced by
\begin{equation}
\begin{aligned}
& 0<1-h_{1}<s \leq 3t< \min \left\{\frac{3}{2}-t-h_{1}, 2-2h_1 \right\},\ h_1 < \frac{1}{6},\\
& 0<1-h_{1}<s<\left(1+\frac{2\lambda}{15}\right)(1-h_1),
\end{aligned}
\label{cond-vol-4}
\end{equation}
which can be further replaced by the condition\footnote{
Note that if $\frac{2\lambda}{15+2\lambda}<h_1<\frac{1}{6}  $, we have $\frac{3}{2}-t-h_1>\left(1+\frac{2\lambda}{15}\right) (1-h_1), \ 2-2h_1> \left(1+\frac{2\lambda}{15}\right) (1-h_1) $ and we can replace the conditions by the sufficient condition
\begin{equation}
\begin{aligned}
&0<1-h_1<s\leq 3t<\left(1+\frac{2\lambda}{15}\right)(1-h_1).
\end{aligned}
\notag
\end{equation}
}
\begin{equation}
\begin{aligned}
&0<1-h_{1}<s \leq 3t< \min \left\{\frac{3}{2}-t-h_{1}, \left(1+\frac{2\lambda}{15}\right)(1-h_1) \right\},\ h_1 < \frac{1}{6}.
\end{aligned}
\label{cond-vol-4a}
\end{equation}
Finally using \eqref{vol-est-1000} and \eqref{est-vol-17}, we deduce
\begin{equation}
\begin{aligned}
&u^{\infty}(\hat{x},\theta)-u_{D}^{\infty}(\hat{x},\theta)\\
&=\mathcal{O}\left(a^{\frac{s+h_1-1}{4}}+a^{2-s-2h_{1}}+a^{3-2t-2s-2h_{1}}+  a^{\frac{5}{2}(1-s-h_1)+\frac{s\lambda}{3}}+a^{(\frac{11}{4})_{+}\left(1-s-h_1\right)+\frac{s}{3}} +a^{1-h_1-s} a^{1-h_1-t} \right).
\end{aligned}
\label{volume-blow}
\end{equation}
\begin{remark}
In the cases when ($\gamma <1,\ \gamma+s=2 $) or ($\gamma=1, \gamma+s=2 $ with the frequency $\omega $ away from the Minnaert resonance), the estimates can be deduced similarly by using \eqref{vol-est-1} instead of \eqref{vol-est-2}. In these cases, we can further compare the far-fields corresponding to $\overline{\mathbf{C}} $ to that of $\overline{\mathbf{C}}_{lead} $ in the following manner. \\
Let us consider the Lippmann-Schwinger equation
\begin{equation}
\tilde{Y}(z)+ \int_{\Omega} \Phi_{\kappa_0}(z,y) K^{M}(y) \overline{\mathbf{C}}_{lead} \tilde{Y}(y) dy = u^{I}(z),
\label{vol-int-lead}
\end{equation}
and the corresponding far-field given by
\begin{equation}
u^{\infty}_{lead}(\hat{x},\theta)=- \int_{\Omega} e^{-i\kappa_{0} \hat{x}\cdot y} K^{M}(y)\overline{\mathbf{C}}_{lead} \tilde{Y}(y) dy.
\notag
\end{equation}
Let us first deal with the case $\gamma<1, \ \gamma+s=2 $. Using the fact that $ \overline{\mathbf{C}}=\overline{\mathbf{C}}_{lead} +\mathcal{O}(a^{2-\gamma})$ in \eqref{vol-int}, we can write \eqref{vol-int} as
\begin{equation}
Y(z)+ \int_{\Omega} \Phi_{\kappa_0}(z,y) K^{M}(y) \overline{\mathbf{C}}_{lead}  Y(y) dy = u^{I}(z)- \mathcal{O}(a^{1-\gamma}) \int_{\Omega} \Phi_{\kappa_0}(z,y) K^{M}(y)   Y(y) dy.
\label{vol-int-2}
\end{equation}
Comparing this with \eqref{vol-int-lead}, we obtain
\begin{equation}
[Y-\tilde{Y}](z)+ \int_{\Omega} \Phi_{\kappa_0}(z,y) K^{M}(y) \overline{\mathbf{C}}_{lead}  [Y-\tilde{Y}](y) dy = - \mathcal{O}(a^{1-\gamma}) \int_{\Omega} \Phi_{\kappa_0}(z,y) K^{M}(y)   Y(y) dy,
\label{vol-int-3}
\end{equation}
whence, using the invertibility of the integral equation \eqref{vol-int-3} and the fact that $\norm{Y}_{L^{2}(\Omega)}=\mathcal{O}(1) $, we derive 
\[\norm{Y-\tilde{Y}}_{L^2(\Omega)}=\mathcal{O}(a^{1-\gamma}) .\] 
Therefore
\begin{equation}
\begin{aligned}
u_{a}^{\infty}(\hat{x},\theta)-u^{\infty}_{lead}(\hat{x},\theta)
&=- \int_{\Omega} e^{-i\kappa_{0} \hat{x}\cdot y} K^{M}(y)\overline{\mathbf{C}}_{lead} [Y-\tilde{Y}](y) dy - \mathcal{O}(a^{1-\gamma})\int_{\Omega} e^{-i\kappa_{0} \hat{x}\cdot y} K^{M}(y) Y(y) dy\\
&=\mathcal{O}\left(a^{1-\gamma} \right),
\end{aligned}
\notag
\end{equation}
and hence
\begin{equation}
\begin{aligned}
u^{\infty}(\hat{x},\theta)-u^{\infty}_{lead}(\hat{x},\theta)
&=\mathcal{O}\left(a^{1-\gamma}+a^{\frac{s\lambda}{3}}+a^{2-s}+a^{3-\gamma-2t-s}+a^{s-t} \right).
\end{aligned}
\label{gamma-small-vol}
\end{equation}
Proceeding similarly, when  $\gamma=1, \gamma+s=2 $ with the frequency $\omega $ away from the Minnaert resonance, we can derive
\begin{equation}
\begin{aligned}
u_{a}^{\infty}(\hat{x},\theta)-u^{\infty}_{lead}(\hat{x},\theta)
&=\mathcal{O}\left(a^{2} \right),
\end{aligned}
\notag
\end{equation}
and hence
\begin{equation}
\begin{aligned}
u^{\infty}(\hat{x},\theta)-u^{\infty}_{lead}(\hat{x},\theta)
&=\mathcal{O}\left(a^{2}+a^{\frac{s\lambda}{3}}+a^{2-s}+a^{3-\gamma-2t-s}+a^{s-t} \right)\\
&=\mathcal{O}\left(a^{\frac{s\lambda}{3}}+a^{2-s}+a^{3-\gamma-2t-s}+a^{s-t} \right).
\end{aligned}
\label{gamma-away-vol}
\end{equation}
Now suppose that $\gamma=1$ and $\omega $ is near the Minnaert resonance, i.e. $1-\frac{\omega^2_M}{\omega^2}=l_M a^{h_1}$, with $l_M \neq 0$ and $h_1 \in (0, 1)$ where $s$ and $t$ satisfy the conditions \[\ s=1-h_1 \mbox{ and } \frac{s}{3} \leq t<\min\{1-h_1,\frac{1}{2}\}.\]
In this case, if we notice that 
\begin{equation}
\omega^2-\overline{\omega}_{M}^{2}=\omega^2 l_{M} a^{h_1}+ \underbrace{(\omega^2-\omega_{M}^{2})}_{\mathcal{O}(a^2)},
\notag
\end{equation}
and use the fact that $s+h_1=1 $ in \eqref{est-vol-17}, we can derive
\begin{equation}
\begin{aligned}
u^{\infty}(\hat{x},\theta)-u^{\infty}_{a}(\hat{x},\theta)
&=\mathcal{O}\left(a^{h_1}+a^{\frac{(1-h_{1})\lambda}{3}}+a^{1-h_{1}}+a^{1-2t}+a^{1-h_{1}-t} \right).
\end{aligned}
\label{gamma-near-vol}
\end{equation}
\end{remark}
\section{The case of metasurfaces}\label{Section-surface}
%\section{A brief description of the proof}\label{Proof-brief-description}
%We define 
%$Q_{l}:= \int_{\partial D_l} \phi_{l}(s)\;d\sigma_{l}(s)$ 
%and
%$V_l :=\int_{\partial D_l}(s-z_l)\phi_{l}(s)\;d\sigma_{l}(s).$
In this section, we deal with the case of surface distribution of the gas bubbles. As in the case of volumetric distribution, we first study the surface integral equations corresponding to the equivalent scattering problem and then compare this equivalent problem to the original scattering problem.
\subsection{The surface integral equations and corresponding estimates}
Let us consider the scattering problem
\begin{equation}
\begin{aligned}
&\left(\Delta +\kappa_{0}^{2} \right)u^{t}_{a}=0,\ \text{in} \ \mathbb{R}^{3}\setminus \Sigma,
\end{aligned}
\label{scat-3a}
\end{equation}
\begin{equation}
[u^{t}_{a}]=0, \ \left[\frac{\partial u^{t}_{a}}{\partial \nu} \right]-h_{*} \sigma u^{t}_{a}=0, \ \text{on} \ \Sigma,
\label{scat-3b}
\end{equation}
\begin{equation}
\begin{aligned}
&u^{t}_{a}=u^{s}_{a}+e^{i \kappa_{0} x \cdot \theta},
\end{aligned}
\label{scat-3c}
\end{equation}
\begin{equation}
\begin{aligned}
&\frac{\partial u^{s}_{a}}{\partial \vert x \vert}-i \kappa_{0} u^{s}_{a} =o\left(\frac{1}{\vert x \vert} \right), \ \vert x \vert\rightarrow \infty,
\end{aligned}
\label{scat-3d}
\end{equation}
and the surface integral equation
\begin{equation}
\begin{aligned}
Y(z)+h_{*} \int_{\Sigma} \Phi_{\kappa_0}(z,y)\sigma(y) Y(y) ds(y)=u^{I}(z),\ z \in \Sigma,
\end{aligned}
\label{main-blow-1}
\end{equation}
where $h_{*} $ is a positive real number and $\sigma=K^{M}\overline{\mathbf{C}} $. We note that $u^{t}_{a} $ is a solution of \eqref{scat-3a}-\eqref{scat-3c} if and only if $Y:=u^{t}_{a}\Big\vert_{\Sigma} $ satisfies \eqref{main-blow-1}. In addition, \[u^{t}_{a}(x)=u^{I}(x)- h_{*} \int_{\Sigma} \Phi_{\kappa_0}(x,y)\sigma(y) Y(y) ds(y),\ x \in \mathbb{R}^{3}.\]
Recall that we are interested in the cases $ h_{*}=\mathcal{O}(1)$ and $h_{*}=a^{1-h_1-s},\ s+h_1>1 $ as $a \rightarrow 0 $.\\
For $\vert s \vert \leq \alpha+1 $ and $\Gamma $ of class $C^{\alpha,1} $, we recall the definitions (see also \cite{HW},\cite{MPS})
\begin{equation}
\begin{aligned}
H^{s}(\Sigma)&:=\{f \vert_{\Sigma} : f \in H^{s}(\Gamma) \} , \ H^{s}_{\overline{\Sigma}}(\Gamma):=\{f \in H^{s}(\Gamma): \text{supp}\ f \subseteq \overline{\Sigma} \} , \\
H^{-s}_{\Sigma}(\Gamma)&:=\{\phi \in H^{-s}(\Gamma): \langle \phi, \psi \rangle_{-s,s}=0,\ \text{for any}\  \psi \in H^{s}_{\Gamma \setminus \Sigma}(\Gamma)\} . 
\end{aligned}
\notag
\end{equation}
The following lemma shows the existence and uniqueness of the solution $Y$ to the surface integral equation \eqref{main-blow-1}. 
\begin{lemma}\label{Lipp-sur}
The surface integral equation \eqref{main-blow-1} is invertible from $L^{2}(\Sigma) $ to itself and the solution $Y$ belongs to $W^{1,p}(\Sigma) $ for $p \in (1,\infty) $.
\end{lemma}
\begin{proof}
Let us denote $\sigma_{h}:=h_{*}\sigma $ 
and first consider the case when $h_{*}=\mathcal{O}(1) $.\\
We shall prove that the integral equation \eqref{main-blow-1} is invertible from $L^{2}(\Sigma) $ to itself. To see this, let us consider the operator $B$ defined by
\begin{equation}
f \mapsto Bf:= \int_{\Sigma} \sigma_{h}(z)\Phi_{\kappa_0}(x,z) f(z) ds(z) .
\notag
\end{equation}
Using the mapping properties of the single layer operator, it follows that $B$ maps $L^{2}(\Sigma) $ to $ H^{1}(\Sigma) $. Note that here $H^{1}(\Sigma) $ is defined as the restriction of functions in $H^{1}(\Gamma) $ to $\Sigma $, where $\Gamma $ is a smooth, closed surface such that $\Sigma \subset \Gamma $.\\
Now since the embedding $H^{1}(\Sigma) \hookrightarrow L^{2}(\Sigma) $ is compact, it follows that
\begin{equation}
Id+B: L^{2}(\Sigma) \rightarrow L^{2}(\Sigma)
\notag
\end{equation}
is Fredholm of index zero. It is therefore sufficient to prove that $\ker(I+B)=\{0 \} $.\\
In this direction, let us assume that $(I+B)f=0 $ for some $f \in L^{2}(\Sigma) $, and consider the function 
\begin{equation}
w(x):= -\int_{\Sigma} \sigma_{h}(z)\Phi_{\kappa_0}(x,z) f(z) ds(z), \ x \in \mathbb{R}^{3}\setminus \overline{\Sigma}.
\notag
\end{equation}
The function $w$ so defined satisfies 
\begin{equation}
(\Delta+\kappa_{0}^{2})w=0, \ \text{in} \ \mathbb{R}^{3} \setminus \overline{\Sigma},
\notag
\end{equation}
and the Sommerfeld conditions. Furthermore 
\begin{equation}
[w]_{\Sigma}=0, \ \text{and}\ [\partial_{\nu}w]_{\Sigma}=\sigma_{h}(x)f(x)=\sigma_{h}(x)w(x).
\notag
\end{equation}
Now let $D $ be a connected and open subset of $\mathbb{R}^{3} $ such that $\partial D=\Gamma $. Then integrating in $\Omega_{r}:=B_{r} \setminus D $ we have
\begin{equation}
\int_{\partial B_r} w^{+} \frac{\partial \overline{w^{+}}}{\partial \nu}=\int_{\Sigma} w^{+} \frac{\partial \overline{w^{+}}}{\partial \nu}  -\kappa_{0}^{2} \int_{\Omega_{r}} \vert w^{+} \vert^2  +\int_{\Omega_{r}} \vert \nabla w^{+} \vert^2.
\label{e1}
\end{equation}
Similarly, integrating in $D $, we obtain 
\begin{equation}
\int_{\Sigma} w^{-} \frac{\partial \overline{w^{-}}}{\partial \nu}=  -\kappa_{0}^{2} \int_{\Omega_{r}} \vert w^{-} \vert^2  +\int_{\Omega_{r}} \vert \nabla w^{-} \vert^2.
\label{e2}
\end{equation}
Adding \eqref{e1} and \eqref{e2} and using the jump relations satisfied by $w$, we conclude that
\begin{equation}
\int_{\partial B_r} w^{+} \frac{\partial \overline{w^{+}}}{\partial \nu}=+\int_{\Sigma} \overline{\sigma}_{h}(x) \vert w \vert^2 -\kappa_{0}^{2} \int_{\Omega_{r} \cup D} \vert w \vert^2  +\int_{\Omega_{r} \cup D} \vert \nabla w \vert^2 ,
\label{e3}
\end{equation}
whence it follows that if $\Im\ \overline{\sigma}_{h}\geq 0 $, then $\Im \int_{\partial B_r} w^{+} \frac{\partial \overline{w^{+}}}{\partial \nu} \geq 0 $. Now recall that $\sigma_h $ is a real-valued function in our case, and therefore the fact that $\Im \int_{\partial B_r} w^{+} \frac{\partial \overline{w^{+}}}{\partial \nu} \geq 0 $ indeed holds true.\\
We can now apply Rellich lemma and conclude that $w^{+}=0$ in $\mathbb{R}^{3} \setminus \overline{D}$ whence the fact $f=0$ follows.
Thus we have shown that \eqref{main-blow-1} is invertible and the solution $Y$ to \eqref{main-blow-1} belongs to $L^2(\Sigma)$. \\
Next using this fact together with the smoothing property of the single layer operator, it follows that $Y \in H^{1}(\Sigma) $.\\
The Sobolev embedding result $H^{1}(\Sigma) \hookrightarrow L^{p}(\Sigma),\ p \in (1,+\infty) $ further implies that the function $Y \in  L^{p}(\Sigma),\ p \in (1,+\infty)$.\\
Now if we use \eqref{main-blow-1} and the fact that the single layer operator maps $L^{p}(\Sigma) $ into $W^{1,p}(\Sigma),\ p\in (1,\infty) $ (see \cite{MM}), we can conclude that the solution $Y$ to the integral equation \eqref{main-blow-1} lies in the class $W^{1,p}(\Sigma),\ p\in (1,+\infty) $.\\  
The case when $h_{*}=a^{1-h_1-s}, \ s+h_1>1 $ follows by a similar argument since the function $\sigma $ and hence $\sigma_h $ is positive.
\end{proof}
The next result provides us with an estimate of the $L^2$ norm of $Y$ in terms of the parameter $h_{*}=a^{1-s-h_1}$ in the regime $s+h_1>1 $. Let us define the semi-classical parameter $h:=a^{\frac{s+h_1-1}{2}} $. 
\begin{lemma}\label{est-sur}
Assume that $\kappa_{0}^{2} $ is not an eigenvalue for the Dirichlet Laplacian in $D$. Then the solution $Y$ to the surface integral equation \eqref{main-blow-1} satisfies the estimate
\begin{equation}
\norm{Y}_{L^{2}(\Sigma)}=\mathcal{O}(h).
\label{est-sur-Y}
\end{equation}
\end{lemma}
\begin{proof}
As a first step, we observe that the scattering problem \eqref{scat-3a}-\eqref{scat-3d} can be transformed into the equivalent boundary value problem 
\begin{equation}
\begin{aligned}
&(\Delta +\kappa_{0}^2)Y =0,\ \text{in} \ B_{R}\setminus \Sigma, \\
&[Y]=0,\ \left[\frac{\partial Y}{\partial \nu} \right]- h^{-2} \sigma Y=0, \ \text{on} \ \Sigma, \\
&\left[\frac{\partial Y}{\partial \nu} \right]-TY= \frac{\partial u^{I}}{\partial \nu} -Tu^{I} ,\ \text{on} \ \partial B_{R},
\end{aligned}
\label{e4}
\end{equation}
where $T: H^{\frac{1}{2}}(\partial B_R) \rightarrow H^{-\frac{1}{2}}(\partial B_R) $ is the Dirichlet to Neumann (D-N) map for the exterior problem on $\mathbb{R}^{3}\setminus B_R $ (see \cite{CK}).\\
Proceeding as in \eqref{e1}-\eqref{e3}, we derive
\begin{equation}
\int_{B_R} \vert \nabla Y \vert^2 -\kappa_{0}^{2} \int_{B_R} \vert Y \vert^2 + h^{-2}\int_{\Sigma} \overline{\sigma}  \vert Y \vert^2 -\langle TY,Y \rangle_{-\frac{1}{2},\frac{1}{2}}=\left\langle  \frac{\partial u^{I}}{\partial \nu}-Tu^{I},Y \right\rangle_{-\frac{1}{2},\frac{1}{2}}.
\label{e5}
\end{equation}
Now the operator $T$ can be decomposed as $T=T_{0}+T_{1} $ (see \cite{CK}) where 
\begin{equation}
\langle -T_{0}Y,Y \rangle_{-\frac{1}{2},\frac{1}{2}} \geq C \norm{Y}^{2}_{\frac{1}{2}},
\label{e6}
\end{equation}
and $T_{1} $ is smoothing and maps from $H^{\frac{1}{2}}(\partial B_R) $ to $H^{\frac{1}{2}}(\partial B_R) $.\\
Also we can write
\begin{equation}
\begin{aligned}
\left\vert \langle T_{1}Y,Y \rangle_{-\frac{1}{2},\frac{1}{2}} \right\vert &= \left\vert \langle T_{1}Y,Y \rangle_{L^2,L^2} \right\vert \\
&\leq \norm{T_{1} Y}_{L^{2}(\partial B_R)} \norm{ Y}_{L^{2}(\partial B_R)}\\
&\leq  \norm{T_{1} Y}_{H^{\frac{1}{2}}(\partial B_R)} \norm{ Y}_{L^{2}(\partial B_R)}\\
&\lesssim  \norm{ Y}_{H^{\frac{1}{2}}(\partial B_R)} \norm{ Y}_{L^{2}(\partial B_R)}\\
&\lesssim \epsilon \norm{ Y}^{2}_{H^{\frac{1}{2}}(\partial B_R)}+\frac{1}{4 \epsilon} \norm{ Y}^{2}_{L^{2}(\partial B_R)} \\
&\lesssim \epsilon \norm{ Y}^{2}_{H^{\frac{1}{2}}(\partial B_R)}+\frac{1}{4 \epsilon} \norm{ Y}^{2}_{H^{s}(\partial B_R)},\ s\in (0,\frac{1}{2})  \\
&\lesssim \epsilon \norm{ Y}^{2}_{H^{1}( B_R)}+\frac{1}{4 \epsilon} \norm{ Y}^{2}_{H^{s+\frac{1}{2}}( B_R)} \\
&\stackrel{interpolation}{\lesssim}  \epsilon \norm{ Y}^{2}_{H^{1}( B_R)}+\frac{1}{4 \epsilon} \left[ \norm{ Y}^{2(1-s-\frac{1}{2})}_{L^{2}( B_R)} \norm{ Y}^{2(s+\frac{1}{2})}_{H^{1}( B_R)} \right]\\
%&\lesssim \epsilon \norm{ u}^{2}_{H^{1}( B_R)}+ \frac{1}{4\epsilon} \left[ \epsilon^{2} \norm{ u}^{2}_{H^{1}( B_R)}+\frac{1}{4\epsilon^2} \norm{ u}^{2}_{L^{2}( B_R)}  \right]\\
&\lesssim \frac{5 \epsilon}{4} \norm{ Y}^{2}_{H^{1}( B_R)} + C(\epsilon) \norm{ Y}^{2}_{L^{2}( B_R)},
\end{aligned}
\label{e7}
\end{equation}
where $\epsilon \in (0,1) $ is chosen to be sufficiently small.\\
Now using the surface integral equation and the fact that $\mathbf{S}_{\kappa_0}:H^{-1}(\Sigma) \rightarrow H^{\frac{1}{2}}(B_R) $, we can write
\begin{equation}
\begin{aligned}
\norm{Y}_{L^2(B_R)}&\leq \norm{h^{-2} \mathbf{S}_{\kappa_0}\left[\sigma Y\right]}_{L^{2}(B_R)} +\mathcal{O}(1)\\
&\leq h^{-2} \norm{ \mathbf{S}_{\kappa_0}\left[\sigma Y\right]}_{H^{\frac{1}{2}}(B_R)} +\mathcal{O}(1)\\
&\lesssim h^{-2} \norm{\sigma Y}_{H^{-1}(\Sigma)} +\mathcal{O}(1).
\end{aligned}
\label{e8}
\end{equation}
Again we can re-write \eqref{main-blow-1} as $\sigma Y=\mathbf{S}_{\kappa_0}^{-1} \left[-h^2 Y+h^2 u^I \right] $ and therefore using the fact (see theorem \ref{Sing-inv}) that \[\mathbf{S}_{\kappa_0}^{-1}: L^2(\Sigma) \rightarrow H^{-1}(\Sigma) ,\] from \eqref{e8}, we obtain
\begin{equation}
\begin{aligned}
\norm{Y}_{L^2(B_R)}&\lesssim  \norm{ Y}_{L^2(\Sigma)} +\mathcal{O}(1).
\end{aligned}
\label{e9}
\end{equation}
Next using \eqref{e7} and \eqref{e9} in \eqref{e5}, we can write
\begin{equation}
\begin{aligned}
&\int_{B_R} \vert \nabla Y \vert^2 -\kappa_{0}^{2} \int_{B_R} \vert Y \vert^2 + h^{-2}\int_{\Sigma} \overline{\sigma}  \vert Y \vert^2 +\langle -T_{0}Y,Y \rangle_{-\frac{1}{2},\frac{1}{2}}\\
&\qquad =\left\langle  \frac{\partial u^{I}}{\partial \nu}-Tu^{I},Y \right\rangle_{-\frac{1}{2},\frac{1}{2}}+\langle T_{1}Y,Y \rangle_{-\frac{1}{2},\frac{1}{2}}\\
&\qquad \leq C \norm{Y}_{H^{\frac{1}{2}}(\partial B_R)}+\frac{5C\epsilon}{4} \norm{Y}_{H^{1}(B_R)}^{2}+C(\epsilon) \norm{Y}_{L^2(B_R)}^{2}\\
&\qquad \lesssim C \norm{Y}_{H^{\frac{1}{2}}(\partial B_R)}+\frac{5C\epsilon}{4} \norm{Y}_{H^{1}(B_R)}^{2}+C(\epsilon) \norm{Y}_{L^2(\Sigma)}^{2}+\mathcal{O}(1).
\end{aligned}
\label{e10}
\end{equation}
Also using \eqref{e6}, we can write
\begin{equation}
\begin{aligned}
&\int_{B_R} \vert \nabla Y \vert^2 -\kappa_{0}^{2} \int_{B_R} \vert Y \vert^2 + h^{-2}\int_{\Sigma} \overline{\sigma}  \vert Y \vert^2 +\langle -T_{0}Y,Y \rangle_{-\frac{1}{2},\frac{1}{2}} \\
&\qquad \geq \int_{B_R} \vert \nabla Y \vert^2 -\kappa_{0}^{2} \int_{B_R} \vert Y \vert^2 + h^{-2} \int_{\Sigma} \overline{\sigma}  \vert Y \vert^2 +C \norm{Y}_{H^{\frac{1}{2}}(\partial B_R)}^{2}.
\end{aligned}
\label{e11}
\end{equation}
From \eqref{e9}-\eqref{e11}, we obtain
\begin{equation}
\begin{aligned}
\left(1-\frac{5C \epsilon}{4} \right) \norm{\nabla Y}_{L^{2}(B_R)}^{2}&+ \left(-\kappa_{0}^2+Ch^{-2}-\frac{5C\epsilon}{4}-C(\epsilon) \right) \norm{Y}_{L^2(\Sigma)}^{2} + C \norm{Y}_{H^{\frac{1}{2}}(\partial B_R)}^{2} \\
& \leq C \norm{Y}_{H^{\frac{1}{2}}(\partial B_R)}+\mathcal{O}(1) \leq C \norm{Y}_{H^{1}(B_R)}+\mathcal{O}(1).
\end{aligned}
\label{e12}
\end{equation}
Now from the fact $\mathbf{S}_{\kappa_0}:H^{-\frac{1}{2}}(\Sigma) \rightarrow H^{1}(B_R) $, it follows that
\begin{equation}
\begin{aligned}
\norm{Y}_{ H^{1}(B_R)} \leq h^{-2} \norm{\mathbf{S}_{\kappa_0}[\sigma Y]}_{ H^{1}(B_R)} +\mathcal{O}(1) \leq h^{-2} \norm{\sigma Y}_{H^{-\frac{1}{2}}(\Sigma)}+\mathcal{O}(1) &\leq h^{-2} \norm{\sigma Y}_{L^{2}(\Sigma)}+\mathcal{O}(1) \\
&\lesssim h^{-2} \norm{ Y}_{L^{2}(\Sigma)}+\mathcal{O}(1).
\end{aligned}
\notag
\end{equation} 
Using this in \eqref{e12}, we deduce that $\norm{ Y}_{L^{2}(\Sigma)}=\mathcal{O}(1) $ and therefore using \eqref{e9}, we obtain $\norm{ Y}_{L^{2}(B_R)}=\mathcal{O}(1) $.\\
Also from \eqref{e12}, we can write
\begin{equation}
\norm{\nabla Y}_{L^{2}(B_R)}^{2}\leq C \norm{Y}_{L^{2}(B_R)} + C \norm{\nabla Y}_{L^{2}(B_R)} +\mathcal{O}(1) \leq C \norm{\nabla Y}_{L^{2}(B_R)} + \mathcal{O}(1),
\notag
\end{equation}
whence it follows that $\norm{\nabla Y}_{L^{2}(B_R)}=\mathcal{O}(1) $.\\
Now using the fact that $\norm{ Y}_{H^{1}(B_R)}=\mathcal{O}(1) $ in \eqref{e12}, we deduce that $\norm{ Y}_{L^{2}(\Sigma)}=\mathcal{O}(h) $.
\end{proof}
Using the estimate $\eqref{est-sur-Y} $, it is easy to see that 
\begin{equation}
 Y^{\infty}(\hat{x},\theta)-Y^{\infty}_{D}(\hat{x},\theta)=\mathcal{O}(h),
 \label{sur-est-2000}
 \end{equation} 
 where 
\begin{itemize}
\item if $\Sigma $ is an open surface, $Y_{D}$ is the unique solution to the Dirichlet crack problem 
\begin{equation}
\left(\Delta+\kappa_{0}^{2} \right)Y_{D}=0, \ \text{in}\ \mathbb{R}^{3}\setminus \Sigma,
\label{scat-4a}
\end{equation}
\begin{equation}
Y_{D}=0,\ \text{on}\ \Sigma,
\label{scat-4b}
\end{equation}
with the Sommerfeld radiation conditions satisfied by $Y_{D}-u^{I} $, and
\item if $\Sigma=\partial D $ for some connected open subset $D \subset \mathbb{R}^3$, 
then $Y_{D}$ is the unique solution to the exterior Dirichlet problem 
\begin{equation}
\left(\Delta+\kappa_{0}^{2} \right)Y_{D}=0, \ \text{in}\ \mathbb{R}^{3}\setminus \overline{D},
\label{scat-5a}
\end{equation}
\begin{equation}
Y_{D}=0,\ \text{on}\ \partial D,
\label{scat-5b}
\end{equation} 
with the Sommerfeld radiation conditions satisfied by $Y_{D}-u^{I} $.
\end{itemize}
Next using lemma \ref{Lipp-sur} and lemma \ref{est-sur}, we deduce the following estimates for the solution $Y$ and its gradient $\nabla Y $.
\begin{proposition}\label{blow-est-sur}
The solution $Y$ to the surface integral equation \eqref{main-blow-1} satisfies the following estimates.\\
\begin{itemize}
\item If $h_{*}=\mathcal{O}(1),\ a \rightarrow 0 $, then
\begin{equation}
\norm{Y}_{L^{\infty}(\Sigma)}=\mathcal{O}(1),\ \norm{Y}_{W^{1,p}(\Sigma)}=\mathcal{O}(1). 
\label{Y-non-blow}
\end{equation}
\item If $h_{*}=a^{1-h_1-s}$, $s+h_1>1 $, then
\begin{equation}
\norm{Y}_{L^{\infty}(\Sigma)}=\mathcal{O}(a^{\frac{3}{2}(1-s-h_1)}),\ \norm{Y}_{W^{1,p}(\Sigma)}=\mathcal{O}(a^{\frac{3}{2}(1-s-h_{1})} ). 
\label{Y-blow}
\end{equation}
\end{itemize}
\end{proposition}
\begin{proof}
\begin{itemize}
\item If $h_{*}=\mathcal{O}(1) $, the estimates \eqref{Y-non-blow} follow directly from the invertibility of the surface integral equation \eqref{main-blow-1} established in lemma \ref{Lipp-sur}. 
\item If $h_{*}=a^{1-h_1-s}$, $s+h_1>1 $, we use the $L^2$ estimate for $Y$ established in lemma \ref{est-sur} as follows. Recall that $h=a^{\frac{s+h_1-1}{2}} $. \\
Now using \eqref{main-blow-1} and \eqref{est-sur-Y}, we can deduce
\begin{equation}
\begin{aligned}
\norm{Y}_{H^{1}(\Sigma)}&\leq \norm{u^{I}}_{H^{1}(\Sigma)}+h^{-2} \norm{\mathbf{S}_{\kappa_0}[\sigma Y]}_{H^{1}(\Sigma)}\\
&\leq \norm{u^{I}}_{H^{1}(\Sigma)}+C h^{-2} \norm{Y}_{L^2(\Sigma)}= \mathcal{O}(h^{-1}),
\end{aligned}
\notag
\end{equation}
and therefore for $p \in (1,\infty) $,
\begin{equation}
\begin{aligned}
\norm{Y}_{W^{1,p}(\Sigma)} &\leq \norm{u^{I}}_{W^{1,p}(\Sigma)}+h^{-2} \norm{\mathbf{S}_{\kappa_0}[\sigma Y]}_{W^{1,p}(\Sigma)}\\
&\leq \norm{u^{I}}_{W^{1,p}(\Sigma)}+C h^{-2} \norm{\sigma Y}_{L^p(\Sigma)}\\
&\leq \norm{u^{I}}_{W^{1,p}(\Sigma)}+C h^{-2} \norm{Y}_{L^p(\Sigma)},\ \text{since} \ \sigma \ \text{is bounded}\\
&\leq \norm{u^{I}}_{W^{1,p}(\Sigma)}+C h^{-2} \norm{Y}_{H^1(\Sigma)}=\mathcal{O}(h^{-3})=\mathcal{O}(a^{\frac{3}{2}(1-s-h_{1})}).
\end{aligned}
\notag
\end{equation}
Since $Y \in W^{1,p}(\Sigma) $ for $p>2 $, it follows that $Y \in C^{0,\eta}(\bar{\Sigma}), \eta=1-\frac{2}{p} $  and hence we can derive the estimate 
\[\norm{Y}_{L^{\infty}(\Sigma)}\leq C \norm{Y}_{W^{1,p}(\Sigma)}=\mathcal{O}(a^{\frac{3}{2}(1-s-h_1)}) . \]
\end{itemize}
\end{proof}
\subsection{The asymptotic approximations}
As in the case of volumetric distributions, we now describe the proof of the results in the following regime where $\gamma=1, 1<s+h_1<\min \{ \frac{3}{2}-t,2-h_1\} $ and when the frequency is near the resonance with $l_M >0 $. The proofs of the other cases follow similarly.\\
In this case, as in \eqref{est-vol-0}, the algebraic system \eqref{LAS-1-theorem} can be rewritten as
\begin{equation}
Y_m+\sum_{{j=1} \atop {j\neq m}}^{M} \Phi_{\kappa_0}(z_m,z_j) \overline{\mathbf{C}} Y_j a^{1-h_1}= u^{I}(z_m),
\label{est-sur-0}
\end{equation}
where $Y_j=-\mathbf{C}^{-1} Q_j,\ \mathbf{C}=\overline{\mathbf{C}} a^{1-h_1},\ j=1,\dots,M $.\\
To compare this with the surface integral equation
\begin{equation}
Y(z)+a^{1-h_1-s} \int_{\Sigma} \Phi_{\kappa_0}(z,y) K^{M}(y) \overline{\mathbf{C}} Y(y) ds(y)=u^{I}(z),
\label{est-sur-0a}
\end{equation}
for $m=1,\dots,M $, we rewrite \eqref{est-sur-0a} as
\begin{equation}
\begin{aligned}
&Y(z_m)+a^{1-s-h_1}\sum_{{j=1} \atop {j\neq m}}^{M} \Phi_{\kappa_0}(z_m,z_j) \overline{\mathbf{C}} Y(z_j) a^{s}\\
&= u^{I}(z_m)+a^{1-s-h_1} \underbrace{ \left[ \sum_{{j=1} \atop {j \neq m}}^{M}  \Phi_{\kappa_0}(z_m,z_j) \overline{\mathbf{C}} Y(z_j) a^{s}
 - \sum_{{j=1} \atop {j \neq m}}^{[a^{-s}]}  \Phi_{\kappa_0}(z_m,z_j) K^M(z_j) \overline{\mathbf{C}} Y(z_j) \vert \Sigma_{j} \vert \right]}_{A_2}\\
&\quad +a^{1-s-h_1} \underbrace{ \left[ \sum_{{j=1} \atop {j \neq m}}^{[a^{-s}]}  \Phi_{\kappa_0}(z_m,z_j) K^M(z_j) \overline{\mathbf{C}} Y(z_j) \vert \Sigma_{j} \vert - \int_{{\cup_{{j=1} \atop {j \neq m}}^{[a^{-s}]}}\Sigma_j}  \Phi_{\kappa_0}(z_m,y) K^M(y) \overline{\mathbf{C}} Y(y) ds(y) \right]}_{B_2} \\
&\quad -a^{1-s-h_1}\underbrace{\int_{\Sigma_m}  \Phi_{\kappa_0}(z_m,y) K^M(y) \overline{\mathbf{C}} Y(y) ds(y)}_{C_2} -a^{1-s-h_1}\underbrace{\int_{\Sigma \setminus \cup_{j=1 }^{[a^{-s}]} \Sigma_j}  \Phi_{\kappa_0}(z_m,y) K^M(y) \overline{\mathbf{C}} Y(y) ds(y)}_{D_2} .
\end{aligned}
\label{est-sur-1}
\end{equation}
\begin{remark}
We note that in case $\Sigma $ is parametrized by more than one chart, we would need to additionally estimate the integral over the part (image of the chart) that doesn't contain the point $z_m $. But this part being away from $z_m $ would imply that the fundamental solution is smooth and therefore the error estimate for this integral would only be better than the ones mentioned above in the splitting. \qed
\end{remark}
The terms $A_2,B_2,C_2 $ and $D_2$ can be estimated by closely following the arguments in the case of volumetric distributions.\\
Let us begin with the term $C_2$. We recall that by our assumption, $\Sigma_m $ is contained in a single chart and hence it is easy to observe (using the local co-ordinates, if necessary) that for $r<\frac{1}{2} a^{\frac{s}{2}} $, the image of the ball $B(z_m,r) $ of radius $r$ is contained in $\Sigma_m $. By an abuse of notation, we shall identify between $B(z_m,r) $ and its representation in the local chart. Therefore using \eqref{Y-blow} and continuing as in the proof of \eqref{est-vol-3}, we can write
\begin{equation}
\begin{aligned}
\vert C_2 \vert &=\left\vert \int_{\Sigma_m} \Phi_{\kappa_0}(z_m,y) K^{M}(y) \overline{\mathbf{C}} Y(y) ds(y) \right\vert \\
&\qquad \leq \norm{K^M}_{L^{\infty}(\Sigma)} \vert \overline{\mathbf{C}} \vert \norm{Y}_{L^{\infty}(\Sigma)} \left\vert \int_{\Sigma_m} \Phi_{\kappa_0} (z_m,y) ds(y) \right\vert \\
%&\qquad \leq C \ a^{\frac{3}{2}(1-s-h_{1})} \left\vert \int_{\Sigma_m} \Phi_{\kappa_0} (z_m,y) ds(y) \right\vert\\
&\qquad \leq \frac{C}{4 \pi} a^{\frac{3}{2}(1-s-h_{1})} \Big(\int_{B(z_m,r)} \frac{1}{\vert z_m-y \vert} ds(y) + \int_{\Sigma_m \setminus B(z_m,r)} \frac{1}{\vert z_m-y \vert} ds(y) \Big), \ \text{where}\ r<\frac{1}{2} a^{\frac{s}{2}}\\
&\qquad \leq \frac{C}{4 \pi} a^{\frac{3}{2}(1-s-h_{1})} \left(2 \pi r+\frac{1}{r} \left[a^s- \pi r^2 \right] \right) \leq \frac{C}{4 \pi} a^{\frac{3}{2}(1-s-h_{1})} a^{\frac{s}{2}},
\end{aligned}
\notag
\end{equation}
Hence, we deduce that
\begin{equation}
 C_2  = \mathcal{O}(a^{\frac{3}{2}(1-s-h_1)+\frac{s}{2}}).
\label{est-sur-2}
\end{equation}
Let us now estimate the term $B_2$. To do so, we split it into two parts $B_{2,1} $ and $B_{2,2} $ given by
\begin{equation}
\begin{aligned}
B_{2} &=\sum_{{j=1} \atop {j \neq m}}^{[a^{-s}]}  \Phi_{\kappa_0}(z_m,z_j) K^M(z_j) \overline{\mathbf{C}} Y(z_j) \vert \Sigma_j \vert -\int_{{\cup_{{j=1} \atop {j \neq m}}^{[a^{-s}]}}\Sigma_j}  \Phi_{\kappa_0}(z_m,y) K^M(y) \overline{\mathbf{C}} Y(y) ds(y) \\
&=\underbrace{\sum_{{j=1} \atop {j \neq m}}^{[a^{-s}]} \int_{\Sigma_j} \Phi_{\kappa_0}(z_m,z_j) K^M(z_j) \overline{\mathbf{C}} Y(z_j) ds(y) - \sum_{{j=1} \atop {j \neq m}}^{[a^{-s}]} \int_{\Sigma_j} \Phi_{\kappa_0}(z_m,y) K^M(z_j) \overline{\mathbf{C}} Y(y) ds(y)}_{B_{2,1}}\\
&\quad +\underbrace{\sum_{{j=1} \atop {j \neq m}}^{[a^{-s}]} \int_{\Sigma_j} \Phi_{\kappa_0}(z_m,y) K^M(z_j) \overline{\mathbf{C}} Y(y) ds(y)-\sum_{{j=1} \atop {j \neq m}}^{[a^{-s}]} \int_{\Sigma_j} \Phi_{\kappa_0}(z_m,y) K^M(y) \overline{\mathbf{C}} Y(y) ds(y)}_{B_{2,2}}.
\end{aligned}
\label{est-sur-4}
\end{equation}
As in the case of volumetric distributions, relative to each $\Sigma_m $, we distinguish the other bubbles as near and far ones using squares (or quadrilaterals). In such an arrangement, the number of squares upto the $n^{th}$ layer is $(2n+1)^2, \ n=0,\dots,\left[a^{-\frac{s}{2}}\right] $ and $\Sigma_m $ is located at the centre. Hence the number of bubbles in the $n^{th} $ layer will be atmost $\left[(2n+1)^{2}-(2n-1)^{2} \right]$ and their distance from $D_m$ is more than $n \left(a^{\frac{s}{2}}-\frac{a}{2} \right) $. \\
To estimate the term $B_{2,1} $, we notice that 
\begin{equation}
\begin{aligned}
\vert B_{2,1} \vert 
&\leq \norm{K^M}_{L^{\infty}(\Sigma)} \vert \overline{\mathbf{C}} \vert \sum_{{j=1} \atop {j \neq m}}^{[a^{-s}]} \left[\int_{\Sigma_j} \left\vert \left[\Phi_{\kappa_0}(z_m,z_j)-\Phi_{\kappa_0}(z_m,y) \right] Y(z_j)\right\vert+  \int_{\Sigma_j} \vert \Phi_{\kappa_0}(z_m,y) \vert \vert Y(z_j)-Y(y)\vert\right] \\
&\leq \norm{K^M}_{L^{\infty}(\Sigma)} \vert \overline{\mathbf{C}} \vert \left[\underbrace{\sum_{{j=1} \atop {j \neq m}}^{[a^{-s}]}\int_{\Sigma_j} \left\vert \left[\Phi_{\kappa_0}(z_m,z_j)-\Phi_{\kappa_0}(z_m,y) \right] Y(z_j)\right\vert}_{I}+ \underbrace{\sum_{{j=1} \atop {j \neq m}}^{[a^{-s}]} \int_{\Sigma_j} \vert \Phi_{\kappa_0}(z_m,y) \vert \vert Y(z_j)-Y(y)\vert}_{II}\right]. 
\end{aligned}
\label{est-sur-5}
\end{equation}
Now using Morrey's inequality, $II $ can be treated as
\begin{equation}
\begin{aligned}
II &\leq \sum_{n=1}^{[a^{-\frac{s}{2}}]}  \left[(2n+1)^2-(2n-1)^2 \right] \frac{C}{n \left(a^{\frac{s}{2}}-\frac{a}{2} \right)} \int_{\Sigma_j} \vert Y(z_j)-Y(y) \vert \\
&\leq \sum_{n=1}^{[a^{-\frac{s}{2}}]}  \left[(2n+1)^2-(2n-1)^2 \right] \frac{C}{n \left(a^{\frac{s}{2}}-\frac{a}{2} \right)} \int_{\Sigma_j} \vert z_j-y \vert^{\eta}\ [Y]_{C^{0,\eta}(\overline{\Sigma})} \\
&\leq \sum_{n=1}^{[a^{-\frac{s}{2}}]}  \left[(2n+1)^2-(2n-1)^2 \right] \frac{C}{n \left(a^{\frac{s}{2}}-\frac{a}{2} \right)}  \int_{\Sigma_j} \vert z_j-y \vert^{\eta} \norm{Y}_{W^{1,p}(\Sigma)} \\
&\leq \sum_{n=1}^{[a^{-\frac{s}{2}}]}  \left[(2n+1)^2-(2n-1)^2 \right] \frac{C}{n \left(a^{\frac{s}{2}}-\frac{a}{2} \right)} a^{\frac{s}{2}\eta} a^{s} \norm{Y}_{W^{1,p}(\Sigma)}.
\end{aligned}
\label{est-sur-6}
\end{equation}
Similarly, we see that $I $ satisfies the estimate
\begin{equation}
\begin{aligned}
I &\leq \norm{Y}_{L^{\infty}(\Sigma)} \sum_{{j=1} \atop {j \neq m}}^{[a^{-s}]} \int_{\Sigma_j} \left\vert \Phi_{\kappa_0}(z_m,y)-\Phi_{\kappa_0}(z_m,z_j) \right\vert \\
&\leq C \norm{Y}_{L^{\infty}(\Sigma)} a^{\frac{s}{2}} a^{s} \sum_{n=1}^{[a^{-\frac{s}{2}}]}  \left[(2n+1)^2-(2n-1)^2 \right] \frac{C}{n \left(a^{\frac{s}{2}}-\frac{a}{2} \right)} \left[\frac{1}{n \left(a^{\frac{s}{2}}-\frac{a}{2} \right)} + \kappa_{0} \right].
\end{aligned}
\label{est-sur-7}
\end{equation}
Therefore using \eqref{Y-blow}, we can infer that
\begin{equation}
\begin{aligned}
\vert B_{2,1} \vert &\leq \norm{K^M}_{L^{\infty}(\Sigma)} \vert \overline{\mathbf{C}} \vert \sum_{n=1}^{[a^{-\frac{s}{2}}]}  \left[(2n+1)^2-(2n-1)^2 \right] a^{\frac{s}{2}} a^{s} \frac{C}{n \left(a^{\frac{s}{2}}-\frac{a}{2} \right)} \left[\frac{1}{n \left(a^{\frac{s}{2}}-\frac{a}{2} \right)} + \kappa_{0} \right] \norm{Y}_{L^{\infty}(\Sigma)} \\
&\quad +\norm{K^M}_{L^{\infty}(\Sigma)} \vert \overline{\mathbf{C}} \vert \sum_{n=1}^{[a^{-\frac{s}{2}}]}  \left[(2n+1)^2-(2n-1)^2 \right] \frac{C}{n \left(a^{\frac{s}{2}}-\frac{a}{2} \right)} a^{\frac{s}{2}\eta} a^{s} \norm{Y}_{W^{1,p}(\Sigma)} \\
%&=\mathcal{O}\left(\sum_{n=1}^{[a^{-\frac{s}{2}}]} \left[8n+2 \right] a^{\frac{3s}{2}} \frac{1}{n^2} a^{-s} \norm{Y}_{L^{\infty}(\Sigma)}+\sum_{n=1}^{[a^{-\frac{s}{2}}]} \left[8n+2 \right] a^{\frac{2s+s\eta}{2}}\frac{1}{n} a^{-\frac{s}{2}} \norm{Y}_{W^{1,p}(\Sigma)}  \right) \\
&=\mathcal{O}\left(\sum_{n=1}^{[a^{-\frac{s}{2}}]} a^{\frac{s}{2}} a^{\frac{3}{2}(1-s-h_1)} \left[\frac{8}{n}+\frac{2}{n^2} \right] +a^{\frac{s}{2}+\frac{s\eta}{2}} a^{\frac{3}{2}(1-s-h_1)} \left[8+\frac{2}{n} \right]\right)\\
& =\mathcal{O}(a^{\frac{3}{2}(1-s-h_1)+\frac{s}{2}} log \ a)+\mathcal{O}(a^{\frac{3}{2}(1-s-h_1)+\frac{s}{2}})+\mathcal{O}\left(a^{\frac{s \eta}{2}+\frac{3}{2}(1-s-h_1)} \right)+\mathcal{O}(a^{\frac{s}{2}+\frac{s\eta}{2}+\frac{3}{2}(1-s-h_1)} log \ a).
\end{aligned}
\label{est-sur-8}
\end{equation}
Similarly using the fact the function $K \in C^{0,\lambda}(\overline{\Sigma}) $, we can deduce 
\begin{equation}
\begin{aligned}
\vert B_{2,2} \vert &\leq  \norm{Y}_{L^{\infty}(\Sigma)} \vert \overline{\mathbf{C}} \vert \sum_{{j=1} \atop {j \neq m}}^{[a^{-s}]} \int_{\Sigma_j}  \left\vert \Phi_{\kappa_0}(z_m,y) \right\vert \left\vert K^M(z_j)-  K^M(y) \right\vert dy \\
&\leq \norm{Y}_{L^{\infty}(\Sigma)} \vert \overline{\mathbf{C}} \vert \sum_{n=1}^{[a^{-\frac{s}{2}}]}  \left[(2n+1)^2-(2n-1)^2 \right] \frac{C}{n \left(a^{\frac{s}{2}}-\frac{a}{2} \right)} \int_{\Sigma_j} \vert z_j-y \vert^{\lambda} [K]_{C^{0,\lambda}(\overline{\Sigma})} dy \\
&\leq \norm{Y}_{L^{\infty}(\Sigma)} \vert \overline{\mathbf{C}} \vert [K]_{C^{0,\lambda}(\overline{\Sigma})} \sum_{n=1}^{[a^{-\frac{s}{2}}]}  \left[(2n+1)^2-(2n-1)^2 \right] \frac{C}{n \left(a^{\frac{s}{2}}-\frac{a}{2} \right)} a^{s} a^{\frac{s}{2}\lambda} \\
&=\mathcal{O}\left(a^{\frac{3}{2}(1-s-h_1)+\frac{s\lambda}{2}} \right)+\mathcal{O}\left(a^{\frac{s}{2}+\frac{s\lambda}{2}+\frac{3}{2}(1-s-h_1)} log \ a\right),
\end{aligned}
\label{est-sur-8a}
\end{equation}
and hence
\begin{equation}
\begin{aligned}
B_{2}&=\mathcal{O}\left(a^{\frac{3}{2}(1-s-h_1)+\frac{s}{2}} log \ a \right)+\mathcal{O}\left(a^{\frac{3}{2}(1-s-h_1)+\frac{s}{2}}\right)+\mathcal{O}\left(a^{\frac{s \eta}{2}+\frac{3}{2}(1-s-h_1)} \right)+\mathcal{O}\left(a^{\frac{s}{2}+\frac{s\eta}{2}+\frac{3}{2}(1-s-h_1)} log \ a \right) \\
&\quad +\mathcal{O}\left(a^{\frac{3}{2}(1-s-h_1)+\frac{s\lambda}{2}} \right)+\mathcal{O}\left(a^{\frac{s}{2}+\frac{s\lambda}{2}+\frac{3}{2}(1-s-h_1)} log \ a\right).
\end{aligned}
\label{est-sur-8b}
\end{equation}
Next we estimate the term $D_2$. Just as in the case of the term $D_1$ in the volumetric distributions, for points $z_m$ located near the boundary $\partial \Sigma $ of $\Sigma $, we split the integral into two parts denoted by $F_m $ and $N_m $. Since $F_m \subset \Sigma   \setminus \cup_{j=1}^{[a^{-s}]} \Sigma_j$, it follows that $\vert F_m \vert $ is of the order $a^{\frac{s}{2}} $ as $a \rightarrow 0 $. To estimate the integral over $N_m$, we divide this part into concentric layers using squares. In this case, we have at most $(2n+1) $ squares intersecting the boundary $\partial \Sigma $, for $n=0,\dots,[a^{-\frac{s}{2}}] $. Therefore the number of bubbles in the $n^{th} $ layer will be at most $[(2n+1)-(2n-1)] $ and their distance from $D_m $ is atleast $n \left(a^{\frac{s}{2}}-\frac{a}{2} \right) $. \\
Keeping this in mind, using \eqref{Y-blow} we can write
\begin{equation}
\begin{aligned}
\vert D_2 \vert &= \left\vert \int_{\Sigma \setminus \cup_{j=1 }^{[a^{-s}]} \Sigma_j}  \Phi_{\kappa_0}(z_m,y) K^M(y) \overline{\mathbf{C}} Y(y) ds(y) \right\vert\\
 &=\left\vert \int_{N_m}  \Phi_{\kappa_0}(z_m,y) K^M(y) \overline{\mathbf{C}} Y(y) ds(y) \right\vert+\left\vert \int_{F_m}  \Phi_{\kappa_0}(z_m,y) K^M(y) \overline{\mathbf{C}} Y(y) ds(y) \right\vert \\
%&\leq \sum_{l=1}^{[a^{-\frac{s}{2}}]} \norm{K^M}_{L^{\infty}(\Sigma)} \norm{Y}_{L^{\infty}(\Sigma)} \vert \overline{\mathbf{C}} \vert \vert \Sigma_l \vert \frac{1}{d_{ml}} + \norm{\Phi_{\kappa_0}(z_m,\cdot)}_{L^{\infty}(F_m)}   \norm{K^M}_{L^{\infty}(\Sigma)} \norm{Y}_{L^{\infty}(\Sigma)} \vert \overline{\mathbf{C}} \vert \vert F_m \vert \\
&\leq \mathcal{O}(a^{\frac{3}{2}(1-s-h_{1})}) \norm{K^M}_{L^{\infty}(\Sigma)}  \vert \overline{\mathbf{C}} \vert a^{s} \sum_{l=1}^{[a^{-\frac{s}{2}}]} \frac{1}{d_{ml}} +C a^{\frac{3}{2}(1-s-h_{1})} a^{\frac{s}{2}} \\
&\leq \mathcal{O}(a^{\frac{3}{2}(1-s-h_{1})}) \norm{K^M}_{L^{\infty}(\Sigma)}  \vert \overline{\mathbf{C}} \vert a^{s} \sum_{l=1}^{[a^{-\frac{s}{2}}]} \Big[(2n+1)-(2n-1) \Big] \left(\frac{1}{n \Big( a^{\frac{s}{2}}-\frac{a}{2} \Big)} \right) +C a^{\frac{3}{2}(1-s-h_{1})} a^{\frac{s}{2}}\\
&= \mathcal{O}(a^{\frac{3}{2}(1-s-h_{1})}) \norm{K^M}_{L^{\infty}(\Sigma)} \vert \overline{\mathbf{C}} \vert a^{\frac{s}{2}} \mathcal{O}(log\ a)+\mathcal{O}(a^{\frac{3}{2}(1-s-h_{1})}) \mathcal{O}(a^{\frac{s}{2}}),
\end{aligned}
\notag
\end{equation}
and hence 
\begin{equation}
 D_2 =\mathcal{O}(a^{\frac{3}{2}(1-s-h_{1})+\frac{s}{2}} log\ a).
\label{est-sur-3}
\end{equation}
To estimate the term $A_2$, we write it as
\begin{equation}
\begin{aligned}
A_{2}&=\sum_{{j=1} \atop {j \neq m}}^{M}  \Phi_{\kappa_0}(z_m,z_j) \overline{\mathbf{C}} Y(z_j) a^{s} -\sum_{{j=1} \atop {j \neq m}}^{[a^{-s}]}  \Phi_{\kappa_0}(z_m,z_j) K^M(z_j) \overline{\mathbf{C}} Y(z_j) \vert \Sigma_{j} \vert\\
&=\sum_{{l=1} \atop {{l \neq m} \atop {z_l \in \Sigma_m}}}^{[K^{M}(z_m)]} \Phi_{\kappa_0}(z_m,z_l) \overline{\mathbf{C}} Y(z_l) a^{s}  +\sum_{{j=1} \atop {j \neq m}}^{[a^{-s}]} \sum_{{l=1} \atop {z_l \in \Sigma_j}}^{[K^M(z_j)]} \Phi_{\kappa_0}(z_m,z_l) \overline{\mathbf{C}} Y(z_l) a^{s}\\
&\quad - \sum_{{j=1} \atop {j \neq m}}^{[a^{-s}]}  \Phi_{\kappa_0}(z_m,z_j) K^M(z_j) \overline{\mathbf{C}} Y(z_j) \vert \Sigma_{j} \vert\\
&=\overline{\mathbf{C}} a^{s} \underbrace{\sum_{{l=1} \atop {{l \neq m} \atop {z_l \in \Sigma_m}}}^{[K^{M}(z_m)]} \Phi_{\kappa_0}(z_m,z_l) Y(z_l)}_{E_{3}}+\sum_{{j=1} \atop {j \neq m}}^{[a^{-s}]} \overline{\mathbf{C}} a^{s}  \underbrace{\left[\left(\sum_{{l=1} \atop {z_l \in \Sigma_j}}^{[K^M(z_j)]}\Phi_{\kappa_0}(z_m,z_l) Y(z_l)  \right)-\Phi_{\kappa_0}(z_m,z_j) \left[K^M(z_j)\right] Y(z_j) \right]}_{E_{4}^{j}}.
\end{aligned}
\label{est-sur-9}
\end{equation}
Now the terms $E_3 $, $E^{j}_{4} $ can be estimated just as in the case of terms $E_1 $ and $E^{j}_{2} $ for volumetric distributions and we can deduce that
\begin{equation}
A_2=\mathcal{O}\left(a^{s-t}\right)+\mathcal{O}\left(a^{\frac{3}{2}(1-s-h_1)+\frac{s}{2}} log \ a \right)+\mathcal{O}\left(a^{\frac{s \eta}{2}+\frac{3}{2}(1-s-h_1)}\right)+\mathcal{O}\left(a^{\frac{s}{2}+\frac{s\eta}{2}+\frac{3}{2}(1-s-h_1)} log \ a \right).
\label{est-sur-12a}
\end{equation}
Using the estimates \eqref{est-sur-2},\eqref{est-sur-8b}, \eqref{est-sur-3} and \eqref{est-sur-12a} in \eqref{est-sur-1}, we can write 
\begin{equation}
\begin{aligned}
&Y(z_m)+\sum_{{j=1} \atop {j \neq m}}^{M} \Phi_{\kappa_0}(z_m,z_j) \overline{\mathbf{C}} Y(z_j) a^{1-h_1}\\
&\quad  =u^{I}(z_m)+\mathcal{O}(a^{1-s-h_1}a^{s-t})
+\mathcal{O}\left(a^{\frac{s \lambda}{2}+\frac{5}{2}(1-s-h_1)} \right)+\mathcal{O}(a^{\frac{s}{2}+\frac{s\lambda}{2}+\frac{5}{2}(1-s-h_1)} log \ a)\\
&\quad \quad + \mathcal{O}(a^{\frac{5}{2}(1-s-h_1)+\frac{s}{2}} log \ a)+\mathcal{O}\left(a^{\frac{s \eta}{2}+\frac{5}{2}(1-s-h_1)} \right)+\mathcal{O}(a^{\frac{s}{2}+\frac{s\eta}{2}+\frac{5}{2}(1-s-h_1)} log \ a),
\end{aligned}
\label{est-sur-13}
\end{equation}
and hence
\begin{equation}
\begin{aligned}
&(Y_m-Y(z_m))+\sum_{{j=1} \atop {j \neq m}}^{M} \Phi_{\kappa_0}(z_m,z_j) \overline{\mathbf{C}} (Y_j-Y(z_j)) a^{1-h_1}\\
&\qquad=\mathcal{O}(a^{1-h_1-t})+\mathcal{O}(a^{\frac{5}{2}(1-s-h_1)+\frac{s}{2}} log \ a)
+\mathcal{O}\left(a^{\frac{s \eta}{2}+\frac{5}{2}(1-s-h_1)} \right)+\mathcal{O}(a^{\frac{s}{2}+\frac{s\eta}{2}+\frac{5}{2}(1-s-h_1)} log \ a)\\
&\qquad \quad+\mathcal{O}\left(a^{\frac{s \lambda}{2}+\frac{5}{2}(1-s-h_1)} \right)+\mathcal{O}(a^{\frac{s}{2}+\frac{s\lambda}{2}+\frac{5}{2}(1-s-h_1)} log \ a).
\end{aligned}
\label{est-sur-14}
\end{equation}
Since $Y_{m}-Y(z_m) $ satisfies \eqref{est-sur-14}, from the invertibility of the algebraic system \eqref{est-sur-14}, we deduce
\begin{equation}
\begin{aligned}
\sum_{m=1}^{M} \vert Y_{m}-Y(z_m) \vert &= \mathcal{O}\left(M (a^{1-h_1-t}+a^{\frac{5}{2}(1-s-h_1)+\frac{s}{2}} log \ a+a^{\frac{s \eta}{2}+\frac{5}{2}(1-s-h_1)} +a^{\frac{s}{2}+\frac{s\eta}{2}+\frac{5}{2}(1-s-h_1)} log \ a) \right)\\
&\quad +\mathcal{O}\left(M(a^{\frac{s \lambda}{2}+\frac{5}{2}(1-s-h_1)} +a^{\frac{s}{2}+\frac{s\lambda}{2}+\frac{5}{2}(1-s-h_1)} log \ a) \right).
\end{aligned}
\label{est-sur-15}
\end{equation}
Using the above estimates, we can now compare the far-field values.
Let us denote
\begin{equation}
u^{\infty}_{a}(\hat{x},\theta)=-a^{1-h_1-s}\int_{\Sigma} e^{-i\kappa_{0} \hat{x} \cdot y} K^{M}(y) \overline{\mathbf{C}} Y(y) ds(y).
\notag
\end{equation}
Therefore using \eqref{Near-resonance}, we can write
\begin{equation}
\begin{aligned}
&u^{\infty}(\hat{x},\theta)-u_{a}^{\infty}(\hat{x},\theta)\\
&=a^{1-h_1-s}\left[\int_{\Sigma} e^{-i\kappa_{0} \hat{x} \cdot y} K^{M}(y) \overline{\mathbf{C}} Y(y) ds(y)-\sum_{j=1}^{M} e^{-i \kappa_{0} \hat{x} \cdot z_{j}} \overline{\mathbf{C}} Y_{j} a^{s} \right] +\mathcal{O}(a^{2-s-h_1}+a^{3-2t-2s-2h_1})\\
&=a^{1-h_1-s}\int_{\Sigma \setminus \cup_{j=1}^{[a^{-s}]} \Sigma_j} e^{-i\kappa_{0} \hat{x} \cdot y} K^{M}(y) \overline{\mathbf{C}} Y(y) ds(y) +a^{1-h_1-s}\sum_{j=1}^{[a^{-s}]} \int_{\Sigma_j} e^{-i\kappa_{0} \hat{x} \cdot y} K^{M}(y) \overline{\mathbf{C}} Y(y) ds(y)\\
&\quad -a^{1-h_1-s}\sum_{j=1}^{M} e^{-i \kappa_{0} \hat{x} \cdot z_{j}} \overline{\mathbf{C}} Y_{j} a^{s} +\mathcal{O}(a^{2-s-h_1}+a^{3-2t-2s-2h_1}).
\end{aligned}
\notag
\end{equation}
%
\begin{comment}
\begin{equation}
\begin{aligned}
&u^{\infty}_{a}(\hat{x},\theta)-u^{\infty}(\hat{x},\theta)\\
&= a^{1-h_1-s}\sum_{j=1}^{[a^{-s}]} K^{M}(z_j) \overline{\mathbf{C}} \int_{\Sigma_j} \left[e^{-i\kappa_{0} \hat{x} \cdot y}  Y(y)-e^{-i\kappa_{0} \hat{x} \cdot z_j}  Y(z_j)  \right] dy \\
&\quad +a^{1-h_1-s}\left[\sum_{j=1}^{[a^{-s}]} \int_{\Sigma_j} e^{-i\kappa_{0} \hat{x} \cdot y} K^{M}(y) \overline{\mathbf{C}} Y(y) dy-\sum_{j=1}^{[a^{-s}]} \int_{\Sigma_j} e^{-i\kappa_{0} \hat{x} \cdot y} K^{M}(z_j) \overline{\mathbf{C}} Y(y) dy\right]\\
&\quad +\sum_{j=1}^{[a^{-s}]}\overline{\mathbf{C}} a^{1-h_1} \sum_{{l=1} \atop {z_{l} \in \Sigma_j}}^{[K^{M}(z_j)]} \left(e^{-i \kappa_{0} \hat{x} \cdot z_{j}} Y(z_j) -e^{-i \kappa_{0} \hat{x} \cdot z_{l}} Y(z_l)\right)+\sum_{j=1}^{M} e^{-i \kappa_{0} \hat{x} \cdot z_{j}} \overline{\mathbf{C}} a^{1-h_1} \left[ Y(z_j)- Y_{j} \right] \\
&\quad +\mathcal{O}\left(a^{\frac{5}{2}(1-h_{1}-s)+\frac{s}{2}}+a^{2-s-2h_{1}}+a^{3-2t-2s-2h_{1}}\right),
\end{aligned}
\notag
\end{equation}
\end{comment}
%
Now using \eqref{est-sur-15} and the fact that $\vert \Sigma\setminus \cup_{j=1}^{M} \Sigma_j \vert=\mathcal{O}\left(a^\frac{s}{2} \right) $, by proceeding as in the case of volumetric distributions we can deduce
\begin{equation}
\begin{aligned}
&u^{\infty}(\hat{x},\theta)-u_{a}^{\infty}(\hat{x},\theta)\\
&= a^{1-h_1-s}\sum_{j=1}^{[a^{-s}]} K^{M}(z_j) \overline{\mathbf{C}} \int_{\Sigma_j} \left[e^{-i\kappa_{0} \hat{x} \cdot y}  Y(y)-e^{-i\kappa_{0} \hat{x} \cdot z_j}  Y(z_j)  \right] ds(y) \\
&\quad +a^{1-h_1-s}\left[\sum_{j=1}^{[a^{-s}]} \int_{\Sigma_j} e^{-i\kappa_{0} \hat{x} \cdot y} K^{M}(y) \overline{\mathbf{C}} Y(y) ds(y)-\sum_{j=1}^{[a^{-s}]} \int_{\Sigma_j} e^{-i\kappa_{0} \hat{x} \cdot y} K^{M}(z_j) \overline{\mathbf{C}} Y(y) ds(y)\right]\\
&\quad +\sum_{j=1}^{[a^{-s}]}\overline{\mathbf{C}} a^{1-h_1} \sum_{{l=1} \atop {z_{l} \in \Sigma_j}}^{[K^{M}(z_j)]} \left(e^{-i \kappa_{0} \hat{x} \cdot z_{j}} Y(z_j) -e^{-i \kappa_{0} \hat{x} \cdot z_{l}} Y(z_l)\right) \\
&\quad +\mathcal{O}\left(a^{\frac{5}{2}(1-h_{1}-s)+\frac{s}{2}}+a^{2-s-2h_{1}}+a^{3-2t-2s-2h_{1}}\right)\\
&\quad +\mathcal{O}\left(M a^{1-h_1} \left[a^{1-h_1-t}+a^{\frac{5}{2}(1-s-h_1)+\frac{s}{2}} log \ a+a^{\frac{s \eta}{2}+\frac{5}{2}(1-s-h_1)} +a^{\frac{s}{2}+\frac{s\eta}{2}+\frac{5}{2}(1-s-h_1)} log \ a\right] \right)\\
&\quad +\mathcal{O}\left(M a^{1-h_1} \left[a^{\frac{s \lambda}{2}+\frac{5}{2}(1-s-h_1)} +a^{\frac{s}{2}+\frac{s\lambda}{2}+\frac{5}{2}(1-s-h_1)} log \ a\right] \right).
\end{aligned}
\label{est-sur-16}
\end{equation}
Now proceeding as in the case of $B_2$, it can be seen that
\begin{equation}
\begin{aligned}
&\sum_{j=1}^{[a^{-s}]}\overline{\mathbf{C}} a^{1-h_1} \sum_{{l=1} \atop {z_{l} \in \Sigma_j}}^{[K^{M}(z_j)]} \left(e^{-i \kappa_{0} \hat{x} \cdot z_{j}} Y(z_j) -e^{-i \kappa_{0} \hat{x} \cdot z_{l}} Y(z_l)\right)
%&\qquad \qquad=\mathcal{O}(a^{1-h_1-s} a^{\frac{s \eta}{2}} \norm{Y}_{W^{1,p}(\Sigma)})+\mathcal{O}(a^{1-h_1-s} a^{\frac{s}{2}} \norm{Y}_{L^{\infty}(\Sigma)})\\
=\mathcal{O}(a^{1-h_1-s} a^{\frac{s \eta}{2}} a^{\frac{3}{2}(1-h_1-s)})+\mathcal{O}(a^{1-h_1-s} a^{\frac{s }{2}} a^{\frac{3}{2}(1-h_1-s)}),
\end{aligned}
\label{est-sur-17}
\end{equation}
\begin{equation}
\begin{aligned}
&a^{1-h_1-s}\sum_{j=1}^{[a^{-s}]} K^{M}(z_j) \overline{\mathbf{C}} \int_{\Sigma_j} \left[e^{-i\kappa_{0} \hat{x} \cdot y}  Y(y)-e^{-i\kappa_{0} \hat{x} \cdot z_j}  Y(z_j)  \right] ds(y) 
%&\qquad =\mathcal{O}\left(a^{1-h_1-s} \left[\sum_{n=1}^{[a^{-\frac{s}{2}}]} [8n+2] a^{\frac{3s}{2}} \norm{Y}_{L^{\infty}(\Sigma)}+\sum_{n=1}^{[a^{-\frac{s}{2}}]} [8n+2] a^{s+\frac{s\eta}{2}} \norm{Y}_{W^{1,p}(\Sigma)} \right]\right)\\
 =\mathcal{O}\left(a^{1-h_1-s} a^{\frac{3}{2}(1-h_1-s)} a^{\frac{s}{2}}\right)+\mathcal{O}(a^{\frac{s\eta}{2}+\frac{5}{2}(1-s-h_1)}),
\end{aligned}
\label{est-sur-18}
\end{equation}
and
\begin{equation}
\begin{aligned}
&a^{1-h_1-s}\left[\sum_{j=1}^{[a^{-s}]} \int_{\Sigma_j} e^{-i\kappa_{0} \hat{x} \cdot y} K^{M}(y) \overline{\mathbf{C}} Y(y) ds(y)-\sum_{j=1}^{[a^{-s}]} \int_{\Sigma_j} e^{-i\kappa_{0} \hat{x} \cdot y} K^{M}(z_j) \overline{\mathbf{C}} Y(y) ds(y)\right]
=\mathcal{O}\left(a^{\frac{s\lambda}{2}+\frac{5}{2}(1-s-h_1)}\right).
\end{aligned}
\label{est-sur-18a}
\end{equation}
Using \eqref{est-sur-17}-\eqref{est-sur-18a} in \eqref{est-sur-16}, we obtain
\begin{equation}
\begin{aligned}
&u^{\infty}(\hat{x},\theta)-u_{a}^{\infty}(\hat{x},\theta)\\
&=\mathcal{O}\left(a^{\frac{5}{2}(1-h_{1}-s)+\frac{s}{2}}+a^{\frac{5}{2}(1-h_{1}-s)+\frac{s\eta}{2}}+a^{\frac{5}{2}(1-h_{1}-s)+\frac{s\lambda}{2}}+a^{2-s-2h_{1}}+a^{3-2t-2s-2h_{1}} \right)\\
&\quad +\mathcal{O}\left( a^{1-h_1-s} \left[a^{1-h_1-t}+a^{\frac{5}{2}(1-s-h_1)+\frac{s}{2}} log \ a+a^{\frac{s\eta}{2}+\frac{5}{2}(1-s-h_1)}+a^{\frac{s}{2}+\frac{s\eta}{2}+\frac{5}{2}(1-s-h_1)} log \ a\right] \right)\\
&\quad +\mathcal{O}\left( a^{1-h_1-s} \left[a^{\frac{s \lambda}{2}+\frac{5}{2}(1-s-h_1)} +a^{\frac{s}{2}+\frac{s\lambda}{2}+\frac{5}{2}(1-s-h_1)} log \ a\right] \right)\\
&=\mathcal{O}\left(a^{\frac{5}{2}(1-h_{1}-s)+\frac{s}{2}}+a^{\frac{5}{2}(1-h_{1}-s)+\frac{s\lambda}{2}}+a^{2-s-2h_{1}}+a^{3-2t-2s-2h_{1}} \right)\\
&\quad +\mathcal{O}\left( a^{1-h_1-s} \left[a^{1-h_1-t}+a^{\frac{5}{2}(1-s-h_1)+\frac{s}{2}} log \ a+a^{\frac{s\lambda}{2}+\frac{5}{2}(1-s-h_1)}+a^{\frac{s}{2}+\frac{s\lambda}{2}+\frac{5}{2}(1-s-h_1)} log \ a\right] \right),
\end{aligned}
\label{est-sur-19}
\end{equation}
since we can choose $\eta $ such that $\lambda < \eta $.\\
We already know that $2-s-2h_{1}>0 $ and $3-2t-2s-2h_{1}>0 $.
\begin{itemize}
\item Note that $a^{1-h_1-s} \cdot a^{1-h_1-t}=a^{2-2h_{1}-s-t} $. Now if $h_{1}+t<\frac{1}{2} $, then
\[2-2h_{1}-s-t> 2-h_{1}-s-\frac{1}{2}=\frac{3}{2}-h_{1}-s>0, \] since we are in the regime $1<s+h_{1}<\min \{\frac{3}{2}-t, 2-h_1\} $.\\
Hence a sufficient condition, in this case, can be written as
\begin{equation}
0<1-h_{1}<s \leq 3t< \min \left\{\frac{3}{2}-t-h_{1}, 2-2h_1 \right\},\ h_1 < \frac{1}{6}. 
\label{cond-1}
\end{equation}
\item We now want conditions to guarantee $\frac{5}{2}(1-h_1-s)+\frac{s}{2}>0 $. Note that $\frac{5}{2}(1-h_1-s)+\frac{s}{2}=\frac{5}{2}-\frac{5h_{1}}{2}-2s $.\\
Now if $s<\frac{5}{4}-\frac{5h_1}{4} $, then we can guarantee $\frac{5}{2}-\frac{5h_{1}}{2}-2s>0 $.\\
Hence a sufficient condition, in this case, can be written as
\begin{equation}
0<1-h_{1}<s<\frac{5}{4}-\frac{5h_1}{4}. 
\label{cond-2}
\end{equation}

\item We next look for conditions to guarantee that $\frac{5}{2}(1-h_1-s)+\frac{s\lambda}{2} $ is greater than $0$. Note that $\frac{5}{2}(1-h_1-s)+\frac{s\lambda}{2}=\frac{5}{2}-\frac{5 h_1}{2}-\frac{5s}{2}+\frac{s \lambda}{2} $.\\
Now if $s+h_{1}<1+\frac{s\lambda}{5} $, then we can guarantee $\frac{5}{2}-\frac{5 h_1}{2}-\frac{5s}{2}+\frac{s \lambda}{2}>0 $. \\
Hence a sufficient condition, in this case, can be written as
\begin{equation}
\begin{aligned}
&1<s+h_{1}<1+\frac{s\lambda}{5}.
\end{aligned}
\label{cond-3}
\end{equation}

\item Next we deal with the term $a^{\frac{s}{2}+\frac{s\lambda}{2}+\frac{7}{2}(1-s-h_1)} log \ a $. Note that $a^{\frac{s\lambda}{2}}\ log\ a  \xrightarrow[a \rightarrow 0]{}  0 $. Therefore if $a^{\frac{s}{2}+\frac{7}{2}(1-s-h_1)}  \xrightarrow[a \rightarrow 0]{} 0 $, then $a^{\frac{s}{2}+\frac{s\lambda}{2}+\frac{7}{2}(1-s-h_1)} log \ a \xrightarrow[a \rightarrow 0]{} 0 $ hold true. Hence we need to ensure that $\frac{s}{2}+\frac{7}{2}(1-s-h_1)=\frac{7}{2}-3s-\frac{7 h_1}{2}>0 $.\\
Now if $s<\frac{7}{6}-\frac{7h_1}{6} $, then we can guarantee that $\frac{7}{2}-3s-\frac{7 h_1}{2}>0  $. \\
Hence a sufficient condition, in this case, can be written as
\begin{equation}
1-h_1<s<\frac{7}{6}-\frac{7h_1}{6}. 
\label{cond-4}
\end{equation}
\item Next we consider the term $a^{\frac{7}{2}(1-s-h_1)+\frac{s}{2}}log \ a $. Since $a^{\alpha}log \ a \xrightarrow[a \rightarrow 0]{} 0 $ for any $\alpha>0 $, it is sufficient that $\frac{7}{2}(1-h_1-s)+\frac{s}{2}>0 $. As seen in the previous case, the condition \eqref{cond-4} is sufficient for this to be true.
\item Finally we consider the term $a^{\frac{7}{2}(1-s-h_1)+\frac{s\lambda}{2}} $. Note that if $s+h_1<1+\frac{s\lambda}{7} $, then $\frac{7}{2}-\frac{7h_1}{2}-\frac{7s}{2}+\frac{s \lambda}{2} >0 $. \\
Therefore a sufficient condition, in this case, can be written as
\begin{equation}
\begin{aligned}
&1<s+h_1<1+\frac{s\lambda}{7}.
\end{aligned}
\label{cond-5}
\end{equation}
\end{itemize}
From \eqref{cond-1}-\eqref{cond-5}, we can derive the following set of sufficient conditions:
\begin{equation}
\begin{aligned}
&0<1-h_{1}<s \leq 3t< \min \left\{\frac{3}{2}-t-h_{1}, 2-2h_1 \right\},\ h_1 < \frac{1}{6}, \\
&0<1-h_1<s<\frac{7}{6}-\frac{7h_1}{6}, \\
&1<s+h_1<1+\frac{s\lambda}{7} .
\end{aligned}
\label{cond-6}
\end{equation}
Now note that if $s+h_1 $ further satisfies the condition $s+h_1<1+\frac{(1-h_1)\lambda}{7}  $, then $s+h_1<1+\frac{s\lambda}{7} $ as well. \\
Also since $\lambda \in (0,1) $, it follows that $1+\frac{(1-h_1)\lambda}{7}-h_1=\frac{7+\lambda}{7}(1-h_1)<\frac{7}{6}(1-h_1) $.\\
Therefore if $s<\frac{7+\lambda}{7}(1-h_1) $, then $s<\frac{7}{6}-\frac{7h_1}{6} $ as well. \\
Hence we can replace the set of conditions \eqref{cond-6} by the following set of sufficient conditions:
\begin{equation}
\begin{aligned}
&0<1-h_{1}<s \leq 3t< \min \left\{\frac{3}{2}-t-h_{1}, 2-2h_1 \right\},\ h_1 < \frac{1}{6}, \\
&0<1-h_1<s<\frac{7+\lambda}{7}(1-h_1) .
\end{aligned}
\label{cond-7}
\end{equation}
which can be further replaced by the condition\footnote{
Note that if $\frac{\lambda}{7+\lambda}<h_1<\frac{1}{6}  $, we have $\frac{3}{2}-t-h_1>\left(1+\frac{\lambda}{7}\right) (1-h_1), \ 2-2h_1> \left(1+\frac{\lambda}{7}\right) (1-h_1) $ and we can replace the conditions by the sufficient condition
\begin{equation}
\begin{aligned}
&0<1-h_1<s\leq 3t<\left(1+\frac{\lambda}{7}\right)(1-h_1).
\end{aligned}
\notag
\end{equation}
}
\begin{equation}
\begin{aligned}
&0<1-h_{1}<s \leq 3t< \min \left\{\frac{3}{2}-t-h_{1}, \left(1+\frac{\lambda}{7}\right)(1-h_1) \right\},\ h_1 < \frac{1}{6}.
\end{aligned}
\label{cond-8}
\end{equation}
%
\begin{comment}
Now if $h_1< \frac{7-2\lambda}{28-2\lambda} $, we have $\frac{3}{2}-3h_1> \frac{7+\lambda}{7}(1-h_1) $ and we can replace the conditions \eqref{cond-7} by the sufficient condition
\begin{equation}
\begin{aligned}
&0<1-h_1<s\leq 3t< \frac{7+\lambda}{7}(1-h_1) .
\end{aligned}
\label{cond-8}
\end{equation}
\end{comment}
%
Finally from \eqref{sur-est-2000} and \eqref{est-sur-19}, we derive
\begin{equation}
\begin{aligned}
&u^{\infty}(\hat{x},\theta)-u^{\infty}_{D}(\hat{x},\theta)\\
&=\mathcal{O}\left(a^{\frac{s+h_1-1}{2}}+a^{2-s-2h_{1}}+a^{3-2t-2s-2h_{1}} 
 + a^{2-2h_1-s-t}+a^{\frac{7}{2}(1-s-h_1)+\frac{s}{2}} log \ a \right)\\
 &\qquad +\mathcal{O}\left(a^{\frac{s\lambda}{2}+\frac{7}{2}(1-s-h_1)}+a^{\frac{s}{2}+\frac{s\lambda}{2}+\frac{7}{2}(1-s-h_1)} log \ a \right).
\end{aligned}
\label{final-sur-blow}
\end{equation}
\begin{remark}
When ($\gamma <1,\ \gamma+s=2 $) or ($\gamma=1, \gamma+s=2 $ with the frequency $\omega $ away from the Minnaert resonance), the estimates can be deduced similarly by using \eqref{Y-non-blow} instead of \eqref{Y-blow}. Also arguing similarly as in the case of volumetric distributions, we can further compare the far-fields corresponding to $\overline{\mathbf{C}} $ to that of $\overline{\mathbf{C}}_{lead} $.  In particular, when $\gamma <1,\ \gamma+s=2 $, we obtain
\begin{equation}
\begin{aligned}
u^{\infty}(\hat{x},\theta)-u^{\infty}_{lead}(\hat{x},\theta)
&=\mathcal{O}\left(a^{1-\gamma}+a^{\frac{s\eta}{2}}+a^{\frac{s\lambda}{2}}+a^{2-s}+a^{3-\gamma-2t-s}+a^{s-t}+a^{\frac{s}{2}} log\ a \right),
\end{aligned}
\label{gamma-small-sur}
\end{equation}
and when  $\gamma=1, \gamma+s=2 $ with the frequency $\omega $ away from the Minnaert resonance, we obtain
\begin{equation}
\begin{aligned}
u^{\infty}(\hat{x},\theta)-u^{\infty}_{lead}(\hat{x},\theta)
&=\mathcal{O}\left(a^{2}+a^{\frac{s\eta}{2}}+a^{\frac{s\lambda}{2}}+a^{2-s}+a^{3-\gamma-2t-s}+a^{s-t}+a^{\frac{s}{2}} log\ a \right)\\
&=\mathcal{O}\left(a^{\frac{s\eta}{2}}+a^{\frac{s\lambda}{2}}+a^{2-s}+a^{3-\gamma-2t-s}+a^{s-t}+a^{\frac{s}{2}} log\ a \right),
\end{aligned}
\label{gamma-away-sur}
\end{equation}
where 
\begin{equation}
u^{\infty}_{a}(\hat{x},\theta)=-\int_{\Sigma} e^{-i\kappa_{0} \hat{x} \cdot y} K^{M}(y) \overline{\mathbf{C}}_{lead} Y(y) ds(y).
\notag
\end{equation}
Now suppose that $\gamma=1$ and $\omega $ is near the Minnaert resonance,i.e. $1-\frac{\omega^2_M}{\omega^2}=l_M a^{h_1}$, with $l_M \neq 0$ and $h_1 \in (0, 1)$ where $s$ and $t$ satisfying the conditions \[\ s=1-h_1 \mbox{ and } \frac{s}{3} \leq t<\min\{1-h_1,\frac{1}{2}\}.\] 
Then if we use the fact that $s+h_1=1 $ in \eqref{est-sur-19}, combined with the fact that
\begin{equation}
\omega^2-\overline{\omega}_{M}^{2}=\omega^2 l_{M} a^{h_1}+ \underbrace{(\omega^2-\omega_{M}^{2})}_{\mathcal{O}(a^2)},
\notag
\end{equation}
\end{remark}
we can derive
\begin{equation}
\begin{aligned}
u^{\infty}(\hat{x},\theta)-u^{\infty}_{a}(\hat{x},\theta)
&=\mathcal{O}\left(a^{h_1}+a^{\frac{(1-h_1)\eta}{2}}+a^{\frac{(1-h_1)\lambda}{2}}+a^{1-h_1}+a^{1-2t}+a^{1-h_1-t}+a^{\frac{1-h_1}{2}} log\ a \right).
\end{aligned}
\label{gamma-near-sur}
\end{equation}
\appendix
\section{}\label{scale}
In this appendix, we derive expansions for the scattering coefficients $\mathbf{C}$ in the regimes $\gamma<1 $ and $\gamma=1 \ \text{but when the frequency is away from the resonance} $.\\
Let us recall that 
\begin{equation}
\mathbf{C}= \frac{\kappa^{2} \vert D \vert}{\frac{\rho}{\rho-\rho_0}-\frac{1}{8\pi}\kappa^{2} \hat{A}}
\notag
\end{equation}
\begin{itemize}

\item[$\gamma<1$:] 
We rewrite $ \mathbf{C}$ as 
\begin{equation}
\begin{aligned}
\mathbf{C}&= \kappa^{2} \vert D \vert \frac{1}{\frac{\rho}{\rho-\rho_0} \left[1-\frac{1}{8\pi} \kappa^{2} \hat{A} \frac{\rho-\rho_0}{\rho} \right]} \\
&=-\kappa^{2} \vert D \vert  \frac{\rho_0-\rho}{\rho \left[ 1-\underbrace{\frac{1}{8\pi} \kappa^{2} \hat{A} \frac{\rho-\rho_0}{\rho}}_{X}  \right]}
=-\kappa^{2} \vert D \vert \frac{\rho_0}{\rho} \left[1-\frac{\rho}{\rho_0} \right] \left[1-X\right]^{-1}.
\end{aligned}
\notag
\end{equation}
Note that 
\begin{equation}
X=\frac{1}{8\pi} \kappa^{2} \hat{A} \frac{\rho-\rho_0}{\rho}= \underbrace{\frac{1}{8 \pi} \kappa^{2} \hat{A}}_{\mathcal{O}(a^2)} -  \underbrace{\frac{1}{8 \pi} \kappa^{2} \hat{A} \frac{\rho_0}{\rho}}_{\mathcal{O}(a^2 \cdot a^{-1-\gamma})} = \mathcal{O}(a^{1-\gamma}).
\notag
\end{equation}
Therefore
\begin{equation}
\begin{aligned}
\mathbf{C}&= -\kappa^{2} \vert D \vert \frac{\rho_0}{\rho} \left[1-\frac{\rho}{\rho_0} \right] \left[1+X+X^{2}+ \dots \right]\\
& =-\kappa^{2} \vert D \vert \frac{\rho_0}{\rho} \left[1-\frac{\rho}{\rho_0} \right]  -\kappa^{2} \vert D \vert \frac{\rho_0}{\rho} \left[1-\frac{\rho}{\rho_0} \right] X -\kappa^{2} \vert D \vert \frac{\rho_0}{\rho} \left[1-\frac{\rho}{\rho_0} \right] X^{2}+\dots \\
&= \underbrace{ -\kappa^{2} \vert D \vert \frac{\rho_0}{\rho}}_{\mathcal{O}(a^{2-\gamma})}-\underbrace{ \kappa^{2} \vert D \vert \frac{\rho_0}{\rho} X}_{\mathcal{O}(a^{1-\gamma} \cdot a^{2-\gamma})}+\underbrace{ \kappa^{2} \vert D \vert}_{\mathcal{O}(a^{1+\gamma} \cdot a^{2-\gamma})}+\underbrace{ \kappa^{2} \vert D \vert X}_{\mathcal{O}(a^{2} \cdot a^{2-\gamma})}+\dots\\
&= -\kappa^{2} \vert D  \vert \frac{\rho_0}{\rho} +\mathcal{O}\left(a^{1-\gamma} \cdot a^{2-\gamma} \right).
\end{aligned}
\notag
\end{equation}
\item[$\gamma=1$ ](and away from resonance):
Proceeding as in the earlier case, we write
\begin{equation}
\begin{aligned}
\mathbf{C}&= \kappa^{2} \vert D \vert \frac{1}{\frac{\rho}{\rho-\rho_0} \left[1-\frac{1}{8\pi} \kappa^{2} \hat{A} \frac{\rho-\rho_0}{\rho} \right]} 
=-\kappa^{2} \vert D \vert  \frac{\rho_0-\rho}{\rho \left[ 1-\underbrace{\frac{1}{8\pi} \kappa^{2} \hat{A} \frac{\rho-\rho_0}{\rho}}_{X}  \right]}.
\end{aligned}
\notag
\end{equation}
Now we can write $1-X $ as 
\begin{equation}
\begin{aligned}
1-X&=1+\underbrace{\frac{1}{8\pi} \kappa^{2} \hat{A} \frac{\rho_0}{\rho}}_{\sim 1} -\underbrace{\frac{1}{8\pi} \kappa^{2} \hat{A} }_{\mathcal{O}(a^2)} \\
&=\underbrace{\left[1+\frac{1}{8\pi} \kappa^{2} \hat{A} \frac{\rho_0}{\rho} \right]}_{X_1} \left[ 1-\underbrace{\frac{\frac{1}{8\pi} \kappa^{2} \hat{A}}{1+\frac{1}{8\pi} \kappa^{2} \hat{A} \frac{\rho_0}{\rho}}}_{X_2} \right]
\end{aligned}
\notag
\end{equation}
which implies 
\begin{equation}
\left[1-X \right]^{-1}=X_{1}^{-1} \left[1-X_2 \right]^{-1},\ \text{where}\ X_1\sim 1, X_2=\mathcal{O}(a^2) .
\notag
\end{equation}
Therefore
\begin{equation}
\begin{aligned}
\mathbf{C}&= -\kappa^{2} \vert D \vert \frac{\rho_0}{\rho} \left[1-\frac{\rho}{\rho_0} \right] X_{1}^{-1} \left[1-X_2 \right]^{-1}\\
&= -\kappa^{2} \vert D \vert \frac{\rho_0}{\rho} \left[1-\frac{\rho}{\rho_0} \right] X_{1}^{-1} \left[1+X_{2}+X_{2}^{2}+ \dots \right]\\
& =-\kappa^{2} \vert D \vert \frac{\rho_0}{\rho} \left[1-\frac{\rho}{\rho_0} \right] X_{1}^{-1} -\kappa^{2} \vert D \vert \frac{\rho_0}{\rho} \left[1-\frac{\rho}{\rho_0} \right] X_{1}^{-1} X_{2} -\kappa^{2} \vert D \vert \frac{\rho_0}{\rho} \left[1-\frac{\rho}{\rho_0} \right] X_{1}^{-1} X_{2}^{2}+\dots \\
&= \underbrace{ -\kappa^{2} \vert D \vert \frac{\rho_0}{\rho} X_{1}^{-1}}_{\mathcal{O}(a)}-\underbrace{ \kappa^{2} \vert D \vert \frac{\rho_0}{\rho} X_{1}^{-1} X_2}_{\mathcal{O}(a \cdot a^{2})}+\underbrace{ \kappa^{2} \vert D \vert X_{1}^{-1}}_{\mathcal{O}(a^{3}) }+\underbrace{ \kappa^{2} \vert D \vert X_{1}^{-1} X_2}_{\mathcal{O}(a^{3} \cdot a^{2})}+\dots\\
&= -\kappa^{2} \vert D \vert \frac{\rho_0}{\rho} X_{1}^{-1}+\mathcal{O}\left(a^{3}  \right).
\end{aligned}
\notag
\end{equation}
\end{itemize}

\section{}\label{Sing}
In this appendix, we outline a proof for the invertibility of the single layer potential $\mathbf{S}_{\kappa_0} $ when restricted to functions defined in an open subset $\Sigma$ of $\Gamma $, where $\Gamma=\partial D  $ for some open connected subset $D$ of  $\mathbb{R}^{3}$.\\
Let us recall that
\begin{equation}
\begin{aligned}
H^{s}(\Sigma)&:=\{f \vert_{\Sigma} : f \in H^{s}(\Gamma) \} , \ H^{s}_{\overline{\Sigma}}(\Gamma):=\{f \in H^{s}(\Gamma): \text{supp}\ f \subseteq \overline{\Sigma} \} , \\
H^{-s}_{\Sigma}(\Gamma)&:=\{\phi \in H^{-s}(\Gamma): \langle \phi, \psi \rangle_{-s,s}=0,\ \text{for any}\  \psi \in H^{s}_{\Gamma \setminus \Sigma}(\Gamma)\} . 
\end{aligned}
\notag
\end{equation}
It can be seen that
\begin{equation}
\left(H^{s}(\Sigma) \right)' \simeq H_{\overline{\Sigma}}^{-s}(\Gamma), \ \left(H^{s}_{\overline{\Sigma}}(\Gamma) \right)' \simeq H^{-s}(\Sigma) .
\notag
\end{equation}
The following property for $\mathbf{S}_{\kappa_0}\Big\vert_{\Sigma} $ has been proved in (theorem $2.4$, \cite{CS}):
\begin{equation}
\mathbf{S}_{\kappa_0}\Big\vert_{\Sigma} \phi \in H^{s+1}(\Sigma) \Longleftrightarrow \phi \in H^{s}_{\overline{\Sigma}}(\Gamma), \ -1<s<0. 
\label{res-1}
\end{equation}
Using \eqref{res-1}, we next study the invertibility of the single layer potential when restricted to $H_{\Sigma}^{-1}(\Gamma) $.
\begin{theorem}\label{Sing-inv}
Let $\Sigma$ be an open subset of $\Gamma $, where $\Gamma=\partial D  $ for some open connected subset $D$ of  $\mathbb{R}^{3}$. Assume that $\kappa_{0}^{2} $ is not an eigenvalue for the Dirichlet Laplacian in $D$. Then the mapping
\begin{equation}
\mathbf{S}_{\kappa_0}\Big\vert_{\Sigma}: H_{\Sigma}^{-1}(\Gamma) \rightarrow L^2(\Sigma) 
\notag
\end{equation}
is invertible.
\end{theorem}
\begin{proof}
We first prove the invertibility in the case when $\kappa_{0}^{2} \neq \omega_{N,D} $, where $\omega_{N,D} $ is an eigenvalue for the Neumann Laplacian in $ D $.
\begin{itemize}
\item {\textit{Surjectivity}}: Let $g \in L^{2}(\Sigma) $ and we define \[\tilde{g}:= \begin{cases}
g, \ \text{in}\ \Sigma,\\
0, \ \text{in}\ \Gamma \setminus \overline{\Sigma}.
\end{cases} \] 
Also let $f \in L^{2}(\Gamma) $ be such that $\left(-\frac{1}{2}Id+\mathbf{K}_{\kappa_0} \right)f=\tilde{g} $, and we define \[v :=\mathbf{K}_{\kappa_0}f. \]
Then the function $v$ satisfies $\left(\Delta+\kappa_{0}^{2} \right)v=0 \ \text{in}\ D $. Also since $f \in L^{2}(\Gamma) $, it follows that $\frac{\partial v}{\partial \nu}\Big\vert_{\Gamma} \in H^{-1}(\Gamma) $.\\
Now let $h \in H^{-1}(\Gamma) $ be such that $\left(-\frac{1}{2}Id+\mathbf{K}_{\kappa_0}^{*} \right)h=\frac{\partial v}{\partial \nu} $ and we set \[w:=\mathbf{S}_{\kappa_0}h. \]
Then the function $v-w$ satisfies
\begin{equation}
\begin{aligned}
\left(\Delta+\kappa_{0}^{2} \right)(v-w)&=0, \ \text{in}\ D,\\
\frac{\partial (v-w)}{\partial \nu}&=0, \ \text{on}\ \partial D.
\end{aligned}
\notag
\end{equation}
Since $\kappa_{0}^{2} \neq \omega_{N,D} $, it follows that $v-w=0 $ and hence $\tilde{g}=\mathbf{S}_{\kappa_0}h $.\\
Let us further assume that $g \in  H^{s}(\Sigma), \ 0<s<\frac{1}{2}$. Then from \eqref{res-1} and the fact that $\tilde{g}=\mathbf{S}_{\kappa_0}h $, we can conclude that $\text{supp}(h)\subset \overline{\Sigma} $ since $\text{supp}(\tilde{g})\subset \overline{\Sigma} $.\\
By a density argument and using the invertibility of $\mathbf{S}_{\kappa_0}: H^{-1}(\Gamma)\rightarrow L^{2}(\Gamma) $, we can now deduce that $L^{2}(\Sigma) \subset \mathbf{S}_{\kappa_0}\big\vert_{\Sigma}\left(H^{-1}_{\Sigma}(\Gamma) \right) $, that is, $\mathbf{S}_{\kappa_0}\Big\vert_{\Sigma} $ is surjective.
\item {\textit{Injectivity}}: Let $\mathbf{S}_{\kappa_0}\Big\vert_{\Sigma} f=0 $ for $f \in H_{\Sigma}^{-1}(\Gamma) $.\\
Now by definition, $f \in H_{\Sigma}^{-1}(\Gamma) $ implies that $f \in H^{-1}(\Gamma) $.
Therefore using the fact that the mapping $\mathbf{S}_{\kappa_0}: H^{-1}(\Gamma) \rightarrow L^{2}(\Gamma)  $ is invertible, we immediately have $f =0 $ whence injectivity follows.
\end{itemize}
Hence if $\kappa_{0}^{2} \neq \omega_{N,D} $, the above argument shows that 
\[\mathbf{S}_{\kappa_0}\Big\vert_{\Sigma}: H_{\Sigma}^{-1}(\Gamma) \rightarrow L^2(\Sigma) \]
is a bijection. Since it is continuous, it follows that $ \mathbf{S}_{\kappa_0}\Big\vert_{\Sigma}$ is an isomorphism. \\
Now if $\kappa_{0}^{2} = \omega_{N,D} $, we write 
\[\mathbf{S}_{\kappa_0}\Big\vert_{\Sigma}=\mathbf{S}_{k}\Big\vert_{\Sigma}+\left(\mathbf{S}_{\kappa_0}\Big\vert_{\Sigma}-\mathbf{S}_{k}\Big\vert_{\Sigma} \right), \]
where $k^2 $ is not an eigenvalue for the Neumann or Dirichlet Laplacian in $D $. From the previous step, we know that $\mathbf{S}_{k}\Big\vert_{\Sigma} $ is an isomorphism. Also it is easy to see that $\mathbf{S}_{\kappa_0}\Big\vert_{\Sigma}-\mathbf{S}_{k}\Big\vert_{\Sigma}  $ is compact. Therefore $\mathbf{S}_{\kappa_0}\Big\vert_{\Sigma} $ is Fredholm with index zero. The injectivity (and hence invertibility) of $\mathbf{S}_{\kappa_0}\Big\vert_{\Sigma} $ follows as in the earlier case as the proof holds for any $\kappa_{0}^2 $ which is not an eigenvalue for the Dirichlet Laplacian in $D$.   
\end{proof}


\begin{thebibliography}{10}

\bibitem{ACKS}
B. Ahmad, D.P. Challa, M. Kirane, and M. Sini,
The equivalent refraction index for the acoustic scattering by many small obstacles: with error estimates,
\emph{J. Math. Anal. Appl.}, \textbf{424}, no. 1, 563--583 (2015).

\bibitem{Habib-Minnaert}
H. Ammari, B. Fitzpatrick, D. Gontier, H. Lee, and H. Zhang,
\newblock{Minnaert resonances for acoustic waves in bubbly media.}

\newblock{http://www.sam.math.ethz.ch/sam\_reports/reports\_final/reports2016/2016-18\_fp.pdf}

\bibitem{Habib-bubbles}
H. Ammari and H. Zhang,
\newblock { Effective medium theory for acoustic waves
in bubbly fluids near Minnaert resonant frequency.} 
\newblock{SIAM J. Math. Anal.}, 49 (2017), 3252-3276.  


\bibitem{H-F-G-L-Z-1}
H. Ammari, B. Fitzpatrick, D. Gontier, H. Lee, and H. Zhang, 
\newblock{ Sub-wavelength focusing of acoustic waves in bubbly media.}
\newblock{Proceedings of the Royal Society A.}, 473 (2017), 20170469.

\bibitem{H-F-G-L-Z}
H. Ammari, B. Fitzpatrick, D. Gontier, H. Lee, and H. Zhang, 
\newblock{ A mathematical and numerical framework for bubble meta-screens.}
\newblock {SIAM J. Appl. Math.}, 77 (2017), 1827-1850.

\bibitem{ACCS}
H. Ammari, D.P. Challa, A.P. Choudhury, and M. Sini,
The point-interaction approximation for the fields generated by contrasted bubbles at arbitrary fixed frequencies, \emph{J. Differential Equations}, \textbf{267}, no. 4, 2104--2191 (2019).


\bibitem{CMS}
D.P. Challa, A. Mantile, and M. Sini,
Characterization of the equivalent acoustic scattering for a cluster of an extremely large number of small holes, arXiv:1711.05003 .

\bibitem{CK}
D. Colton, and R. Kress,
\newblock{\em Inverse acoustic and electromagnetic scattering theory},
\newblock Third edition. Applied Mathematical Sciences, 93. Springer, New York, 2013. xiv+405 pp.

\bibitem{CS}
M. Costabel, and E.P. Stephan, 
An improved boundary element {G}alerkin method for three-dimensional crack problems, \emph{Integral Equations Operator Theory}, \textbf{10}, no. 4, 467–-504 (1987).

\bibitem{HW}
G.C. Hsiao, and W.L. Wendland, 
\newblock{\em Boundary Integral Equations},
Springer, Berlin, 2008.

\bibitem{MPS}
A. Mantile, A. Posilicano, and M.Sini,
Self-adjoint elliptic operators with boundary conditions on not closed hypersurfaces, \emph{J. Differential Equations}, \textbf{261}, no. 1, 1--55 (2016). 

\bibitem{MM}
I. Mitrea, and M. Mitrea,
\newblock{\em Multi-layer potentials and boundary problems for higher-order elliptic systems in {L}ipschitz domains},
\newblock Lecture Notes in Mathematics, 2063. Springer, Heidelberg, 2013. x+424 pp.

\bibitem{Papanicoulaou-2}
 R. Caflisch, M. Miksis, G. Papanicolaou, and L. Ting, 
\newblock Wave propagation in bubbly liquids at finite volume fraction 
\newblock {\em J. Fluid Mec.} (1986), V-160, 1-14.


\bibitem{St}
O. Steinbach,
\newblock{\em Numerical approximation methods for elliptic boundary value problems},
\newblock Finite and boundary elements, Translated from the 2003 German original.
\newblock Springer, New York, 2008.

%\bibitem{Olver} F.W.J. Olver, 
%{\sl Asymptotics and special functions}, 
%Academic Press, New York, 1974.
\begin{comment}
\bibitem{A-G-H-H:AMS2005}
S.~Albeverio, F.~Gesztesy, R.~H{\o}egh-Krohn, and H.~Holden,
\newblock {\em Solvable models in quantum mechanics}.
\newblock AMS Chelsea Publishing, Providence, RI, second edition, 2005.
\newblock With an appendix by Pavel Exner.

\bibitem{Aless-Morassi2002-Inv}
G. Alessandrini, A. Morassi, and E. Rosset, Detecting cavities by electrostatic boundary measurements, \emph{ Inverse Problems}, \textbf{18} (2002) no. 5, 1333-1353.


\bibitem{A-A-C-K-S}
 A. Alsaedi; B. Ahmad; D. P. Challa; M. Kirane, and M. Sini,
 \newblock{A cluster of many small holes with negative imaginary surface impedances may generate a negative refraction index.}
 \emph{Math. Methods Appl. Sci.} 39 (2016), no. 13, 3607–3622.

%\bibitem{Ammari-Kang-2}
%H. Ammari and H. Kang, Reconstruction of small inhomogeneities from boundary measurements, Lecture Notes in Mathematics, 1846. \emph{Springer-Verlag, Berlin,} 2004. x+238 pp.


\bibitem{Habib-Minnaert}
H. Ammari, B. Fitzpatrick, D. Gontier, H. Lee, and H. Zhang,
\newblock{Minnaert resonances for acoustic waves in bubbly media.}
\newblock{Ann. Inst. H. Poincar\'e Anal. Non Lin\'eaire}, 
\newblock{https://doi.org/10.1016/j.anihpc.2018.03.007}. 
%\newblock{http://www.sam.math.ethz.ch/sam\_reports/reports\_final/reports2016/2016-18\_fp.pdf}



\bibitem{H-F-G-L-Z-1}
H. Ammari, B. Fitzpatrick, D. Gontier, H. Lee, and H. Zhang, 
\newblock{ Sub-wavelength focusing of acoustic waves in bubbly media.}
\newblock{Proceedings of the Royal Society A.}, 473 (2017), 20170469.

\bibitem{H-F-G-L-Z}
H. Ammari, B. Fitzpatrick, D. Gontier, H. Lee, and H. Zhang, 
\newblock{ A mathematical and numerical framework for bubble meta-screens.}
\newblock {SIAM J. Appl. Math.}, 77 (2017), 1827-1850.

%\bibitem{H-F-L-Y-Z}
%H. Ammari, B. Fitzpatrick, H. Lee, S. Yu, and H. Zhang, 
%\newblock{ Subwavelength phononic bandgap opening in bubbly media.}
%\newblock {J. Differential Equat.}, 263 (2017), 5610-5629.

\bibitem{Ammari-Kang-1}
H. Ammari and H. Kang, Polarization and moment tensors, With applications to inverse problems and effective medium theory, \emph{Applied Mathematical Sciences, Springer, New York}, \textbf{162} (2007).


\bibitem{Habib-bubbles}
H. Ammari and H. Zhang,
\newblock { Effective medium theory for acoustic waves
in bubbly fluids near Minnaert resonant frequency.} 
\newblock{SIAM J. Math. Anal.}, 49 (2017), 3252-3276.  


\bibitem{Papanicoulaou-1}
R. Caflisch, M. Miksis, G. Papanicolaou, and L. Ting, 
\newblock Effective equations for wave propagation in a bubbly liquid.
\newblock {\em J. Fluid Mec.} (1985), V-153, 259-273. 

\bibitem{Papanicoulaou-2}
 R. Caflisch, M. Miksis, G. Papanicolaou, and L. Ting, 
\newblock Wave propagation in bubbly liquids at finite volume fraction 
\newblock {\em J. Fluid Mec.} (1986), V-160, 1-14. 


%\bibitem{Challa-Choudhury-Sini}
%D.P. Challa, A.P. Choudhury and M. Sini, Mathematical imaging using electric or magnetic nanoparticles as contrast agents, arXiv:1705.01498 .

\bibitem{Challa-Sini-1}
D.P. Challa and M. Sini, On the justification of the Foldy-Lax approximation for the acoustic scattering by small rigid bodies of arbitrary shapes, \emph{Multiscale Model. Simul.}, \textbf{12} (2014), no. 1, 55-108.

\bibitem{Challa-Sini-2}
D.P. Challa and M. Sini, Multiscale analysis of the acoustic scattering by many scatterers of impedance type, \emph{Z. Angew. Math. Phys.}, \textbf{67} (2016), no. 3, Art. 58.

\bibitem{CMS-2017}
D.P. Challa, A. Mantile and M. Sini, Characterization of the equivalent acoustic scattering for a cluster of an extremely large number of small holes.
 \emph{arXiv:1711.05003}.

\bibitem{Foldy}
 L. L. Foldy. 
\newblock{The multiple scattering of waves. I. General theory of isotropic scattering by randomly distributed scatterers.}
\newblock{Phys. Rev. (2), 67:107-119, 1945.}
 
 
\bibitem{L-S:2016}
A. Lamacz and B. Schweizer,
\newblock{ A negative index meta-material for Maxwell's equations.}
\newblock{\em SIAM J. Math. Anal.} 48, no.6, 4155-4174 (2016)

\bibitem{Lax} 
M. Lax. 
\newblock{Multiple scattering of waves.}
\newblock{Rev. Modern Physics, 23:287-310, 1951.}


\bibitem{Martin:2006}
P.A. Martin.
\newblock { Multiple scattering}, volume 107 of {\em Encyclopedia of
  Mathematics and its Applications}.
\newblock Cambridge University Press, Cambridge, 2006.
\newblock Interaction of time-harmonic waves with $N$ obstacles.


\bibitem{Pan:2000}
M. Panfilov, 
\newblock{ Macroscale models of flow through highly heterogeneous porous media,}
\newblock {\em Kluwer Academic,} Dordrecht, Boston, London, 2000. 

\bibitem{Papanicoulaou-3}
G C. Papanicolaou, 
\newblock{Diffusion in random media,}
\newblock{Surveys in Applied Mathematics, volume
1, Edited by J P. Keller, D W. McLaughlin and G C. Papanicolaou, Plenum Pre
ss, NewYork, 1995.}

\end{comment}

%\bibitem{W1}
%\newblock $https://en.wikipedia.org/wiki/Relative\_permittivity$

\end{thebibliography}
\end{document}